\documentclass[12pt]{amsart}

\usepackage{multicol}

\usepackage{amssymb}
\usepackage{fancyhdr}
\usepackage{amsmath}
\usepackage{amsfonts}
\usepackage{mathrsfs}
\usepackage[all]{xy} 
\usepackage{color}

\newcommand{\disk}{\ensuremath{\mathbb{D}} } 
\newcommand{\sphere}{\overline{\Bbb{C}}} 
\newcommand{\riem}{\Sigma}  
\renewcommand{\Bbb}[1]{\ensuremath{\mathbb{#1}}}

\newcommand{\qs}{\operatorname{QS}}

\newcommand{\h}{\mathbb{H}}
\newcommand{\R}{\mathbb{R}}

\theoremstyle{plain}
        \newtheorem{theorem}{Theorem}[section]
        \newtheorem{lemma}[theorem]{Lemma}
        
        \newtheorem{corollary}[theorem]{Corollary}

\theoremstyle{definition}
        \newtheorem{definition}[theorem]{Definition}

\theoremstyle{remark}
    \newtheorem{remark}[theorem]{Remark}

\numberwithin{equation}{section} 
\numberwithin{figure}{section} 

\usepackage{fullpage}
\title
[Analysis on quasidisks]{Analysis on quasidisks \\ a unified approach through transmission and jump problems }

\dedicatory{\emph{Dedicated to our friend Ian Graham}}

\author[E. Schippers]{Eric Schippers}
\author[W. Staubach]{Wolfgang Staubach}

\address{\newline
       Eric Schippers \newline
       Machray Hall, Dept. of Mathematics,
   University of Manitoba, \newline Winnipeg, MB
   Canada R3T 2N2}
       \email{eric.schippers@umanitoba.ca}

\address{\newline
       Wolfgang Staubach \newline
       Department of  Mathematics, Uppsala University, \newline
       S-751 06 Uppsala, Sweden}
       \email{wulf@math.uu.se}


\subjclass[2010]{Primary: {}, Secondary: {}}

\begin{document}

\vspace*{-1cm}
\maketitle
\vspace*{-0.9cm}
 \begin{abstract} 
  We give an exposition of results from a crossroad between geometric function theory, harmonic analysis, boundary value problems and approximation theory, which characterize quasicircles. We  will specifically  expose the interplay between the jump decomposition, singular integral operators and approximation by Faber series. Our unified point of view is made possible by the the concept of transmission.

 \end{abstract} 
\vspace*{-0.3cm}
\tableofcontents

\begin{section}{Introduction}

A quasiconformal map in the plane is a homeomorphism between planar domains which maps small circles to small ellipses of bounded eccentricity. A quasicircle is by definition the image of the circle $\mathbb{S}^1$ under a quasiconformal map and a quasidisk is the interior of a quasicircle. In geometric function theory quasicircles play a fundamental role in the description of the universal Teichm\"uller space. They also play an important role in complex dynamical systems. The reader is referred to the book by F. Gehring and K. Hag \cite{Gehring} for a nice introduction to various ramifications of this topic. \\  

It is a familiar fact in the field that quasicircles have an unusually large number of   characterizations which are not obviously equivalent, and indeed are qualitatively quite different.  See e.g. \cite[Chapters 8,9]{Gehring} for some of the classical and also some less well-known ones.  It is somewhat astonishing that these continue to be found.   In this paper, we will focus on the relatively recent ones, due to A. \c{C}avu\c{s} \cite{Cavus}, Y. Y. Napalkov and R. S. Yulmukhametov \cite{Nap_Yulm}, Y. Shen \cite{ShenFaber}, and the authors \cite{Schippers_Staubach_JMAA,Schippers_Staubach_CAOT}.  Indeed, our purpose here is to highlight a characterization based on an interplay between geometric function theory, harmonic analysis, boundary value problems and approximation theory. This point of view was investigated by the authors in a series of papers, and in these works, it emerged that the key to a unified approach is the method of transmission of harmonic functions (or forms).\\

The goal of this paper is to give an essentially self-contained and unified exposition of this circle of ideas and the method of transmission, not least because of its potential applications outside geometric function theory. In doing so we have also refined and improved many of our theorems in previous papers.\\

To define the notion of transmission, let $\Gamma$ be a Jordan curve separating the Riemann sphere $\sphere$ into two components $\Omega_1$ and $\Omega_2$.  Given a harmonic function $h$ on $\Omega_1$ which extends continuously to $\Gamma$, there is a harmonic function on $\Omega_2$ with the same continuous extension on $\Gamma$.  We call the new function the transmission of $h$.  
 We generalize the concept of transmission to Dirichlet bounded harmonic functions.  For such harmonic functions, the transmission exists and is bounded with respect to the Dirichlet semi-norm if and only if the curve $\Gamma$ is a quasicircle.\\

 
Returning to the problem of characterization of quasicircles, it came to light that in the setting of Dirichlet bounded harmonic functions, a number of perfect equivalences arise, which make a unified treatment of a number of topics possible. To begin with, given a Jordan curve $\Gamma$ as above, the following three statements are equivalent.
 \begin{enumerate}
     \item $\Gamma$ is a quasicircle.\\
     \item There is a bounded transmission from the Dirichlet space of harmonic functions on  $\Omega_1$ to the Dirichlet space of harmonic functions on $\Omega_2$, which agrees with transmission of continuous functions. \\
     \item The linear operator taking the boundary values of a Dirichlet bounded harmonic function to its Plemelj-Sokhotski jump decomposition is a bounded isomorphism.  
 \end{enumerate}
 These results are due to the authors \cite{Schippers_Staubach_JMAA,Schippers_Staubach_CAOT}.\\

 The first three equivalent statements also are closely related to approximability by Faber series, the Faber and Grunsky operators, and the Schiffer operator.  We thus have the following further equivalent statements.  Attributions in brackets refers to the first proofs of the equivalence with (1), unless clarified below.\\      
 \begin{enumerate}  \setcounter{enumi}{3}
     \item  The Faber operator corresponding to $\Omega_2$ is an isomorphism (authors \cite{Schippers_Staubach_CAOT}).\\   
     \item  The sequential Faber operator is an isomorphism (\c{C}avu\c{s}  \cite{Cavus}, Shen \cite{ShenFaber}).\\
     \item  Every element of the holomorphic Dirichlet space of $\Omega_2$ is uniquely approximable by a Faber series (\c{C}avu\c{s} \cite{Cavus}, Shen \cite{ShenFaber}).\\  
     \item  The Schiffer operator is an isomorphism (Napalkov and Yulmukhametov \cite{Nap_Yulm}).  \\ 
\end{enumerate}      
 The implications (1) $\Rightarrow$ (5) and (1) $\Rightarrow$ (6) are due to \c{C}avu\c{s}, and later independently by Shen, while the reverse implications are due to Shen.  
 For the special case of rectifiable Jordan curves, the equivalence of (1) and (4) is due to H. Y. Wei, M. L. Wang, and Y. Hu  \cite{WeiWangHu}.    
 In this paper, we give  proofs of the equivalence of (4) - (7) with (1) which rely on the transmission result (2).  The proofs given here that (4)-(7) imply (1) are new.

 Finally, all of these results are closely connected to the classical result that $\Gamma$ is a quasicircle if and only if
\begin{enumerate} \setcounter{enumi}{7}
     \item  The norm of the Grunsky operator is strictly less than one.
 \end{enumerate}
 The implication (1) $\Rightarrow$ (8) is due to R. K\"uhnau \cite{Kuehnau_Verzerrungsatze} and (8) $\Rightarrow$ (1) is due to C. Pommerenke \cite{Pommerenkebook}.  
 In the literature, all proofs of the implication (k) $\Rightarrow$ (1) for $k=4,\ldots,7$ (including those due to the authors) rely on the result (8) $\Rightarrow$ (1). However the proofs given in this paper do not.\\  
 
An important issue in connection to transmission is that  some notion of boundary values is necessary in order to define the transmission in a sensible way. To this end we also include an exposition of a conformally invariant notion of non-tangential boundary value, which we call conformally non-tangential (CNT for short).  This was developed by the authors for Jordan curves in Riemann surfaces \cite{Schippers_Staubach_transmission}.  The existence of such boundary values for the Dirichlet space of a simply connected domain is an automatic consequence of a well-known result of A. Beurling.  On the other hand, it is not true in general that the boundary values of a harmonic function in one connected component of the complement of $\Gamma$ are boundary values of a harmonic function in the other component.  Even potential-theoretically negligible sets are not obviously the same: for example, sets of harmonic measure zero with respect to one side are not necessarily harmonic measure zero with respect to the other, even for quasicircles. At any rate, we give a general framework for the application of the CNT boundary values to sewing and transmission.  Aside from the bounded transmission theorem mentioned above, the most important of these results are:
\begin{enumerate}
    \item [(i)] for quasicircles, the potential-theoretically negligible sets on the boundary of $\Omega_1$ are also negligible for $\Omega_2$;\\
    \item [(ii)] the operator (what we call the bounce operator) taking a Dirichlet-bounded harmonic function on a doubly-connected region in $\Omega_1$, one of whose boundaries is $\Gamma$, to the harmonic function on $\Omega_1$ with the same boundary values, is bounded for any Jordan curve;\\
    \item [(iii)] limiting integrals taken over level curves of Green's function are the same for any two Dirichlet bounded harmonic function in a collar near $\Gamma$ which have the same CNT boundary values (the anchor lemma). 
\end{enumerate}
The precise statements are given in Theorems \ref{th:null_both_sides_quasicircle}, \ref{th:bounce_bounded}, and Theorem \ref{le:anchor_lemma} respectively. \\ 
 
To conclude, we strive in this paper to show the clarifying power of the transmission theorem for understanding approximation by Faber series, the Grunsky operator, the Plemelj-Sokhotski jump theorem, and Schiffer operators.
 The results should have many applications in the investigation of the behaviour of function spaces, boundary value problems, and related operators under sewing.  The results here are 
 also the basis for a scattering theory of harmonic functions and one-forms for general Riemann surfaces \cite{Schippers_Staubach_monograph}.\\



The paper is organized as follows.  In Section \ref{se:boundary_values}, we state necessary definitions and results regarding conformally non-tangential boundary values of Dirichlet bounded functions.  After preliminaries on the Dirichlet and Bergman space, and quasisymmetric mappings in Sections \ref{se:Dirichlet_Bergman_preliminaries} and \ref{se:quasi_preliminaries}, we define certain potential-theoretically negligible sets on a Jordan curve $\Gamma$ with respect to the enclosed domain in Section \ref{se:null_sets}, which we call null sets, and derive their basic properties.   A particularly crucial fact is that, in the case that the Jordan curve is a quasicircle, sets that are null with respect to one of the regions enclosed by $\Gamma$ are also null with respect to the other.  In fact, for quasicircles not containing $\infty$,  null sets in $\Gamma$ are precisely Borel sets of capacity zero.   After reviewing some basic results on boundary values of the Dirichlet space of the disk in Section \ref{se:boundary_values_disk}, we give the definition of CNT boundary values in Section \ref{se:CNT_boundaryvalues} and basic properties.\\  
  
  Section \ref{se:transmission} contains the first of the main results, namely (1) $\Leftrightarrow$ (2): a bounded transmission exists on Dirichlet space if and only if $\Gamma$ is a quasicircle.
  Section \ref{se:Nag_Sullivan} reviews some known theorems which characterize quasisymmetries in terms of their action on the homogeneous Sobolev space $H^{1/2}$, and a reformulation in terms of CNT boundary values up to null sets. This refinement is necessary because sets of harmonic measure zero on a quasicircle with respect to one side of a curve - which are the images of sets of Lebesgue measure zero on the circle under a conformal map -  need not be of harmonic measure zero with respect to the other side of the curve. Thus, null sets are necessary.  Section \ref{se:transmission_sub} contains the transmission result.  In Section \ref{se:bounce_anchor_density}, we establish several useful results regarding boundary values and integrals. We prove that the so-called bounce operator described in the introduction is bounded.  We also prove the ``anchor lemma'', which shows that certain limiting integrals taken over curves approaching the non-rectifiable Jordan curve depend only on  the CNT boundary values.  Finally, we give a few useful dense subsets of Dirichlet spaces on simply- and multiply-connected domains.  These ultimately rely on density of polynomials.\\  
  
  Section \ref{se:Schiffer_Cauchy} contains the main results on Plemelj-Sokhotski jump isomorphism and Schiffer isomorphism, that is $(1) \Leftrightarrow (3) \Leftrightarrow (7).$  Section \ref{se:Schiffer} defines the Schiffer operator and proves basic analytic results, M\"obius invariance, and an identity of Schiffer.  Section \ref{se:jump_isomorphisms} defines a Cauchy integral operator adapted to non-rectifiable curves using limits of integrals over curves approaching the boundary. We show that for quasicircles the value of this operator is the same for curves approaching $\Gamma$ over either side, in a certain sense involving transmission.  We also prove basic identities relating the Cauchy integral operator to the Schiffer operators, and the M\"obius invariance of the operator.  Section \ref{se:jump_isomorphisms} contains the main results which show that the Plemelj-Sokhotski jump decomposition exists when $\Gamma$ is a  quasicircle, and in a certain sense this decomposition is an isomorphism if and only if $\Gamma$ is a quasicircle.  We also give a new proof of Napalkov and Yulmukhametov's result that the Schiffer operator is an isomorphism if and only if $\Gamma$ is a quasicircle.\\
  
  In Section \ref{se:Faber_Grunsky} we prove that the Faber operator is an isomorphism 
  if and only if $\Gamma$ is a quasicircle, as well as the existence and uniqueness of Faber series; that is,  (1) $\Leftrightarrow$ (4) $\Leftrightarrow$ (5) $\Leftrightarrow$ (6).    We also give a brief review of the equivalence with strict Grunsky inequalities.  

  Finally, Section \ref{se:notes} contains notes on the literature, as well as some fine points which could not be put in the main text without interrupting the flow of the paper.  Although the notes are fairly extensive for a paper of this size, we make no claims to completeness, and merely indicate the tip of the literary iceberg. 
  

\end{section}
\begin{section}{Function spaces and boundary values} \label{se:boundary_values}
\begin{subsection}{Dirichlet and Bergman spaces}  \label{se:Dirichlet_Bergman_preliminaries}


 We denote the complex plane by $\mathbb{C}$ and the Riemann sphere by $\sphere$.  We define $\disk^+ = \{ z\in \mathbb{C} \,:\, |z|< 1 \}$ and also $\disk^- = \{ z \in \mathbb{C}  \,: \, |z | > 1 \} \cup \{ \infty \}$.  The circle $\mathrm{bd}(\disk) = \{ z \in \mathbb{C} \,:\, |z|=1 \}$ is denoted $\mathbb{S}^1$.  In this paper, a conformal map is always assumed to be one-to-one (not just locally one-to-one). That is, a conformal map is a biholomorphism onto its image.  
 
 The Riemann sphere $\sphere$ is endowed with the standard complex structure given by the charts 
 \begin{align*} 
  \psi_0:\mathbb{C} &  \rightarrow \mathbb{C} \\
  \psi_0(z) & = z
 \end{align*}
 \begin{align*} 
  \psi_\infty:\sphere \backslash \{0\} &  \rightarrow \mathbb{C} \\
  \psi_\infty(z) & = 1/z \ \ z \neq \infty \\
  \psi_\infty(\infty) & = 0,
 \end{align*} 
 and holomorphicity or harmonicity is defined with respect to these charts.
 That is,
 let $\Omega$ be an open connected set in $\sphere$.  A function $h$ is holomorphic on $\Omega$ if (1) it is holomorphic on $\sphere \backslash \{ \infty \}$ and (2) if $\infty \in \Omega$ then $g(z)=f(1/z)$ is holomorphic in a neighbourhood of $0$. Anti-holomorphic and harmonic functions on $\Omega$ are defined similarly.\\ 
 
 We will also consider smooth one-forms on subsets of $\sphere$, where these are defined in the usual way in terms of the Riemann surface structure of $\sphere$.  Any one-form $\alpha$ is given in local coordinates by $h_1(z)\, dz + h_2(z)\, d\bar{z}$ for smooth functions $h_1(z)$ and $h_2(z)$.   A one-form $\alpha$ on $\Omega$ is said to be holomorphic if it can be expressed in local coordinates $z$ as $h(z) \,dz$ where $h(z)$ is holomorphic.  That is, $\alpha = a(z)\,dz$ on $\Omega \backslash \{\infty \}$, and if $\infty \in \Omega$, then $b(w) = -a(1/w)/w^2$ is holomorphic in an open set containing $0$ (so that in a chart at $\infty$, we may write $\alpha = b(w)\, dw$).  A one-form is anti-holomorphic if it is the complex conjugate of a holomorphic one-form.  
 
 We also define the $\ast$-operator as follows.  If $\alpha = h_1(z)\, dz + h_2(z) \,d\bar{z}$ in local coordinates we define 
 \[  \ast \alpha = \ast (h_1(z)\, dz + h_2(z) \,d\bar{z} ) = -i h_1 \,dz + i h_2 \,d\bar z. \]
 It is easily checked that this is well-defined with respect to the change of coordinates $z = \psi_0 \circ \psi^{-1}_\infty(w) = 1/w$.

  Define
  \begin{equation}\label{defn:length of a form}
   \| \alpha \|_{\Omega}^2 =  \frac{1}{2 \pi} \iint_\Omega \alpha \wedge \ast \overline{ \alpha}   
  \end{equation}
  which might of course diverge.  Since any smooth one-form $\alpha$ on $\Omega$ can be written (uniquely) in $z$ coordinates as 
  \begin{equation} \label{eq:coordinates_alpha} \alpha = h_1(z)\, dz +h_2(z)\, d\bar{z} 
  \end{equation}
  for smooth functions $h_1$ and $h_2$, then if  in (\ref{eq:coordinates_alpha}) $z$ is the parameter in $\mathbb{C}$ (that is, in $\psi_0$ coordinates), then \eqref{defn:length of a form} can be written as 
  \begin{equation} \label{eq:removable_point_temp1}
    \| \alpha \|^2_\Omega = \frac{1}{\pi} \iint_{\Omega \backslash \{\infty \}} 
    \left( |h_1(z)|^2 + |h_2(z)|^2 \right)\,dA,
  \end{equation}
  where $dA = (d\bar{z} \wedge dz)/2i$ is the Euclidean area element in $\mathbb{C}$.
  This is justified as follows:  when $\infty \in \Omega$, if $\| \alpha\|^2_\Omega <\infty$ then it is easily verified that there is an $R$ such that 
  \[  \iint_{|z| >R }  \left(|h_1(z)|^2 + |h_2(z)|^2 \right) \,dA < \infty. \] 
  Thus the point at $\infty$ can be removed from the domain of integration without changing the convergence properties or value of the integral.
  
\begin{definition}\label{defn:harmonic one-form}
A smooth one-form $\alpha$ is said to be harmonic if $d \alpha = 0$ and $d \ast \alpha =0$; equivalently, for any point $p \in \Omega$, $\alpha = dh$ for some harmonic function $h$ on some open neighbourhood of $p$.  Note that if $\infty \in \Omega$, this restricts the behaviour of $\alpha$ at $\infty$ since $h(1/z)$ must be harmonic at $0$.
\end{definition}  
   We then define the space of $L^2$ harmonic one-forms $\mathcal{A}_{\text{harm}}(\Omega)$ to consist of those harmonic one-forms $\alpha$ on $\Omega$ such that $\| \alpha \|_{\Omega} < \infty$.  This is a Hilbert space with inner product
  \begin{equation}\label{defn:scalar product of forms}
   (\alpha,\beta) = \frac{1}{2 \pi} \iint_{\Omega} \alpha \wedge \ast \overline{\beta},   
  \end{equation} 
 which is also consistent with \eqref{defn:length of a form}. 
  The Bergman space of one-forms is 
  \[ \mathcal{A}(\Omega) = \{ \alpha \in \mathcal{A}_{\text{harm}}(\Omega) \,:\, \alpha \text{ is holomorphic}\},  \]
  and for $\alpha = h_1(z)\, dz,$ ${{\beta= h_2(z)\, d{z}}} \in \mathcal{A}(\Omega)$
  we have 
  \[ (\alpha,\beta) = \frac{1}{\pi} \iint_{\Omega \backslash \{ \infty \}} 
   h_1(z)\, \overline{h_2(z)} \,dA. \]
  The anti-holomorphic Bergman space $\overline{\mathcal{A}(\Omega)}$ consists of complex conjugates of elements of $\mathcal{A}(\Omega)$.   
  
  Observe that $\mathcal{A}(\Omega)$ and $\overline{\mathcal{A}(\Omega)}$ are orthogonal with respect to the inner product.   We then obtain the decomposition 
  \[   \mathcal{A}_{\text{harm}}(\Omega) = \mathcal{A}(\Omega) \oplus \overline{\mathcal{A}(\Omega)}     \]
  which induce the projection operators
  \begin{align}\label{defn:projections in bergman}
    \mathbf{P}(\Omega) :\mathcal{A}_{\text{harm}}(\Omega) & \rightarrow \mathcal{A}(\Omega)\nonumber \\
    \overline{\mathbf{P}(\Omega)} :\mathcal{A}_{\text{harm}}(\Omega) & \rightarrow \overline{\mathcal{A}(\Omega)} 
  \end{align}
  \begin{definition}\label{defn:Dirichlet spaces}
  For an open connected set $\Omega$ and a smooth function $h:\Omega \rightarrow \mathbb{C}$ we define the Dirichlet energy of $h$ by
 \begin{equation}\label{defn:dirichlet norm of form}
  D_\Omega(h) = \| dh \|^2_{\Omega}.   
 \end{equation}  
  The harmonic Dirichlet space  $\mathcal{D}_{\text{harm}}(\Omega)$ consists of those harmonic functions $h$ on $\Omega$ such that $D_\Omega(h) < \infty$.  
  If $z$ is the coordinate in $\mathbb{C}$ then \eqref{defn:dirichlet norm of form} can be written as
  \begin{equation}
   D_\Omega(h) = \frac{1}{\pi} \iint_{\Omega \backslash \{ \infty \}} 
   \left( \left| \frac{\partial h}{\partial z}\right|^2 +  \left| \frac{\partial 
   h}{\partial \bar{z}} \right|^2 \right) \,dA.    
  \end{equation} 
  The holomorphic Dirichlet space $\mathcal{D}(\Omega)$ is the set of holomorphic functions in $\mathcal{D}_{\text{harm}}(\Omega)$, and the anti-holomorphic Dirichlet space $\overline{\mathcal{D}(\Omega)}$ is given by the set of complex conjugates of elements of $\mathcal{D}(\Omega)$. 
  The Dirichlet energy on $\mathcal{D}(\Omega)$ restricts to 
  \[   D_\Omega(h) = \frac{1}{\pi}  \iint_{\Omega \backslash \{\infty\}} |h'(z)|^2\, dA   \]
  and similarly for $\overline{\mathcal{D}(\Omega)}$.  
  Observe that $\mathcal{D}_{\text{harm}}(\Omega)$ does not decompose into a sum of elements of $\mathcal{D}(\Omega)$ and $\overline{\mathcal{D}(\Omega)}$ unless $\Omega$ is simply connected (and even in that case, the decomposition is not unique because constants belong to both spaces). 
  \end{definition}

  The Dirichlet energy is not a norm, since $D(c)=0$ for constants $c$.  It becomes a norm if we restrict to normalized functions $h(p)=0$ for some $p \in \Omega$. 
  When such a normalization is imposed, we use the notations $\mathcal{D}_p(\Omega)$, $\mathcal{D}_{\text{harm}}(\Omega)_p$, and use $\Vert \cdot \Vert_{\mathcal{D}_p(\Omega)}$ and so on, for the corresponding norms. 
  
  We use the following notation for projection.   Fix a simply-connected domain $\Omega$ in $\sphere$.  Fix $p \in \Omega$. Define the decompositions 
 \begin{align*}
    \mathcal{D}_{\text{harm}}(\Omega) & = \overline{\mathcal{D}(\Omega)}  \oplus \mathcal{D}_p(\Omega) \\ 
    & = \overline{\mathcal{D}_{p}(\Omega)} \oplus \mathcal{D}(\Omega).
 \end{align*}
  {
   We then have projections 
   \begin{align} \label{eq:holomorphic_normalization_projection}
    \overline{\mathbf{P}^{h}_p(\Omega)}: \mathcal{D}_{\text{harm}}(\Omega) & \rightarrow \overline{\mathcal{D}(\Omega)} \nonumber \\
    {\mathbf{P}^{h}_p(\Omega)}: \mathcal{D}_{\text{harm}}(\Omega) & \rightarrow {\mathcal{D}_p(\Omega)}
   \end{align}
    induced by the first decomposition, and projections 
    \begin{align} \label{eq:antiholomorphic_normalization_projection}
     \overline{\mathbf{P}^{a}_p(\Omega)} : \mathcal{D}_{\text{harm}}(\Omega) 
     & \rightarrow \overline{\mathcal{D}_p(\Omega)} \nonumber \\
     {\mathbf{P}^{a}_p(\Omega)} : \mathcal{D}_{\text{harm}}(\Omega) & \rightarrow \mathcal{D}(\Omega)
    \end{align}
   induced by the second. Note that all four projections depend on the location of $p$.  The superscripts ``$h$'' and ``$a$'' stand for ``holomorphic'' and ``anti-holomorphic'' normalizations. }  
 \\
 
 Finally, for a domain $G$ in the plane, with boundary $\Gamma$, we also define the Sobolev spaces $H^1(G)$ and $H^{1/2}(\Gamma)$:
 \begin{definition}
 $H^1(G)$ consists of functions in $L^2(G)$ such that 
  \begin{equation}\label{sobolev norm}
      \Vert h\Vert_{H^1(G)}:=\left(D_{G}(h)+\Vert h\Vert^2_{L^2(G)}\right) ^{1/2}<\infty.
  \end{equation}
  Moreover, if $\Gamma$ is regular enough then one can also take the restriction (trace) of an $H^1(G)$-function to $\Gamma$, which yields a function $h|_{\Gamma}\in H^{1/2}(\Gamma)$ where $H^{1/2}(\Gamma)$ is the space of functions in $L^2(\Gamma)$ for which
  
    \begin{equation}\label{sobolev half norm}
      \Vert f\Vert_{H^{1/2}(\Gamma)}:=\left( \int_{\Gamma} \int_{\Gamma} \frac{|f(z)-f(\zeta)|^2}{|z-\zeta|^2}\,|dz|\, |d \zeta|+\Vert f\Vert^2_{L^2(\Gamma)}\right)^{1/2}<\infty,
  \end{equation}
 \end{definition}

see Chapter 4 in \cite{Taylor} for all the details regarding Sobolev spaces. 
  \end{subsection}

\begin{subsection}{Quasisymmetries and quasiconformal maps} \label{se:quasi_preliminaries} \ 
 
 In this section we review definitions and results about quasisymmetries and quasiconformal maps.

 \begin{definition} \label{de:quasiconformal}
  Let $A$ and $B$ be open connected subsets of the complex plane.  An orientation-preserving homeomorphism $\Phi:A \rightarrow B$ is a $k$-{\it quasiconformal mapping} if\\ 
  \begin{enumerate}
      \item for every rectangle $[a,b] \times [c,d] \subset A$, $\Phi(x,\cdot)$ is absolutely continuous on $[c,d]$ for almost every $x \in [a,b]$;\\
      \item for every rectangle $[a,b] \times [c,d] \subset A$, $\Phi(\cdot,y)$ is absolutely continuous on $[a,b]$ for almost every $y \in [c,d]$;\\
      \item there is a $k \in (0,1)$ such that $| \Phi_{\bar{z}} | \leq k | \Phi_{z} |$ almost everywhere in $A$. 
  \end{enumerate}
  \vspace{0.3cm}
  We say that a map is quasiconformal if it is $k$-quasiconformal for some $k \in (0,1)$.
 \end{definition}

 Define 
 \begin{align*} \iota:\mathbb{C} \backslash \{0\} & \rightarrow \mathbb{C} \backslash \{0\}. \\
  z & \mapsto 1/z.
 \end{align*}
 Let $A$ and $B$ be open connected subsets of $\sphere$.  We say that a homeomorphism $\Phi:A \rightarrow B$ is a $k$-quasiconformal mapping if 
 \[  \Phi, \ \ \ \iota \circ \Phi, \ \ \ \Phi \circ \iota, \ \ \ \text{and} \ \ \ \iota \circ \Phi \circ \iota      \]
 are all $k$-quasiconformal on their maximal domains of definition; as above, we say that $\Phi$ is quasiconformal if it is $k$-quasiconformal for some $k$.  If $A,B \subsetneq \sphere$ then $\Phi$ is quasiconformal if,  given M\"obius transformations $S$ and $T$ such that $S(A) \subset \mathbb{C}$ and $T(B) \subset \mathbb{C}$, $T \circ \Phi \circ S^{-1}$ is quasiconformal from $S(A)$ to $T(B)$.

 Similarly, for open connected sets $A,B \subset \sphere$ we say that a map $f:A \rightarrow B$ is conformal if
 \[  f, \ \ \ \iota \circ f, \ \ \ f \circ \iota, \ \ \ \text{and} \ \ \ \iota \circ f \circ \iota      \]
 are all conformal on their maximal domains of definition.  Conformal maps are $0$-quasiconformal onto their image, and it can be shown that $0$-quasiconformal maps are conformal  (see e.g. \cite{Ahlfors_quasiconformal}).  Furthermore, if $\Phi:A \rightarrow B$ is quasiconformal and $g:A' \rightarrow A$ is conformal then $\Phi \circ g:A' \rightarrow B$ is quasiconformal, and if $f:B \rightarrow B'$ is conformal, then $f \circ \Phi$ is quasiconformal.  If $C$ is any open connected subset of $A$, then the restriction of $\Phi$ to $C$ is a quasiconformal map onto $\Phi(C)$. 
\begin{remark} \label{re:extension_to_infinity}
 Any quasiconformal map $\Phi:\mathbb{C} \rightarrow \mathbb{C}$ extends to a quasiconformal map from $\sphere$ to $\sphere$, which takes $\infty$ to $\infty$ \cite[Theorem I.8.1]{Lehto_Virtanen}.
\end{remark} 
\begin{definition}\label{de:quasisymmetriccircle}
 An orientation-preserving homeomorphism $h$ of $\mathbb{S}^1$ is called a \emph{quasisymmetric mapping}, iff there is a constant $k>0$, such that for every $\alpha$, and every $\beta$ not equal to a multiple of $2\pi$, the inequality
 \[  \frac{1}{k} \leq \left| \frac{h(e^{i(\alpha+\beta)})-h(e^{i\alpha})}{h(e^{i\alpha})-h(e^{i(\alpha-\beta)})} \right|
    \leq k \]
 holds. 
\end{definition}
Let $\qs(\mathbb{S}^1)$ denote the set of quasisymmetric maps from $\mathbb{S}^1$ to $\mathbb{S}^1$.  $\qs(\mathbb{S}^1)$ consists precisely of boundary values of quasiconformal maps, as the following two theorems show. 
\begin{theorem} \label{th:qs_are_qc_boundary}
 Let $\Phi: \mathbb{\disk}^+ \rightarrow \disk^+$ be a quasiconformal map. Then $\Phi$ has a continuous extension to $\mathbb{S}^1 \cup \disk^+$, and the restriction of this extension to $\mathbb{S}^1$ is a quasisymmetry.  Conversely, if $\phi:\mathbb{S}^1 \rightarrow \mathbb{S}^1$ is a quasisymmetry, then $\phi$ is the restriction to $\mathbb{S}^1$ of the continuous extension of a quasiconformal map $\Phi: \disk^+ \rightarrow \disk^+$.  In the above, one may replace $\disk^+$ everywhere by $\disk^-$ and the result still holds.
\end{theorem}
\begin{proof}
 { By reduction of the problem to the upper-half plane using the conformal equivalence of the unit disk and the former, this is just Ahlfors-Beurling's result in \cite{Ahlfors-Beurling}.}
\end{proof}

By a Jordan curve $\Gamma$ in $\sphere$, we mean the image of $\mathbb{S}^1$ under a continuous map into $\sphere$ which is a homeomorphism onto its image.  Equivalently, it is the image of a Jordan curve in the plane under a M\"obius transformation. 
\begin{definition}
 A Jordan curve $\Gamma$ in $\sphere$ is a {\it quasicircle} if and only if it is the image of $\mathbb{S}^1$ under a quasiconformal map $\Phi:\sphere \rightarrow \sphere$.  We say that a Jordan domain is a quasidisk if its boundary is a quasicircle.  
\end{definition}

Quasidisks have the following important property  \cite[Corollary 2.1.5]{Gehring}.
\begin{theorem} 
 Let $\Omega$ be a quasidisk.  If $f:\disk^{\pm} \rightarrow \Omega$ is a biholomorphism, then $f$ extends to a quasiconformal map of $\sphere$.  
\end{theorem}


One of the main tools in this paper is the conformal welding theorem.
\begin{theorem}[Conformal welding theorem] \label{th:conformal_welding}
 For any quasisymmetry $\phi:\mathbb{S}^1 \rightarrow \mathbb{S}^1$, there are conformal maps $f:\disk^+ \rightarrow \mathbb{C}$ and $g:\disk^- \rightarrow \sphere$, with the following properties.
 \begin{enumerate}
     \item $f$ and $g$ are quasiconformally extendible to $\sphere$ $($so that, in particular, $\Omega^+ = f(\disk^+)$ and $\Omega^-=g(\disk^-)$ are quasidisks$)$;
     \item $\mathrm{bd} (f(\disk^+)) = \mathrm{bd} (g(\disk^-))$, where $\mathrm{bd}$ denotes the boundary; and
     \item $\phi = (\left. g \circ f^{-1} \right)|_{\mathbb{S}^1}$.
 \end{enumerate}
 If we specify the normalization $f(0)=0$, $g(\infty)=\infty$, and $g'(\infty)=1$, then $f$ and $g$ are uniquely determined.  
\end{theorem}

The normalization above can be replaced with any three normalizations in the interior of the domains of $f$ or $g$ if desired.

\end{subsection}
\begin{subsection}{Null sets} \label{se:null_sets}  In this section, we define null sets, which are potential-theoretically negligible sets on the boundary of a Jordan domain.  That is, in specifying a  harmonic function of bounded Dirichlet energy on a Jordan domain by its boundary values, changes to (or non-existence of) the boundary values on null sets have no effect.  We will see that in the special case that the Jordan domain is bounded by a quasicircle, null sets are those sets of logarithmic capacity zero.   

We first recall the definition of logarithmic capacity  \cite{Ahlfors_conformalinvariants,Ransford_book}; we follow \cite{Ransford_book}. 
\begin{definition}\label{defn:logarithmic capacity}
Let $\mu$ be a finite Borel measure in $\mathbb{C}$ with compact support.  The potential of $\mu$ is the function 
 \[  p_\mu(z) = \iint_{\mathbb{C}} \log{|z-w|} d\mu(w).   \]
 The energy of $\mu$ is then defined to be 
 \[  I(\mu) = \iint p_\mu(z) d\mu(z).    \]
 The equilibrium measure of a compact set $K$ is the measure $\nu$ such that 
 \[  I(\nu) = \sup_{\mu \in \mathcal{P}(K)} I(\mu)  \]
 where $\mathcal{P}(K)$ is the set of Borel probability measures on $K$.  
 Every compact set posseses an equilibrium measure \cite[Theorem 3.3.2]{Ransford_book}. Now the logarithmic capacity of a set $E \subseteq \mathbb{C}$ is defined as
 \[ c(E) = \sup_{  \substack{\mu \in \mathcal{P}(K) \\ K \subseteq E \ {\text{compact}}} } e^{I(\mu)}.  \]
\end{definition}
 
 For compact sets $K$ we have 
 \[  c(K) = e^{I(\nu)} \]
 where $\nu$ is the equilibrium measure of $K$.\\
 
We say a property holds \emph{quasieverywhere} if it holds except possibly on a set of logarithmic capacity zero. 
\begin{remark}
There are sets
$E$ which have Lebesgue measure zero but positive logarithmic capacity. However, if a
property holds quasi-everywhere, it holds almost everywhere. 
\end{remark} 
In what follows we will often drop the word ``logarithmic'' and use simply the word ``capacity''. 
 

 
 
 
 The outer logarithmic capacity  of a set $E \subseteq \mathbb{C}$ \cite{El-Fallah_etal_primer} is defined as 
 \[  c^*(E) = \inf_{\substack{ E \subseteq U \subseteq \mathbb{C} \\ U \ \text{open}}} c(U).    \]
 By Choquet's theorem \cite{Choquet,Ransford_book}, for any bounded Borel set $E$, $c(E) = c^*(E)$.   

\begin{lemma} \label{le:Borel_replace} Every bounded set of outer logarithmic capacity zero in $\mathbb{\mathbb{C}}$ is contained in a Borel set of logarithmic capacity zero.  
\end{lemma}
 \begin{proof}
  Let $F$ be a set of outer capacity zero.  Thus, there are open sets $U_n$, $n \in \mathbb{N}$, containing $F$ such that $c(U_n) < 1/n$.  We can choose these sets such that $U_{n+1} \subseteq U_n$ for all $n$, by replacing $U_n$ with $U_n' = \cap_{k=1}^n U_k$ if necessary, and observing that by \cite[Theorem 5.1.2(a)]{Ransford_book} $c(U_n') \leq c(U_n) <1/n$ since $U_n' \subseteq U_n$.  
  
  The set $V = \cap_{n=1}^\infty U_n$ is a Borel set containing $F$.  Since $V \subseteq U_n$ for all $n \in \mathbb{N}$, again applying \cite[Theorem 5.1.2(a)]{Ransford_book} we see that $c(V) < 1/n$ for all $n \in \mathbb{C}$, so $c(V) =0$. 
 \end{proof}
 
 Quasiconformal maps preserve compact sets of logarithmic capacity zero.
 We are grateful to Malik Younsi for suggesting the following lemma and its proof.
 \begin{lemma}
  Let $K \subseteq \mathbb{C}$ be compact.  Let $U$ be an open set containing $K$ and let $f:U \rightarrow V$ be a homeomorphism onto the open set $V \subset \mathbb{C}$, which is H\"older continuous of exponent $\alpha >0$.   
  If $K$ has capacity zero, then $f(K)$ also has capacity zero.    
 \end{lemma}
 \begin{proof}
  Let $\mu$ be a probability measure with support in $f(K)$.  If we define the Borel probability measure 
  $\nu = f^*(\mu)$ by $\nu(A) = \mu(f(A))$ then 
  \begin{align*}
    I(\mu) & = \iint_{V} \iint_{V} \log{|z-w|} \, d\mu(z) \, d\mu(w) = 
    \iint_{U} \iint_{U} \log{|f(z)-f(w)|} \, d\nu(z) \, d\nu(w) \\
    & = \iint_{U} \iint_{U} \log{\frac{|f(z)-f(w)|}{|z-w|^\alpha}} \, d\nu(z) \, d\nu(w) 
      + \alpha I(\nu) \\
  \end{align*}
  Since the capacity of $K$ is zero, $I(\nu) = -\infty$.  Moreover the H\"older-continuity of $f$ means that $|f(z)-f(w)| \leq M |z-w|^\alpha$, which therefore yields $I(\mu) = -\infty$. Now since $\mu$ was arbitrary, $f(K)$ has capacity zero.  
 \end{proof}
 From this, it follows that 
 \begin{lemma}  \label{le:Holder_preserves_capzero}
  Let $E \subseteq \mathbb{C}$ be a bounded Borel set.  Let $f:\mathbb{C} \rightarrow \mathbb{C}$ be a homeomorphism which is H\"older continuous of exponent $\alpha >0$.  If $E$ has capacity zero, then $f(E)$ has capacity zero.  
 \end{lemma}
 \begin{proof}
   By \cite[Theorem 5.1.2(b)]{Ransford_book},
  \begin{equation} \label{eq:inner_capacity_temp}
   c(f(E)) = \sup_{\substack{K \subseteq f(E) \\ K \  \text{compact}}} c(K)
  \end{equation}
  (indeed, this follows directly from the definition of capacity).   
  Thus, if $f(E)$ does not have capacity zero, there is a compact set $K \subseteq f(E)$ such that $c(K) >0$.  Since $f$ is a homeomorphism, $f^{-1}(K)$ is a compact subset of $E$, so by the previous lemma $c(f^{-1}(K)) >0$.  Applying (\ref{eq:inner_capacity_temp}) again with $E$ in place of $f(E)$ we see that $c(E) > 0$, a contradiction.  
 \end{proof}
 
 In particular, quasiconformal maps preserve bounded Borel sets of capacity zero, since they are uniformly H\"older on every compact subset (\cite{Lehto_Virtanen} p71).
 
 \begin{corollary} \label{co:quasisymmetry_nullpreserve} Let $\phi:\mathbb{S}^1 \rightarrow \mathbb{S}^1$ be a quasisymmetry.  Then $I \subseteq \mathbb{S}^1$ is a Borel set of logarithmic capacity zero if and only if $\phi(I)$ is a Borel set of logarithmic capacity zero. 
 \end{corollary}
 \begin{proof} 
  By the Beurling-Ahlfors extension theorem (Theorem \ref{th:qs_are_qc_boundary}), $\phi$ has a quasiconformal extension $\Psi:\disk \rightarrow \disk$.  In fact, this extends to a quasiconformal map of the plane via 
  \begin{equation*}
   \Phi(z)  = \left\{  \begin{array}{cc}  \Psi(z) & z \in \text{cl} \,\disk \\ 1/\overline{\Psi(1/\bar{z})} & z \in \mathbb{C} \backslash \text{cl} \, \disk.   \end{array}   \right.
  \end{equation*}
  Since a quasiconformal map is uniformly H\"older-continuous on every compact subset, the claim follows from Lemma \ref{le:Holder_preserves_capzero}.
 \end{proof}
 
 We now define null sets.  Note that in the sphere, the boundary of a domain is taken with respect to the sphere topology.  So it might include $\infty$.   
 \begin{definition} Let $\Omega$ be a Jordan domain in $\sphere$ with boundary $\Gamma$.  Let $I \subset \Gamma$.  We say that $I$ is null with respect to $\Omega$ if $I$ is a Borel set, and there is a biholomorphism $f:\disk^+ \rightarrow \Omega$, such that $f^{-1}(I)$ has logarithmic capacity zero.  
 \end{definition}
 The meaning of $f^{-1}(I)$ requires an application of Carath\'eodory's theorem, which says that since $\Omega$ is a Jordan domain, any biholomorphism $f$ has a continuous extension which takes $\mathbb{S}^1$ homeomorphically to $\Gamma$.  This is true even if $\Gamma$ contains the point at $\infty$, as can be seen by composing $f$ by a M\"obius transformation taking $\Gamma$ onto a bounded curve and applying Carath\'eodory's theorem there, and then using the fact that $T$ is a homeomorphism of the sphere.  Thus $f^{-1}(I)$ is defined using the extension of $f$.  Note that $I$ is a Borel set if and only if $f^{-1}(I)$ is Borel.
 
 If there is one biholomorphism $f$ such that $f^{-1}(I)$ has capacity zero, then $g^{-1}(I)$ has capacity zero for all biholomorphisms $g:\disk^+ \rightarrow \Omega$.  This is because the M\"obius transformation $T = g^{-1} \circ f$ preserves Borel sets of capacity zero in $\mathbb{S}^1$, for example by Corollary \ref{co:quasisymmetry_nullpreserve}.  
 Also, it is easily seen that one may replace $\disk^+$ with $\disk^-$ in the above definition.
 
 If $\Gamma$ is a Jordan curve, bordering domains $\Omega_1$ and $\Omega_2$, then $I$ might be null with respect to $\Omega_1$ but not with respect to $\Omega_2$, or vice versa. However, for quasicircles, the concept of null set is independent of the choice of ``side'' of the curve.  This is a key fact.
 \begin{theorem}  \label{th:null_both_sides_quasicircle}
  Let $\Gamma$ be a quasicircle in $\sphere$, and let $\Omega_1$ and $\Omega_2$ be the connected components of $\sphere \backslash \Gamma$.  Then $I\subset \Gamma$ is null with respect to $\Omega_1$ if and only if it is null with respect to $\Omega_2$.
 \end{theorem}
 \begin{proof}  Choose a M\"obius transformation $T$ so that $T(\Gamma)$ is a quasicircle in $\mathbb{C}$ (and in particular bounded).  Clearly $T(I)$ is null in $T(\Gamma)$ with respect to $T(\Omega_i)$ if and only if it is null in $\Gamma$ with respect to $\Omega_i$, for $i=1,2$.  Thus it suffices to prove the claim for a quasicircle $\Gamma$ in $\mathbb{C}$.  
 
 Let $\Omega^+$ and $\Omega^-$ be the bounded and unbounded components of the complement of $\Gamma$ respectively.   Let $f_\pm:\disk^\pm \rightarrow \Omega^\pm$ be conformal maps.  These have quasiconformal extensions to $\mathbb{C}$.   Thus $\phi = f_-^{-1} \circ f_+$ has a quasiconformal extension to $\mathbb{C}$, and in particular is a quasisymmetry.  
 
 By definition $I$ is null with respect to $\Omega^+$ if and only if $f_+^{-1}(I)$ is a Borel set of logarithmic capacity zero in $\mathbb{S}^1$.  By Corollary \ref{co:quasisymmetry_nullpreserve}, this holds if and only if $f_-^{-1}(I) = \phi (f_+^{-1}(I))$ is a Borel set of logarithmic capacity zero in $\mathbb{S}^1$, that is if and only if $I$ is null with respect to $\Omega^-$.
 \end{proof}
 \begin{remark} \label{re:quasicircle_capacity_zero}
  The proof can be modified to show that if $\Omega$ is a Jordan domain bounded by a quasicircle, and $I \subseteq \Gamma$ is null with respect to $\Omega$, then there is a M\"obius transformation $T$ such that $T(I)$ is a bounded Borel set of capacity zero (in fact, for any M\"obius transformation such that $T(I)$ is bounded, it is a set of capacity zero).
 \end{remark}

\end{subsection}
\begin{subsection}{Dirichlet space of the disk and boundary values}  \label{se:boundary_values_disk}  \ 

If we write $f(z)=\sum_{n=0}^{\infty} a_n z^n$
as a power series and setting $z =re^{i\theta}$, one can use polar coordinates to see that
\begin{equation}
  D_{\disk^+}(f)= \sum_{n=0}^{\infty} n |a_n|^2 .
\end{equation}

Another important fact about the Dirichlet space is that if $f\in \mathcal{D}(\disk^+)$ then $f$ has radial boundary values, i.e. for almost every $z\in \mathbb{S}^1$, the limit $\lim_{r\to 1^-} f(rz)=: \tilde{f}(z)$ exists, see e.g. \cite{El-Fallah_etal_primer}. Moreover a result of J. Douglas \cite{Douglas}, one has that

\begin{equation}\label{Douglas formula}
  D_{\disk^+}(f)= \int_{0}^{2\pi} \int_{0}^{2\pi} \frac{|\tilde{f}(z)-\tilde{f}(\zeta)|^2}{|z-\zeta|^2}\,|dz|\, |d \zeta|.
\end{equation}
Now for $\zeta\in\mathbb{S}^1$, let 
\begin{equation}\label{defn:dirichlet potential}
  K(\zeta)= \frac{1}{|1-\zeta|^{1/2}},
\end{equation}
and define the convolution of two functions $f,\, g$ defined on the unit circle via
\begin{equation}
  (f\ast g)(z):=\int_{0}^{2\pi} f(z\overline{\zeta}) \, g(\zeta) \, |d \zeta|.
\end{equation}
If $z\in \disk^+$ and $\zeta\in\mathbb{S}^1$ then 
\begin{equation}
 P_z(\zeta) = \frac{1-|z|^2}{
 |z-\zeta|^2},
 \end{equation}
 denotes the Poisson kernel of the disk, and we set

\begin{equation}
P(u)(z)= \int_{0}^{2\pi} P_z(\zeta) u(\zeta)\, |d\zeta|.
 \end{equation}


Regarding boundary values of harmonic functions with bounded Dirichlet energy, we will use the following two results:
\begin{theorem}\label{thm: radial lim for dirichlet}
Let $f$ be a harmonic function in $\disk^+$ and $D(f)<\infty.$ Then
$\tilde{f}(z) := \lim_{r\to 1^-} f(rz)$
quasieverywhere. 
\end{theorem}
\begin{proof}
 This is a classical result due to Beurling, see \cite{Beurling}.
\end{proof}
\begin{theorem}\label{thm: fund of harmonic dirichlet}
 Let $f = P(K\ast\varphi)$ for some $\varphi\in L^2(\mathbb{S}^1)$. For fixed $\theta \in [0,2\pi)$, consider the following four limits:
\begin{enumerate}
    \item $\lim_{r\to 1^-} f(rz)$ $($the radial limit of $f$$)$
    \item $\lim_{N\to \infty}\sum_{n=-N}^{N} \widehat{(K\ast \varphi)}(n) e^{in\theta}$ $($the limit of the partial sums of the Fourier series for $K\ast \varphi$$)$
\item $\lim_{h\to 0^+}\frac{1}{2h} \int_{\theta-h}^{\theta+h} (K\ast \varphi)(e^{it})\, dt$ $($the boundary trace of  $f$$)$
\end{enumerate}
\end{theorem}

If one of them exists and is finite, they all do and they are equal. The equivalence of (1) and (2) is Abel's theorem and a result of E. Landau \cite{Landau} p. 65-- 66. The equivalence of (2) and (3) is from Beurling in \cite{Beurling}.

The boundary behaviour of elements of Dirichlet space is better than this result indicates in two ways.  Firstly, the limit exists not just rdially but non-tangentially.  Secondly, the limit exists not just almost everywhere, but up to a set of outer capacity zero.

 We now define non-tangential limit. 
  A non-tangential wedge in $\disk$ with vertex at $p \in \mathbb{S}^1$ is a set of the form
 \begin{equation}\label{nontangential wedge}
     W(p,M)  = \{ z \in \disk : |p-z| < M(1-|z|)   \}
 \end{equation}  
 for $M \in (1,\infty)$.  
 \begin{definition}\label{defn:CNT limit}
 We say that a function $h:\disk \rightarrow \mathbb{C}$ has a non-tangential limit of $\zeta$ at $p$ in $\mathbb{S}^1$ if 
 \[ \lim_{\substack{ z \rightarrow p \\ z \in W(p,M) }} h(z) = \zeta  \]
 for all $M \in (1,\infty)$.  
 \end{definition}
 Equivalently, in the above definition one may replace non-tangential wedges with Stolz angles
 \[  \Delta(p,\alpha,\rho) = \{ z : |\text{arg}(1-\bar{p} z) |< \alpha \ \text{and} \ |z-p| < \rho   \} \]
 where $\alpha \in (0,\pi/2)$ and $\rho \in (0,2 \cos{\alpha})$. \\

The following theorem of Beurling \cite[Theorem 3.2.1]{El-Fallah_etal_primer} improves our understanding of the boundary behaviour, as promised.
\begin{theorem}  \label{th:boundaryvalues_exist_disk}
 Let $h \in \mathcal{D}_{\mathrm{harm}}(\disk)$.  Then there is a set $I \subseteq \mathbb{S}^1$ of outer logarithmic capacity zero such that the non-tangential limit of $h$ exists on $\mathbb{S}^1 \backslash I$.  
\end{theorem}  
\begin{remark} \label{re:Borel_replace_circle} By Lemma \ref{le:Borel_replace}, we may take $I$ to be a Borel set of capacity zero.   
\end{remark}

 Since a wedge at $p$ contains a radial segment terminating at $p \in \mathbb{S}^1$, it is immediate that if the non-tangential limit exists, then the radial limit exists and equals the non-tangential limit.
Using Theorems \ref{thm: radial lim for dirichlet} and \ref{thm: fund of harmonic dirichlet} one then has:

\begin{theorem} Let $h \in \mathcal{D}_{\mathrm{harm}}(\disk)$.  Let $H$ be the non-tangential boundary values of $h$. The Fourier series of $H$ converges, except possibly on a set of outer logarithmic capacity zero, to $H$.  
\end{theorem}

{Finally, we have the following. 
\begin{theorem} \label{th:extension_unique}
 Let $h_1, h_2 \in \mathcal{D}_{\mathrm{harm}}(\disk)$.  If the non-tangential limits of $h_1$ and $h_2$ are equal except on a Borel set of capacity zero, then $h_1 = h_2$. 
\end{theorem}
To see this, it is enough to see that if the non-tangential limit of $h\in \mathcal{D}_{\mathrm{harm}}(\disk) $ is zero, then $h$ is zero. This follows essentially from the equality of the radial and non-tangential limits and \eqref{Douglas formula}. }

\end{subsection}
\begin{subsection}{Conformally non-tangential boundary values} \label{se:CNT_boundaryvalues} \ 
 
 We now extend the notion of non-tangential limits to arbitrary Jordan domains.  This extension is an immediate consequence of the Riemann mapping theorem, and is uniquely determined by the requirement that the definition be conformally invariant.  Although this extension is by itself trivial, substantial results arise when one considers boundary values from two sides of the curve, as we will see in Section \ref{se:transmission}.

 \begin{definition} \label{de:CNT}   
  Let $\Omega$ be a Jordan domain in $\sphere$ with boundary $\Gamma$. Let $h:\Omega \rightarrow \mathbb{C}$ be a function. 
  We say that the conformally non-tangential (CNT) limit of $h$ is $\zeta$ at $p \in \Gamma$ if, for a biholomorphism $f:\disk^+ \rightarrow \Omega$, the non-tangential limit of $h \circ f$ is $\zeta$ at $f^{-1}(p)$.   
 \end{definition}
 The existence of the limit does not depend on the choice of biholomorphism, as the following lemma shows.
 \begin{lemma} \label{le:automorphism_invariance_NT}
  Let $h: \disk \rightarrow \mathbb{C}$, and let $T:\disk \rightarrow \disk$ be a disk automorphism.  Then $h$ has a non-tangential limit at $p \in \partial \disk$ if and only if $h \circ T$ has a non-tangential limit at $T^{-1}(p)$, and these are equal.  
 \end{lemma}
 \begin{proof}
  The claim follows from the easily verified fact that every Stolz angle at $p$ is contained in the image under $T$ of a Stolz angle at $T^{-1}(p)$, and every Stolz angle at $T^{-1}(p)$ is contained in the image under $T^{-1}$ of a Stolz angle at $p$.   
 \end{proof}
 If $h \circ f$ has non-tangential limit $\zeta$ at $f^{-1}(p)$ for a biholomorphism $f:\disk^+ \rightarrow \Omega$, then $h \circ g$ has non-tangential limit $\zeta$ at $g^{-1}(p)$ for any biholomorphism $g:\disk^+ \rightarrow \Omega$, by applying the lemma above to $T = g^{-1} \circ f$.  
 
 \begin{remark} \label{re:CNT_conformally_invariant}
  This notion of CNT limit is conformally invariant, in the following sense.  If $\Omega_1$ and $\Omega_2$ are Jordan domains and $f:\Omega_1 \rightarrow \Omega_2$ is a biholomorphism, then the CNT boundary values of $h:\Omega_2 \rightarrow \mathbb{C}$ exists and equals $\zeta$ at $p \in \partial \Omega_2$ if and only if the CNT limit of $h \circ f$ exists and equals $\zeta$ at $f^{-1}(p) \in \Omega_1$.  The only role that the regularity of the boundary curves plays in the definition, is that we use Carath\'eodory's theorem implicitly to uniquely associate points on the boundary of $\partial \Omega_1$ with points on $\partial \Omega_2$. Therefore the boundary is required to be a Jordan curve.  
  However, even this condition can be removed, by replacing the boundary of the domain in $\sphere$ with the ideal boundary \cite{Schippers_Staubach_monograph}.
 \end{remark}
 \begin{remark}
   An obviously equivalent definition is as follows. The CNT limit of $h:\Omega \rightarrow \mathbb{C}$ is $\zeta$ at $p \in \partial \Omega$ if, given a conformal map $f:\disk^+ \rightarrow \Omega$, defining $V(p,M)  = f(W(f^{-1}(p),M))$, one has that
  \[ \lim_{\substack{ z \rightarrow p \\ z \in V(p,M) }} h(z) = \zeta.  \]
  Note that, treating the ideal boundary of $\Omega$ as a border of $\Omega$ \cite{AhlforsSario} (which can be done since $\Omega$ is biholomorphic to a disk), the angle of the wedge $V(p,M)$ has a sensible geometric meaning.  That is, let $\phi$ be a border chart taking a neighbourhood $U$ of $p$ in $\Omega$ to a half-disk which takes a segment of the ideal boundary containing $p$ to a segment of the real axis.  In this neighbourhood, $\phi(V(p,M) \cap U)$ is a wedge in the ordinary sense.  The boundary of $\phi(V(p,M) \cap U)$ meets the real axis at two angles which are independent of the choice of chart.  
 \end{remark}

 Using CNT limits, we can formulate a conformally invariant version of Beurling's theorem on non-tangential limits.  
 \begin{theorem}
  Let $\Omega$ be a Jordan domain with boundary $\Gamma$.  For $h \in \mathcal{D}_{\mathrm{harm}}(\Omega)$, the \emph{CNT} boundary values of $h$ exist at every point in $\Gamma$ except possibly on a null set $I \subset \Gamma$ with respect to $\Omega$.  If $h_1$ and $h_2$ are \emph{CNT} boundary values of some element of $H_1$ and $H_2$ in $\mathcal{D}_{\mathrm{harm}}(\Omega)$ respectively, and $h_1 = h_2$ except possibly on a null set, then $H_1 = H_2$.      
 \end{theorem}
 This follows directly from Theorem \ref{th:boundaryvalues_exist_disk}, {Theorem \ref{th:extension_unique},} Lemma \ref{le:Borel_replace}, and the conformal invariance of CNT limits (Remark \ref{re:CNT_conformally_invariant}).    \\ 
 
 We now define a particular class of boundary values.  Let $\Omega$ be a Jordan domain in $\sphere$ with boundary $\Gamma$.  We say that two functions  $h_1$ and $h_2$ on $\Gamma$ are equivalent if $h_1 = h_2$ except possibly on a null set $I$ with respect to $\Omega$.  Denote the set of such functions up to equivalence by $\mathcal{B}(\Gamma,\Omega)$.  We say that $h_1 = h_2$ if they are equivalent.  
 
 \begin{definition}
  The Osborn space of $\Gamma$ with respect to $\Omega$,  denoted $\mathcal{H}(\Gamma,\Omega)$, is the set of functions $h \in \mathcal{B}(\Gamma,\Omega)$ which arise as boundary values of elements of $\mathcal{D}_{\text{harm}}(\Omega)$.  
 \end{definition}
 
 We then define the trace operator
 \[  \mathbf{b}_{\Omega,\Gamma}:\mathcal{D}_{\text{harm}}(\Omega) \rightarrow \mathcal{H}(\Gamma,\Omega)  \]
 and the extension operator 
 \[ \mathbf{e}_{\Gamma,\Omega}:\mathcal{H}(\Gamma,\Omega) \rightarrow \mathcal{D}_{\text{harm}}(\Omega)  \]
 accordingly.

In the case that $\Gamma = \mathbb{S}^1$ and $\Omega = \disk^+$, these maps have simple expressions in terms of the Fourier series: 

\[  {\bf{b}}_{\disk^+,\mathbb{S}^1}\,\Big(\sum_{n=0}^\infty a_n z^n + \sum_{n=1}^\infty a_{-n} \bar{z}^n\Big) =  \sum_{n=-\infty}^\infty a_n e^{in \theta},  \]
and
\[  {\bf{e}}_{\mathbb{S}^1,\disk^+}\left( \sum_{n=-\infty}^\infty a_n e^{in \theta} \right)  = 
    \sum_{n=0}^\infty a_n z^n + \sum_{n=1}^\infty a_{-n} \bar{z}^n  . \]

 Similar expressions can be obtained for $\disk^-$.\\  
 
 By the Jordan curve theorem, a Jordan curve divides $\sphere$ into two connected components $\Omega_1$ and $\Omega_2$.  This leads to the following important question: \\ 
 
 {\bf{Question:}}
 for which Jordan curves $\Gamma$ is $\mathcal{H}(\Gamma,\Omega_1) = \mathcal{H}(\Gamma,\Omega_2)$? \\ 
 
 That is the topic of the next section.
 
\end{subsection}

\end{section}
\begin{section}{Transmission of harmonic functions in quasicircles}  \label{se:transmission}
\begin{subsection}{Vodop'yanov-Nag-Sullivan theorem}  \label{se:Nag_Sullivan}

{

First we recall a result due to K. Vodop'yanov \cite{Vodopyanov} regarding the boundedness of composition operators on fractional Sobolev spaces which will be useful in proving a characterization results for quasisymmetric homeomorphims of $\mathbb{S}^1$. However the original result is formulated for Sobolev spaces on the real line. To this end, one defines the homogeneous Sobolev $($or Besov$)$ space $\dot{H}^{1/2}(\R)$ as the closure of $\mathcal{C}^\infty_c(\R)$ $($smooth compactly supported functions$)$ in the seminorm 
\begin{equation}\label{equiv hhalf norms}
\Vert f\Vert_{\dot{H}^{1/2}(\R)}=\left( \int_\R\int_\R\frac{|f(x)-f(y)|^2}{|x-y|^2}\, dx\, dy\right)^{\frac{1}{2}}.
\end{equation}
  \begin{theorem}\label{Vodopyanov}
   The composition map $\mathbf{C}_\phi (h):= h\circ \phi$ is bounded from $\dot{H}^{\frac{1}{2}}(\R)$ to
$\dot{H}^{\frac{1}{2}}(\R),$ if and only if $\phi$ is a quasisymmetric homeomorphism of $\R$ to $\R$.
  \end{theorem}
 See \cite{Vodopyanov} Theorem 2.2.
 \begin{remark}
As is customary in Sobolev space theory, the constructions of compositions, traces and so on, are done using dense  subsets of Sobolev spaces, e.g. the set of smooth compactly supported functions, where for example the composition $\mathbf{C}_\phi (h)$ is well-defined (i.e for $h\in \mathcal{C}_c^{\infty}(\R)).$ Thereafter one seeks boundedness estimates with bounds that are independent of $h$ and extends the results by density to the desired Sobolev space.   
 \end{remark}
As a side-note, using the existence of the solution of the Dirichlet's problem and quasi-isometric extensions of quasisymmetries, the authors of the current exposition showed in \cite{SchippersStaubach_Proc} that $\mathbf{C}_\phi$ is bounded on $H^{\frac{1}{2}}(\mathbb{S}^1)$. To see this first we extend the quasisymmetry $\phi$ to a quasi-isometry $\Phi$ on $\mathbb{D},$ which is possible thanks to a result of Z. Ibragimov, see \cite[Theorem 3.1 (5)]{Ibra}, and the conformal equivalence of the half plane and the disk.
Next we let $F$ to be the harmonic extension of $f\in H^{\frac{1}{2}}(\mathbb{S}^1)$ which according to Proposition 1.7 on page 360 in \cite{Taylor} belongs to $H^1(\disk)$ and satisfies the estimate $\Vert F\Vert_{H^1(\disk)}\leq C \Vert f\Vert_{H^{\frac{1}{2}}(\mathbb{S}^1)}$.\\ Now since the boundary value of $\mathbf{C}_\Phi (F)$ is $\mathbf{C}_\phi (f)$, by the continuity of the restriction to the
boundary (see \cite{Taylor} Proposition 4.5), one has that $\|\mathbf{C}_ \phi (f) \|_{{H}^{\frac{1}{2}}(\mathbb{S}^1)} \leq C \| \mathbf{C}_\Phi(F) \|_{H^1(\mathbb{D})}.$  But it is well-known that $\Vert\mathbf{C}_\Phi (F) \Vert_{H^1(\mathbb{D})}\lesssim \Vert
F\Vert_{H^1(\mathbb{D})},$   for every quasi-isometric homeomorphism $\Phi$, see e.g. \cite[Theorem $4.4'$]{GR}  and its Corollary 1. Thus
\[  \|\mathbf{C}_ \phi (f) \|_{{H}^{\frac{1}{2}}(\mathbb{S}^1)} \lesssim \| \mathbf{C}_\Phi (F) \|_{H^1(\mathbb{D})} \lesssim
     \| F \|_{H^1(\mathbb{D})} \lesssim \| f \|_{H^{1/2}(\mathbb{S}^1)},  \]
which proves the claim.\\

In \cite{NS} S. Nag and D. Sullivan showed that quasisymmetries of $\mathbb{S}^1$ are characterized by the fact that they are bounded maps of the Sobolev space $H^{1/2}(\mathbb{S}^1)/\R$ and in doing so reproved Theorem \ref{Vodopyanov}. In what follows we give a presentation of their result adding also some more references for the sake of completeness.\\ 
  \begin{theorem}  \label{th:quasisymmetry_CNT_charac}
  Let $\phi:\mathbb{S}^1 \rightarrow \mathbb{S}^1$ be a homeomorphism.  Then the following are equivalent.  
  \begin{enumerate}
   \item $\phi$ is a quasisymmetry;
   \item $\phi$ has the following three properties:
    \begin{enumerate}
     \item $\phi$ takes Borel sets of capacity zero to Borel sets of capacity zero;
     \item for every $h \in \mathcal{H}(\mathbb{S}^1)$, $\mathbf{C}_\phi (h) \in \mathcal{H}(\mathbb{S}^1)$; 
     \item the map $h \mapsto h \circ \phi$ obtained in $($b$)$ is bounded in the sense that 
      there is a $C$ such that 
      \begin{equation}\label{energy inequality} 
      D_{\disk^+}( {\mathbf{e}}_{\mathbb{S}^1,\disk^+} \,( h \circ \phi )) \leq C D_{\disk^+}( {\mathbf{e}}_{\mathbb{S}^1,\disk^+} \, h ).
      \end{equation}
    \end{enumerate}
    
  \end{enumerate}
 \end{theorem}

\begin{proof} That (1) implies 2(a) is Corollary \ref{co:quasisymmetry_nullpreserve}.\\

That (1) implies 2(b) are equivalent can be shown by transferring the problem to the real line.
As a consequence of a much more general result for divergence-type elliptic operators due to A. Barton and S. Mayboroda \cite[Theorem 7.18]{bartmay}, if $\h$ denotes the upper half-plane, then there exists a solution to the Dirichlet's problem 

\begin{equation}
\begin{cases}
\Delta u=0\,\,\, \mathrm{on} \,\,\, \h,\\
u|_{\partial\h} =f \in \dot{H}^{1/2}(\R),
\end{cases}
\end{equation}
which is unique $($up to additive constants$)$  and the estimate
\begin{equation}\label{estim for sol to dir}
\| u \|_{\dot{H}^1(\h)} \leq C \| f \|_{\dot{H}^{\frac{1}{2}}(\R)},
\end{equation}
holds.\\
Now the fact that for every $h \in \mathcal{H}(\mathbb{R})$, the composition $\mathbf{C}_\phi h\in \mathcal{H}(\mathbb{R})$, is then a consequence of Theorem \ref{Vodopyanov}. \\

That (1) implies 2(c) can be shown as follows.  Let $H \in \mathcal{D}_{\text{harm}}(\disk^+)$ be the function whose CNT boundary values equal $h$ quasieverywhere.  Let $\Phi:\disk^+ \rightarrow \disk^+$ be a quasiconformal map whose boundary values equal $\phi$ (which exists by the aforementioned Beurling-Ahlfors extension theorem). By quasi-invariance of Dirichlet energy (see e.g. \cite{Ahlfors_quasiinv}) we have
  \[  D_{\disk^+}( \mathbf{C}_{\Phi}H  ) \leq C' 
    D_{\disk^+}( H ) = C' D_{\disk^+}( \mathbf{e}_{\mathbb{S}^1,\disk^+} h )    \]
  where $C'$ is of course independent of $H$. 
Let $F:=\mathbf{C}_{\Phi}H - \mathbf{e}_{\mathbb{S}^1,\disk^+} \, (\mathbf{C}_{\phi}h )\in H^1( \disk^+)$. Then using  $F|_{\mathbb{S}^1} =0,$ the harmonicity of $\mathbf{e}_{\mathbb{S}^1,\disk^+} \, (\mathbf{C}_{\phi}h)$ and the Sobolev space divergence theorem (see e.g. Theorem 4.3.1 page 133 in \cite{EvansGariepy}) one can show that
  $$\int_{\disk^+} \partial (\mathbf{e}_{\mathbb{S}^1,\disk^+} \, (\mathbf{C}_{\phi}h)) \, \overline{\partial F}\, dA=0.$$ This yields that
  \begin{align}
  D_{\disk^+}( \mathbf{e}_{\mathbb{S}^1,\disk^+} \, (\mathbf{C}_{\phi}h))& \leq    D_{\disk^+}( \mathbf{e}_{\mathbb{S}^1,\disk^+} \, (\mathbf{C}_{\phi}h)) +   D_{\disk^+}(F) \nonumber \\ & =  D_{\disk^+}( \mathbf{e}_{\mathbb{S}^1,\disk^+} \, (\mathbf{C}_{\phi}h)) +2 \int_{\disk^+} \partial( \mathbf{e}_{\mathbb{S}^1,\disk^+} \, (\mathbf{C}_{\phi}h )) \, \overline{\partial F}\, dA +D_{\disk^+}(F) \\ & =  D_{\disk^+}(\mathbf{C}_{\Phi}H). \nonumber
  \end{align} 
  Finally if \eqref{energy inequality} is valid for any homeomorphism $\phi$, then transference of Douglas's result to the real line in equation \eqref{Douglas formula} yields that $\Vert \mathbf{C}_{\phi}u\Vert_{\dot{H}^{1/2}(\R)}\leq C \Vert u\Vert_{\dot{H}^{1/2}(\R)}$, which by part 2(b) yields that $\phi$ is a quasisymmetric homeomorphism of the real line.
  This completes the proof.
\end{proof}

In the remainder of the paper, we will say that an operator between Dirichlet spaces is bounded with respect to Dirichlet energy if it satisfies an estimate of the form given by equation (\ref{energy inequality}).

Conditions (2)(a) and (2)(b) of Theorem are not easy to verify, but of course the direction (2) $\rightarrow$ (1) can be stated in the following way.

\begin{theorem}  \label{th:experimental_NSV}
  Let $\phi:\mathbb{S}^1 \rightarrow \mathbb{S}^1$ be a homeomorphism.  Assume that there is a dense set $\mathscr{L} \subseteq \dot{H}^{1/2}(\mathbb{S}^1)$ such that $\mathscr{L} \subseteq \mathcal{C}(\mathbb{S}^1)$ and 
  there is an $M$ such that $\| \mathbf{C}_{\phi}h \|_{\dot{H}^{1/2}(\mathbb{S}^1)} \leq M \| h \|_{\dot{H}^{1/2}(\mathbb{S}^1)}$.  Then $\phi$ is a quasisymmetry.
 \end{theorem}
\begin{proof}
 This is also a result whose proof is embedded in the proof of Theorem \ref{Vodopyanov}. See also Corollary 3.2 in \cite{NS} and Theorem 1.3 in \cite{Bourdaud}.
\end{proof}}

\end{subsection}
\begin{subsection}{Transmission (Overfare)} \ \label{se:transmission_sub}

 We are now able to prove the transmission theorem in the simplest case.
 \begin{theorem} \label{th:transmission}
  Let $\Gamma$ be a Jordan curve in $\sphere$, and let $\Omega_1$ and $\Omega_2$ be the components of the complement.  The statements \emph{(1), (2)}, and \emph{(3)} below are equivalent.
  \begin{enumerate}
   \item $\Gamma$ is a quasicircle.
   \item 
    \begin{enumerate}
     \item If $I \subseteq \Gamma$ is null with respect to $\Omega_1$ then it is null with respect to $\Omega_2$, 
     \item $\mathcal{H}(\Gamma,\Omega_1) \subseteq \mathcal{H}(\Gamma,\Omega_2)$, and 
     \item the map $\mathbf{e}_{\Gamma,\Omega_2} \mathbf{b}_{\Gamma,\Omega_1}:\mathcal{D}_{\mathrm{harm}}(\Omega_1) 
      \rightarrow \mathcal{D}_{\mathrm{harm}}(\Omega_2)$ is bounded {\it with respect to Dirichlet energy}. 
    \end{enumerate}
   \item 
    \begin{enumerate}
     \item If $I \subseteq \Gamma$ is null with respect to $\Omega_2$ then it is null with respect to $\Omega_1$, 
     \item $\mathcal{H}(\Gamma,\Omega_2) \subseteq \mathcal{H}(\Gamma,\Omega_1)$, and 
     \item the map $\mathbf{e}_{\Gamma,\Omega_1} \mathbf{b}_{\Gamma,\Omega_2}:\mathcal{D}_{\mathrm{harm}}(\Omega_2) 
      \rightarrow \mathcal{D}_{\mathrm{harm}}(\Omega_1)$ is bounded with respect to Dirichlet energy. 
    \end{enumerate}
  \end{enumerate}
 \end{theorem}
 \begin{proof}
  We show that (2) implies (1).  
The truth of either (1) or (2) is unaffected by applying a global M\"obius transformation, so we can assume that $\Gamma$ is bounded.  Let $\Omega^\pm$ be the connected components of the complement in $\sphere$; assume for definiteness that $\Omega_1 = \Omega^+$ (this can be arranged by composing by $1/z$).  
  
  Now let $f_\pm:\disk^\pm \rightarrow \Omega^\pm$ be conformal maps. By Carath\'eodory's theorem, $f_\pm$ each extend to homeomorphisms from $\mathbb{S}^1$ to $\Gamma$; denote the extensions also by $f_\pm$.  The function $\phi= \left. f_+^{-1} \circ f_- \right|_{\mathbb{S}^1}:\mathbb{S}^1 \rightarrow \mathbb{S}^1$ is thus a homeomorphism.    We will show that $\phi$ is a quasisymmetry. Once this is shown, it follows from the conformal welding theorem that $\Gamma$ is a quasicircle.  
  
  To do this, we show that $\phi$ has properties 2(a), 2(b), and 2(c) of Theorem \ref{th:quasisymmetry_CNT_charac}.  Let $I$ be a Borel set of capacity zero in $\mathbb{S}^1$.  Then $f_-(I)$ is by definition null with respect to $\Omega^+$.  So by 2(a) of the present theorem, $f_-(I)$ is null with respect to $\Omega^-$.  By definition $\phi(I) = f_+^{-1}(f_-(I))$ is a Borel set of capacity zero in $\mathbb{S}^1$.  This shows that $\phi$ has the property 2(a) of Theorem \ref{th:quasisymmetry_CNT_charac}.  
  
  Given $h \in \mathcal{H}(\mathbb{S}^1)$, there is an 
  $H \in \mathcal{D}_{\text{harm}}(\disk^+)$ with CNT boundary values equal to $h$ except possibly on a null set $I$.  Also, $H \circ f_+^{-1} \in \mathcal{D}_{\text{harm}}(\Omega^+)$.  By definition, $H \circ f_+^{-1}$ has CNT boundary values except on the null set $f_+(I)$.  By assumption 2(b) of the present theorem, there is a function
  \[  u = \mathbf{e}_{\Gamma,\Omega^-} \mathbf{b}_{\Gamma,\Omega^+} (H \circ f_+^{-1}) 
   \in \mathcal{D}_{\text{harm}}(\Omega^-)   \] 
  whose CNT boundary values agree with those of $H \circ f_+^{-1}$ except on a null set $K$ containing $f_+(I)$.  Set $I' = f_+^{-1}(K)$, which is a  null set containing $I$.  
  
  By definition, $u \circ f_- \in \mathcal{D}_{\text{harm}}(\disk^-)$ has CNT boundary values except on the null set $f_-^{-1}(K) = \phi^{-1}(I')$, which contains $\phi^{-1}(I)$.  These CNT boundary values agree with $h \circ f_+^{-1} \circ f_- = h \circ \phi$ except on $\phi^{-1}(I')$.  Thus the function $u \circ f_-(1/\bar{z})$ has CNT boundary values equal to $h$ except on $\phi^{-1}(I')$.  That is, 
  \[  u \circ f_-(1/\bar{z}) = \mathbf{e}_{\mathbb{S}^1,\mathbb{D}^+} ( h\circ \phi),  \]
  which shows that $\mathbf{C}_{\phi}h \in \mathcal{H}(\mathbb{S}^1)$.  Since $h$ is arbitrary, this shows that property 2(b) of Theorem \ref{th:quasisymmetry_CNT_charac} holds.
  
  To show that 2(c) of Theorem \ref{th:quasisymmetry_CNT_charac} holds, by 2(c) of the present theorem and conformal invariance of the Dirichlet norm, 
  there is a constant $C>0$ such that $D_{\Omega^+} (\mathbf{e}_{\Gamma,{\Omega^+}} \mathbf{b}_{\Gamma,\Omega^-} v ) \leq C D_{\Omega^{-}}(v)$ for all $v \in \mathcal{D}_{\text{harm}}(\Omega^-)$. Then for arbitrary $h \in \mathcal{H}(\mathbb{S}^1)$, using the notation above we have
  \begin{align*}
   D_{\disk^+} (\mathbf{e}_{\mathbb{S}^1,\disk^+} (h \circ \phi)) & = D_{\disk^+} (u \circ f_-(1/\bar{z})) = D_{\disk^-}(u \circ f_-)   \\
   & =  D_{\Omega^-} (u)  \leq C D_{\Omega^+}(H \circ f_+^{-1}) \\
   & = C D_{\disk^+}(H) = C   D_{\disk^+} (\mathbf{e}_{\mathbb{S}^1,\disk^+} h).
  \end{align*}
  
  Thus $\phi$ is a quasisymmetry, completing the proof that (2) implies (1). 
  It is easy to see that this proof, with minor changes, shows that (3) implies (1).  
  
  So we need only show that (1) implies (2).  Again, we can assume that $\Gamma$ is bounded and denote the bounded and unbounded components of the complement by $\Omega^+$ and $\Omega^-$ respectively.  Let $f_\pm:\disk^\pm \rightarrow \Omega^\pm$ be conformal maps, which have quasiconformal extensions to $\mathbb{C}$.  Thus $\phi = f_+^{-1} \circ f_-$ is a quasisymmetry of $\mathbb{S}^1$, and properties 2(a)--2(c) of Theorem \ref{th:quasisymmetry_CNT_charac} hold.    
  
  Given a Borel set $I \subseteq \Gamma$ which is null with respect to $\Omega^+$, by definition $f_+^{-1}(I)$ is a Borel set of capacity zero in $\mathbb{S}^1$.  Thus since $\phi^{-1}$ is a quasisymmetry, by Theorem \ref{th:quasisymmetry_CNT_charac} $\phi^{-1}(f_+^{-1}(I)) = f_-^{-1}(I)$ is a Borel set of capacity zero.  Thus by definition $I$ is null with respect to $\Omega^-$.  This shows that 2(a) of the present theorem holds.\\  
  
  Denoting $R(z)=1/\bar{z}$, a proof similar to that given above for the reverse implication shows that 2(b) of the present theorem holds with the extension to $\Omega^-$ given by   
  \[   \mathbf{e}_{\Gamma,\Omega^-} \mathbf{b}_{\Gamma,\Omega^+} H = [\mathbf{e}_{\mathbb{S}^1,\disk^+} (\mathbf{b}_{\mathbb{S}^1,\disk^+} (H \circ f_+) \circ \phi)] \circ  R \circ f_-^{-1},      \]
 where $H \in \mathcal{D}_{\text{harm}}(\Omega^+).$\\
 To show that 2(c) of the present theorem holds, let $C$ be the constant in Theorem \ref{th:quasisymmetry_CNT_charac} part 2(c). Then we have 
  \begin{align*}
   D_{\Omega^-} (\mathbf{e}_{\Gamma,\Omega^-} \mathbf{b}_{\Gamma,\Omega^+} H) & =  D_{\disk^+} (\mathbf{e}_{\mathbb{S}^1,\disk^+} (\mathbf{b}_{\mathbb{S}^1,\disk^+} (H \circ f_+) \circ \phi)) \\
  & \leq C D_{\disk^+}(H \circ f_+)  = C D_{\Omega^+} (H)
  \end{align*}
  which completes the proof.
 \end{proof}
 
 Again, conditions (2/3)(a) and (2/3)(b) are difficult to verify in practice. So we give a more practical version of the (2/3) $\rightarrow$ (1) version of this theorem.  
 
 First, we observe that harmonic functions which extend continuously to the boundary have a transmission.  That is, let $\Gamma$ be a Jordan 
 curve separating $\sphere$ into two components $\Omega_1$ and $\Omega_2$, and denote the set of functions continuous on the closure of $\Omega_j$ by $\mathcal{C}(\text{cl} \Omega_j)$ and the set of functions in $\mathcal{C}(\text{cl} \Omega_j)$ which are additionally harmonic in $\Omega_j$ by $\mathcal{C}_{\text{harm}}(\Omega_j)$.   
 Then by the existence and uniqueness of solutions to the Dirichlet problem, given any $h_1 \in \mathcal{C}_{\text{harm}}(\Omega_1)$ there is an $h_2 \in \mathcal{C}_{\text{harm}}(\Omega_2)$ whose boundary values agree with those of $h_1$ everywhere. We thus have the well-defined maps
 \begin{align*}
  \hat{\mathbf{O}}_{\Omega_1,\Omega_2} : \mathcal{C}_{\text{harm}}(\Omega_1) & \rightarrow 
  \mathcal{C}_{\text{harm}}(\Omega_2) \\
  \hat{\mathbf{O}}_{\Omega_2,\Omega_1} : \mathcal{C}_{\text{harm}}(\Omega_2) & \rightarrow 
  \mathcal{C}_{\text{harm}}(\Omega_1).
 \end{align*}
 It follows immediately from the definition of CNT boundary values that if $h$ extends continuously to a boundary point $p \in \Gamma$ then the CNT boundary value exists and equals its CNT limit.  This motivates the definition of a transmission operator $ {\mathbf{O}}_{\Omega_1,\Omega_2}$  by restricting $\hat{\mathbf{O}}_{\Omega_1,\Omega_2}$ to ${\mathcal{D}_{\text{harm}(\Omega_1)} \cap \mathcal{C}_{\text{harm}(\Omega_1)}},$ which we shall define momentarily. Before doing that we gather our observations in the following theorem:\\

 \begin{theorem} \label{th:experimental_transmission}
  Let $\Gamma$ be a Jordan curve in $\sphere$, and let $\Omega_1$ and $\Omega_2$ be the connected components of the complement of $\Gamma$. If there is a dense set $\mathscr{L} \subseteq \mathcal{D}_{\mathrm{harm}}(\Omega_2)$, such that $\mathscr{L} \subset \mathcal{C}(\mathrm{cl}\, \Omega_2)$, and the continuous transmission is bounded with respect to Dirichlet energy on $\mathscr{L}$, then $\Gamma$ is a quasicircle.
 \end{theorem}
 Here, by dense set, we mean that for any $h \in \mathcal{D}_{\text{harm}}(\Omega_2)$, for all $\epsilon >0$ there is an element $u \in \mathscr{L}$ such that $D_{\Omega_2}(u-h) <\epsilon$.
 \begin{proof}
  Let $f:\disk^+ \rightarrow \Omega_1$ and $g:\disk^- \rightarrow \Omega_2$ be biholomorphisms.  Then by Carath\'eodory's theorem $\phi:= g^{-1} \circ f$ is a well-defined homeomorphism of $\mathbb{S}^1$.
  
  Given a $\mathscr{L}$ satisfying the hypotheses, observe that $\mathbf{C}_g \mathscr{L}$ is dense in $\mathcal{D}_{\text{harm}}(\Omega_2)$ by conformal invariance of Dirichlet energy, and by Carath\'eodory's theorem $\mathbf{C}_g \mathscr{L} \subset \mathcal{C}(\mathbb{S}^1)$.  Also $\mathbf{b}_{\disk^-,\mathbb{S}^1}\mathbf{C}_g \mathscr{L}$ is dense. Now for $h \in \mathbf{b}_{\disk^-,\mathbb{S}^1}\mathbf{C}_g \mathscr{L}$, define
  \[ \hat{\mathbf{C}}_\phi h = \mathbf{b}_{\disk^+,\mathbb{S}^1} \mathbf{C}_f \hat{\mathbf{O}}_{\Omega_2,\Omega_1} \mathbf{C}_{g^{-1}} \mathbf{e}_{\mathbb{S}^1,\disk^-} h   \]
  and note that 
  \[  \hat{\mathbf{C}}_\phi h = \mathbf{C}_\phi h.       \]
  By conformal invariance of the Dirichlet spaces and the hypothesis, this is a bounded operator on $\dot{H}^{1/2}(\mathbb{S}^1)$.  Thus applying Theorem \ref{th:experimental_NSV} we see that $\phi$ is a quasisymmetry which in turn yields that  $\Gamma$ is a quasicircle.  
 \end{proof}

 Theorem \ref{th:transmission} shows that if $\Gamma$ is a quasicircle and $\Omega_1$, $\Omega_2$ are the connected components of the complement, then 
 $\mathcal{H}(\Gamma,\Omega_1) = \mathcal{H}(\Gamma,\Omega_2)$.  We thus define 
 \[  \mathcal{H}(\Gamma) =  \mathcal{H}(\Gamma,\Omega_1) = \mathcal{H}(\Gamma,\Omega_2) \]
 in this special case.
 Now we are ready to define the transmission operators. 
 \begin{definition}
  We have well-defined maps 
 \begin{align*}
  \mathbf{O}_{\Omega_1,\Omega_2} & = \mathbf{e}_{\Gamma,\Omega_2} \mathbf{b}_{\Gamma,\Omega_1} : \mathcal{D}_{\text{harm}}(\Omega_1) \rightarrow \mathcal{D}_{\text{harm}}(\Omega_2) \\  
  \mathbf{O}_{\Omega_2,\Omega_1} & = \mathbf{e}_{\Gamma,\Omega_1} \mathbf{b}_{\Gamma,\Omega_2} : \mathcal{D}_{\text{harm}}(\Omega_2) \rightarrow \mathcal{D}_{\text{harm}}(\Omega_1) 
 \end{align*}
 which are bounded with respect to Dirichlet energy. 
 \end{definition}
 We will also use the simplified notation  
 \[   \mathbf{O}_{1,2} =  \mathbf{O}_{\Omega_1,\Omega_2}, \ \ \  \mathbf{O}_{2,1} =  \mathbf{O}_{\Omega_2,\Omega_1},  \]
 wherever it can be done without ambiguity.
 
 \begin{remark}
The symbol ``$\mathbf{O}$'' stands for old english ``oferferian'' meaning ``to transmit'', which could be rendered as ``overfare'' in modern english.
 \end{remark}
  The overfare operators are inverses of each other by definition:
 \begin{align*}
   \text{Id}_{\mathcal{D}_{\text{harm}}(\Omega_1)} & = \mathbf{O}_{1,2}\mathbf{O}_{2,1}   \\
   \text{Id}_{\mathcal{D}_{\text{harm}}(\Omega_2)} & = \mathbf{O}_{2,1}\mathbf{O}_{1,2}
 \end{align*}
 where $\text{Id}$ stands of course for the identity on the space indicated by the subscript.
 
 The overfare operators have a simple form in the case that $\Gamma = \mathbb{S}^1$: 
 \[  [\mathbf{O}_{\disk^+,\disk^-} h^+](z)  = h^+(1/\bar{z}), \ \ \ 
   [\mathbf{O}_{\disk^-,\disk^+} h^-](z)  = h^-(1/\bar{z})\]
 for $h^\pm \in \mathcal{D}_{\text{harm}}(\disk^\pm)$.
 
 We also observe that there is a transmission on Bergman space.  Namely, if $\Gamma$ is a quasicircle we define
 \[ \mathbf{O}'_{1,2} : \mathcal{A}_{\text{harm}}(\Omega_1) \rightarrow 
  \mathcal{A}_{\text{harm}}(\Omega_2) \]
  to be the unique operator satisfying 
  \begin{equation} \label{eq:Oprime_definition}
   \mathbf{O}'_{1,2} d = d \mathbf{O}_{1,2} 
  \end{equation}
  and similarly for $\mathbf{O}'_{2,1}$.  It is easily checked that this is well-defined using the fact that the transmission of a constant is (the same) constant.
 Similarly, for arbitrary Jordan curves $\Gamma$, continuous transmission induces the transmission on harmonic one-forms
 \[  \hat{\mathbf{O}}'_{1,2}:d \mathcal{C}_{\text{harm}}(\Omega_1) \rightarrow
   d \mathcal{C}_{\text{harm}}(\Omega_2).  \]
   
 The formulation of CNT boundary values and limits was entirely conformally invariant.  However, in the context of transmission, the existence of the overfare depends on the relative geometry of the domain $\Omega$ and the sphere.  That is, it depends on the regularity of the boundary.  It is remarkable that complete symmetry between the boundary value problems for the inside and outside domains occurs precisely for quasicircles. To the authors, this is an indication of the principle that Teichm\"uller theory can be seen as a scattering theory for harmonic one-forms \cite{Schippers_Staubach_monograph}.
 
 Finally, we record the following result. 
 \begin{corollary} Let $\Gamma$ be a Jordan curve in $\sphere$ and $\Omega_1$ and $\Omega_2$ be the connected components, and assume that $f:\disk^+ \rightarrow \Omega_1$  and $g:\disk^- \rightarrow \Omega_2$ are biholomorphisms. If there is a dense set $\mathscr{L} \subseteq \mathcal{D}_{\mathrm{harm}}(\Omega_2)$ such that $\mathscr{L} \subset \mathcal{C}(\mathrm{cl}\, \Omega_2)$, on which  $\mathbf{C}_f \hat{\mathbf{O}}_{2,1}$ is bounded, then $\Gamma$ is a quasicircle.  Conversely, if $\Gamma$ is a quasicircle, then $\mathbf{C}_f \mathbf{O}_{2,1}$ is bounded.
 \end{corollary}
 \begin{proof}
 If $\Gamma$ is a quasicircle, then $\mathbf{O}_{2,1}$ is bounded by Theorem \ref{th:transmission}, and $\mathbf{C}_f$ is an isometry. 
  
  Conversely, assume that there is a dense subset $\mathscr{L}$ with the stated properties.  Then $\mathbf{C}_{g}\mathscr{L}$ is dense in $\mathcal{D}_{\text{harm}}(\disk^-)$ since $\mathbf{C}_g$ preserves the Dirichlet energy, and furthermore $\mathbf{C}_g \mathscr{L}\subset\mathcal{C}(\mathrm{cl}\, \disk^-)$.  By assumption $\mathbf{O}_{\disk^+,\disk^-} \mathbf{C}_f \hat{\mathbf{O}}_{2,1} \mathbf{C}_{g^{-1}}$ is bounded on $\mathcal{D}_{\text{harm}}(\disk^-)$.  Hence $\mathbf{C}_{g^{-1} \circ f}$ is bounded on $\dot{H}^{1/2}(\mathbb{S}^1)$ and therefore by Theorem \ref{th:experimental_NSV} $\phi = g^{-1} \circ f$ is a quasisymmetry.  Thus $\Gamma$ is a quasicircle.  
 \end{proof}
\end{subsection}
\begin{subsection}{The bounce operator and density theorems} \ \label{se:bounce_anchor_density} 

\begin{definition}
 Let $\Gamma$ be a Jordan curve bounding a Jordan domain $\Omega$ in $\sphere$. A collar neighbourhood  of $\Gamma$ in $\Omega$ is a set of the form 
 \[  A_{p,r} = f(\mathbb{A}_r)   \]
 where $\mathbb{A}_r = \{z \in \mathbb{C} : r < |z| <1 \}$ and $f:\disk^+ \rightarrow \Omega$ is a biholomorphism such that $f(0)=p$.  
\end{definition}
In \cite{Schippers_Staubach_transmission,Schippers_Staubach_Plemelj} we used the term collar neighbourhood referred for more general domains, but this special case suffices for our purposes here.

We will show that functions in the Dirichlet space of a collar neighbourhood of $\Gamma$ have CNT boundary values.  To prove this, we need a lemma.
\begin{lemma}  \label{le:harmonic_function_decomposition}
 Let 
 \begin{align*}
     A & = \{ z \in \mathbb{C} : r < |z| <R \} \\
     B_1 & = \{z \in \mathbb{C} : |z| <R \}  \ \ \text{and} \\
     B_2 & = \{ z \in \mathbb{C} :  r <|z| \} \cup \{ \infty \}.
 \end{align*}  
 For any $h \in \mathcal{D}_{\mathrm{harm}}(A)$, there is a constant $c \in \mathbb{C}$ and functions $h_i \in \mathcal{D}_{\mathrm{harm}}(B_i)$ for $i=1,2$, such that 
 \[  h = h_1 + h_2 + c \log{(|z|/R)}  \]
 for all $z \in A$.  If $h$ is real, it is possible to choose $h_1$, $h_2$, and $c$ real.  
\end{lemma}
\begin{proof}
 We prove the claim for $h$ real; the general case follows by separating $h$ into real and imaginary parts.  
 
 Choose $s \in (r,R)$ and let $\gamma$ be the curve $|z|=s$ traced once counterclockwise.  Set 
 \[  c= \frac{1}{2 \pi} \int_{\gamma} \ast dh.     \]
 Since 
 \[  \int_{\gamma} \ast d \log{(|z|)} = 2 \pi    \]
 we then have that 
 \[  \int_{\gamma} \ast d ( h - c \log{(|z|)}) = 0.     \] 
 Set $H = h - c \log{|z|}$. Since $\ast dH$ is exact, $H$ has a single-valued harmonic anti-derivative $G$ in $A$, which is the harmonic conjugate of $H$.  Thus $F = H+ i G$ is a holomorphic function in $A$.   Now define   
 \[  F_1(z) = \lim_{s \nearrow R} \, \frac{1}{2\pi i}  \int_{\gamma} \frac{F(\zeta)}
 {{\zeta - z}}\, d\zeta, \ \ \ z \in B_1; \]
 and define $F_2$ by 
 \[  F_2(z) = \lim_{s \searrow r} \, \frac{1}{2\pi i}  \int_{\gamma} \frac{F(\zeta)}
 {{\zeta - z}}\, d\zeta, \ \ \ z \in B_2 \backslash \{\infty \}   \]
 and $F_2(\infty)=0$.  Observe that $F_1$ is holomorphic on $B_1$ and $F_2$ is holomorphic on $B_2$.  Furthermore for $z \in A$ clearly $F(z) = F_1(z) - F_2(z)$.  
 Now setting $h_1 = \text{Re}(F_1)$ and $h_2 = \text{Re}(F_2)$ we obtain the desired decomposition, where $h_1$, $h_2$, and $c$ are real.  It remains to show that $h_i \in \mathcal{D}_{\text{harm}}(B_i)$ for $i=1,2$.  
 
 To show that $h_1 \in \mathcal{D}_{\text{harm}}(B_1)$, it is enough to show that there is an annulus $A' = \{ z \in \mathbb{C} : r' < |z| < R \}$ for $r' \in (r,R)$ such that $h_1$ is in $\mathcal{D}_{\text{harm}}(A')$, since $h_1$ is holomorphic on an open neighbourhood of the closure of $|z| < r'$. 
 
 Given any such $r' \in (r,R)$, the closure of $A'$ is in $B_2$, and thus the restriction of $h_2$ to $A'$ is in $\mathcal{D}_{\text{harm}}(A')$. Furthermore, the restriction of $h$ to $A'$ is in $\mathcal{D}_{\text{harm}}(A')$, and a direct computation shows that $\log{(|z|)}$ is in $\mathcal{D}_{\text{harm}}(A)$, and in particular in $\mathcal{D}_{\text{harm}}(A')$.  Since $h_1 = h - h_2 - c \log{(|z|)}$, this proves that $h_1 \in \mathcal{D}_{\text{harm}}(A')$ and hence in $\mathcal{D}_{\text{harm}}(B_1)$. 
 
 The same argument shows that $h_2 \in \mathcal{D}_{\text{harm}}(B_2)$.
\end{proof}
\begin{remark}
 It is easy to adapt this argument to any doubly-connected domain bordered by non-intersecting Jordan curves, even on Riemann surfaces \cite{Schippers_Staubach_transmission}. It can be shown that the decomposition is unique, up to the additive  constant which can be transferred between $h_1$ and $h_2$.  
\end{remark}

\begin{theorem} \label{th:bounce_exists}
 Let $\Gamma$ be a Jordan curve bounding a Jordan domain $\Omega$ in $\sphere$. Let $A$ be a collar neighbourhood of $\Gamma$ in $\Omega$.  If $h \in \mathcal{D}_{\mathrm{harm}}(A)$ then $h$ has \emph{CNT} boundary values except possibly on a null set with respect to $\Omega$.  Furthermore, there is an $H \in \mathcal{D}_{\mathrm{harm}}(\Omega)$ whose \emph{CNT} boundary values agree with those of $h$ except possibly on a null set.
\end{theorem}
\begin{proof}
 By definition of collar neighbourhood, for some $p \in \Omega$ and $r \in (0,1)$, $A = A_{p,r} = f(\mathbb{A}_r)$ where $\mathbb{A}_r = \{z  \in \mathbb{C} : r < |z|<1 \}$.  
 By conformal invariance of Dirichlet spaces and CNT boundary values, it suffices to show this for $\Gamma = \mathbb{S}^1$, $\Omega = \disk^+$, and $A = \mathbb{A}_r$.  
 
 Let $h \in \mathcal{D}_{\text{harm}}(\mathbb{A}_r)$.  By Lemma \ref{le:harmonic_function_decomposition}, $h= h_1 + h_2 + c \log{|z|}$ for 
 some functions $h_i \in \mathcal{D}_{\text{harm}}(B_i)$, $i=1,2$ where $B_1 = \disk^+$ and $B_2 = \{ z: |z|>r \}\cup \{ \infty \}$.  Now $c\log{|z|}$ extends continuously to $0$ on $\mathbb{S}^1$, and thus the non-tangential boundary values exist and are zero everywhere on $\mathbb{S}^1$.  Since $h_1 \in \mathcal{D}(\disk^+)$, it has non-tangential boundary values except possibly on a null set by a direct application of Beurling's theorem \ref{th:boundaryvalues_exist_disk}.  Now $h_2$ is continuous on an annular neighbourhood of $\mathbb{S}^1$ and thus the non-tangential boundary values exist with respect to $\disk^+$ everywhere. Thus the non-tangential boundary values of $h$ exist except possibly on a null set.
 
 Furthermore, $u(z) = h_2(1/\bar{z}) \in \mathcal{D}_{\text{harm}}(\disk^+)$ is continuous on an open neighbourhood of $\disk^+$, and its non-tangential boundary values exist everywhere with respect to $\disk^+$ and equal those of $h_2$ with respect to $\disk^+$. Thus the function 
 $H= h_1 + u$ is in $\mathcal{D}_{\text{harm}}(\disk^+)$ has non-tangential boundary values equal to $h$ except possibly on a null set.
\end{proof}
\begin{remark}
 The proof actually shows a slightly stronger statement: there is an $H \in \mathcal{D}_{\mathrm{harm}}(\Omega)$ whose CNT boundary values exist and equal those of $h$, precisely where those of $h$ exist.
\end{remark}

\begin{definition}
 Since the function $H \in \mathcal{D}_{\text{harm}}(\Omega)$ is uniquely determined by its CNT boundary values on $\Gamma$, Theorem \ref{th:bounce_exists} induces a well-defined operator
\[  \mathbf{G}_{A,\Omega}:\mathcal{D}_{\text{harm}}(A) \rightarrow 
 \mathcal{D}_{\text{harm}}(\Omega)  \]
for any collar neighbourhood $A$ of the boundary $\Gamma$ of a Jordan domain. We call this the ``bounce'' operator.
\end{definition}

It follows immediately from the conformal invariance of the Dirichlet space and CNT limits that the bounce operator is conformally invariant.  That is, if $f:\Omega' \rightarrow \Omega$ is a biholomorphism and $A'$ is the domain such that $f(A')= A$, then 
\begin{equation} \label{eq:bounce_conformally_invariant}
 \mathbf{G}_{A',\Omega'} \left( h \circ f \right) = \left( \mathbf{G}_{A,\Omega}\, h \right) \circ f.
\end{equation}
{We shall also need a result about the agreement of Sobolev and Osborn spaces.
\begin{theorem}\label{th:oseborn-sobolev}
 Given a function $f\in H^{1/2}(\mathbb{S}^1)$ there exits a unique harmonic function $F\in \mathcal{D}_{\mathrm{harm}}(\disk)$ whose \emph{CNT} boundary values agree almost everywhere with values of $f$ on $\mathbb{S}^1$.  
\end{theorem}
\begin{proof}
By the existence and uniqueness of the solution to the Dirichlet problem (see e.g. Proposition 4.5 on page 334 in \cite{Taylor}), $f$ has a unique harmonic extension $F\in H^1(\disk)$, and the CNT boundary values of $F$ are equal to $f$ almost everywhere.
\end{proof}
Using this we can prove an energy inequality for the bounce operator.}
\begin{theorem} \label{th:bounce_bounded} Let $\Omega$ be a Jordan domain in $\sphere$ bounded by a Jordan curve $\Gamma$.  
 For any collar neighbourhood ${A}$ of $\Gamma$ in $\Omega$, $\mathbf{G}_{A,\Omega}$ is bounded with respect to the Dirichlet energy.  That is, there is a constant $C$ such that 
 \[    D_{\Omega}(\mathbf{G}_{A,\Omega}\, h) \leq C D_{{{A}}}(h) \]
 for all $h \in \mathcal{D}_{\mathrm{harm}}({A})$.
\end{theorem}
\begin{proof}
 By conformal invariance of the Dirichlet semi-norm and CNT limits, it suffices to 
 prove this for ${A} = \mathbb{A}_r = \{ z: r< |z| < 1\}$ and $\Omega = \disk^+$. 
  
  Let $h \in \mathcal{D}_{\text{harm}}({A})$. Then by Proposition 1.25.2 in \cite{Medkova}, $h$ is in $H^1({A})$.
  By Theorem \ref{th:oseborn-sobolev}, $\mathbf{G}_{A,\disk^+} h$ is the unique Sobolev extension of the Sobolev trace of $h$ in $H^{1/2}(\mathbb{S}^1)$. Furthermore by the result on the unique Sobolev extension, see e.g. Proposition 4.5 on page 334 in \cite{Taylor} and the fact that $\mathbb{S}^1 \subsetneq \partial {A}$, yields that
  \[  \| \left. h \right|_{\Gamma} \|_{H^{1/2}(\mathbb{S}^1)} \leq  \| \left. h \right|_{\partial \Omega} \|_{H^{1/2}(\partial {A})}  \leq C_1 \| h \|_{H^1({A})}. \]
  
  Also, by the existence of the unique solution to the Dirichlet problem with boundary data in Sobolev spaces (see e.g. Proposition 1.7 on page 360 in \cite{Taylor}), the harmonic Sobolev extension $H$ of $\left. h \right|_{\mathbb{S}^1}$ satisfies 
  \[  \| H \|_{{H}^1(\disk^+)} \leq C_2 \| \left. h \right|_{\mathbb{S}^1} \|_{H^{1/2}(\mathbb{S}^1)} . \]
  This together with the estimate for $\Vert h|_\Gamma \Vert_{H^{1/2}(\mathbb{S}^1)}$ above yields that 
  \begin{equation}\label{Sobolev estimate for bounce}
  \Vert  \mathbf{G}_{A,\disk^+}h \Vert _{H^1(\disk^+)} \leq  \Vert h\Vert _{H^1({A})}.
  \end{equation}
  Now if
  one applies \eqref{Sobolev estimate for bounce} to the harmonic function $h-h_{{A}}$ where $h_{{A}}$ is the average of $h$ given by $\frac{1}{|{A}|} \int_{{A}} h$, then one has that
$$\Vert \mathbf{G}_{A,\disk^+}h-\mathbf{G}_{A,\disk^+}h_{{A}}\Vert_{H^1(\disk^+)} \leq C \Vert h-h_{{A}}\Vert_{H^1({A})}. $$
Moreover we know that
$$D_{\disk^+}(\mathbf{G}_{A,\disk^+}h)^{1/2}= D_{\disk^+}(\mathbf{G}_{A,\disk^+}h-\mathbf{G}_{A,\disk^+}h_{{A}})^{1/2} \leq \Vert \mathbf{G}_{A,\disk^+}h-\mathbf{G}_{A,\disk^+}h_{{A}}\Vert_{H^1(\disk^+)}$$ and that $$\Vert h-h_{{A}}\Vert_{H^1({A})}= D_{{A}}(h-h_{{A}})^{1/2}+ \Vert h-h_{{A}}\Vert_{L^2({A})}\leq D_{{A}}(h)^{1/2}+ D_{{A}}(h)^{1/2}= 2D_{{A}}(h)^{1/2},$$ where the inequality $ \Vert h-h_{{A}}\Vert_{L^2(\Omega)}\leq C D_{{A}}(h)^{1/2}$ is the well-known Poincar\'e-Wirtinger inequality. Thus
$D_{\disk^+}(\mathbf{G}_{A,\disk^+}h )\leq C D_{{A}}(h),$ as desired.
  
\end{proof}

\begin{theorem}  \label{th:bounce_of_holomorphic_dense}
 Let $\Omega$ be a Jordan domain in $\sphere$ bounded by $\Gamma$ and let $A$ be a collar neighbourhood of $\Gamma$ in $\Omega$.  The set $\mathbf{G}_{A,\Omega}(\mathcal{D}(A))$ is dense in $\mathcal{D}_{\mathrm{harm}}(\Omega)$ with respect to the Dirichlet semi-norm.
\end{theorem}
\begin{proof}
 By conformal invariance of the Dirichlet semi-norm and (\ref{eq:bounce_conformally_invariant}), we may assume that $A = \mathbb{A}_r$ and $\Omega = \disk^+$ as above.  
 
 First, observe that the polynomials $\mathbb{C}[z,z^{-1}]$ are contained in $\mathcal{D}(\mathbb{A}_r)$.  But for any integer $n>0$
 \[  \mathbf{G}_{\mathbb{A}_r,\disk^+} z^n = z^n \ \ \text{and}  \ \ 
   \mathbf{G}_{\mathbb{A}_r,\disk^+} z^{-n} = \bar{z}^n   \]
   so $\mathbf{G}_{\mathbb{A}_r,\disk^+} \mathbb{C}[z,z^{-1}]= \mathbb{C}[z,\bar{z}]$.  Since $\mathbb{C}[z,\bar{z}]$ is a dense subset of $\mathcal{D}_{\text{harm}}(\disk^+)$ this proves the claim.
\end{proof}

\begin{theorem} \label{th:density_polynomials_doubly_connected}
 Let $A$ be any domain in $\sphere$ bounded by two non-intersecting Jordan curves, such that $0$ and $\infty$ are in distinct components of the complement of the closure of $A$.  Then Laurent polynomials $\mathbb{C}[z,z^{-1}]$ are dense in $\mathcal{D}(A)$.
\end{theorem}
\begin{proof}
 Without loss of generality assume that the component of the complement of $\Gamma_2$ containing $A$ also contains $\infty$, and let $B_2$ denote this component.  Let $B_1$ then denote the component of the complement of $\Gamma_1$  containing $A$; it must also contain $0$. We have that $A = B_1 \cap B_2$.   
 
 Now let $f_i:\disk^+ \rightarrow B_i$ be biholomorphisms for $i=1,2$.  Let $\gamma^r_i = f_i(|z|=r)$ for $r \in (0,1)$, endowed with positive orientations with respect to $0$.  For any $h \in \mathcal{D}(A)$, setting 
 \[  h_i(z) = \lim_{r \nearrow 1} \frac{1}{2\pi i} \int_{\gamma^r_i} 
 \frac{h(\zeta)}{{\zeta - z}} \,d\zeta \ \ \ z \in B_i,  \ i=1,2   \]
 we have that $h_i$ are holomorphic on $B_i$, and $h=h_1-h_2$.   
 
 We will show that $h_i$ are in $\mathcal{D}(B_i)$ for $i=1,2$.  Let $C_1$ denote the open domain in $B_1$ bounded by $\Gamma_1$ and $f_1(\gamma^s_2)$ for $s$ chosen close enough to $1$ that it is entirely in $A$.  This can be done, because the function $z \mapsto |f^{-1}_1(z)|$ is continuous on $B_1$, and is strictly less than one on $B_1$.  This function has a maximum $R <1$ on $\Gamma_2$ since $\Gamma_2$ is compact, so we can choose $s \in (R,1)$. To show that $h_1 \in \mathcal{D}(B_1)$ it suffices to show that $h_1 \in \mathcal{D}(C_1)$, since $h_1$ is holomorphic on an open neighbourhood of  $f( |z|\leq s)$.  Now $h \in \mathcal{D}(C_1)$ since $C_1 \subseteq A$, and $h_2 \in \mathcal{D}(C_1)$, since the closure of $C_1$ is contained in $B_2$.  Since $h_1 = h + h_2$, this proves the claim.  A similar argument shows that $h_2 \in \mathcal{D}(B_2)$. 
 
 Now $B_1$ is a Jordan domain and hence a Carath\'eodory domain, so polynomials $\mathbb{C}[z]$ are dense in $\mathcal{D}(B_1)$ \cite[v.3, Section 15]{Markushevich}. Similarly $\mathbb{C}[z]$ is dense in $\mathcal{D}(1/B_2)$, so $\mathbb{C}[1/z]$ is dense in $\mathcal{D}(B_2)$.  So given any $\epsilon >0$ there exists $p_1 \in \mathbb{C}[z]$ and $p_1 \in \mathbb{C}[1/z]$ such that 
 \[  \| h_i - p_i \|_{\mathcal{D}(A)} \leq \| h_i - p_i \|_{\mathcal{D}(B_i)} < \epsilon/2. \]
 Thus since $h= h_1 - h_2$ we see that 
 \[ \| h - p_1 + p_2 \|_{\mathcal{D}(A)} < \epsilon.  \]
 This proves the claim.
 \end{proof}

\begin{corollary}  \label{co:density_smaller_collar} Let $\Gamma$ be a Jordan curve in $\sphere$ and let $\Omega_1$ 
 and $\Omega_2$ be the connected components of the complement.  Let $A_1$ and $A_2$ be collar neighbourhoods of $\Gamma$ in $\Omega_1$ and $\Omega_2$ respectively, and let $U= A_1 \cup A_2 \cup \Gamma$. Let $\mathbf{R}_i:\mathcal{D}(U) \rightarrow \mathcal{D}(A_i)$ denote restriction from $U$ to $A_i$ for $i=1,2$.  Then 
 $\mathbf{R}_i (\mathcal{D}(U))$ is dense in $\mathcal{D}(A_i)$ for $i=1,2$.
\end{corollary}
\begin{proof}
 Observe that $U$ is open, so the statement of the theorem makes sense.

 Now $A_1$ and $A_2$ are each bounded by two non-intersecting Jordan curves in $\sphere$.
 By applying a M\"obius transformation and conformal invariance of the Dirichlet spaces and Dirichlet semi-norm, we can assume that $\infty$ and $0$ are each contained in the interior of one of the connected components of the complement of $U$, and not both in the same one.  In that case, the same holds for $A_1$ and $A_2$.  Thus $\mathbb{C}[z,z^{-1}]$ is dense in $\mathcal{D}(A_1)$ and $\mathcal{D}(A_2)$ by Theorem \ref{th:density_polynomials_doubly_connected}.  Since $\mathbb{C}[z,z^{-1}] \subseteq \mathcal{D}(U)$, the theorem is proven.
\end{proof}

 It is an indispensable fact that the limiting integral of harmonic functions against $L^2$ one-forms is unaffected by application of the bounce operator.  
 
 \begin{lemma}[Anchor Lemma]  \label{le:anchor_lemma}
  Let $\Gamma$ be a Jordan curve in $\sphere$ bounding a Jordan domain $\Omega$.  Let $A$ be a collar neighbouhood of $\Gamma$ in $\Omega$ and let $\Gamma_\epsilon = f(|z|=e^{-\epsilon})$ for a biholomorphism $f:\disk^+ \rightarrow \Omega$ and $\epsilon >0$.  For any $h \in \mathcal{D}_{\mathrm{harm}}(A)$ and $\alpha \in\mathcal{A}(A)$ 
  \begin{equation} \label{eq:anchor_lemma_main}
     \lim_{\epsilon \searrow 0} \int_{\Gamma_\epsilon} \alpha(w) h(w) = \lim_{\epsilon \searrow 0} \int_{\Gamma_\epsilon} \alpha(w) \mathbf{G}_{A,\Omega} h (w). 
  \end{equation}
  In particular, if $h$
  has \emph{CNT} boundary values equal to zero except possibly on a null set, then for any $\alpha \in\mathcal{A}(A)$   
  \[   \lim_{\epsilon \searrow 0} \int_{\Gamma_\epsilon} \alpha(w) h(w) = 0.      \]
 \end{lemma}
 \begin{proof}
  We assume that $\Gamma_\epsilon$ are positively oriented with respect to $0$.  
  The fact that the left integral in (\ref{eq:anchor_lemma_main}) is finite follows from the fact that $\alpha$ and $dh$ are $L^2$ on $A$, since fixing $\epsilon_0$ such that $\Gamma_{\epsilon_0}$ is in $A$, we have by Stokes' theorem that
  \[  \lim_{\epsilon \searrow 0} \int_{\Gamma_\epsilon} \alpha(w) h(w)  = 
       \int_{\Gamma_{\epsilon_0}} \alpha(w) h(w) + \iint_{A'} \alpha \wedge dh   \]
  where $A' \subset A$ is the region bounded by $\Gamma_{\epsilon_0}$ and $\Gamma$.
  
  Setting $\tilde{\alpha}(w) = \alpha(f(w)) f'(w)$ and $\tilde{h}(w) = h(f(w))$, and denoting the circle $|z|=e^{-\epsilon}$ traced counterclockwise by $C_\epsilon$, we have
  \[    \int_{\Gamma_\epsilon} \alpha(w) h(w) = \int_{C_\epsilon} \tilde{\alpha}(w) \tilde{h}(w)  \]
  so it suffices to prove the claim for $A = \mathbb{A} = \{ z: e^{-\epsilon_0}<|z|<1 \}$, $\Omega = \disk^+$, and $\Gamma_\epsilon = C_\epsilon$.

  We first show (\ref{eq:anchor_lemma_main}) for $\alpha(w) = w^n dw$ for some integer $n$.  
  By Lemma \ref{le:harmonic_function_decomposition} we can write $h= h_1 + h_2 + c \log{|z|}$ where $h_1 \in \mathcal{D}_{\text{harm}}(\disk^+)$ and $h_2 \in \mathcal{D}_{\text{harm}}(B)$ where $B = \{ z: |z| > e^{-\epsilon_0} \} \cup \{\infty \}$.  Now $\alpha$ and $h_2$ extend continuously to $\mathbb{S}^1$; thus so does $\mathbf{G}_{\mathbb{A},\disk^+} h_2$ and so trivially 
  \[   \lim_{\epsilon \searrow 0} \int_{C_\epsilon} \alpha(w) h_2(w) = \lim_{\epsilon \searrow 0} \int_{C_\epsilon} \alpha(w) \mathbf{G}_{\mathbb{A},\disk^+} h_2(w).  \]
  Similarly 
  \[   \lim_{\epsilon \searrow 0} \int_{C_\epsilon} \alpha(w) \log{|w|} = \lim_{\epsilon \searrow 0} \int_{C_\epsilon} \alpha(w) \mathbf{G}_{\mathbb{A},\disk^+} \log{|w|};  \]
  in fact, both sides are zero.  Finally, since $\mathbf{G}_{\mathbb{A},\disk^+} h_1= h_1$, the claim follows.  
  
  Thus the claim holds for any $\alpha(w) = p(w) \,dw$ for $p(w) \in \mathbb{C}[z,1/z]$.  Now the set of such $\alpha$ is dense in $\mathcal{A}(\mathbb{A})$. This follows from the density of $\mathbb{C}[z,1/z]$ in $\mathcal{D}(\mathbb{A})$ (which is a special case of Theorem \ref{th:density_polynomials_doubly_connected}), and the fact  that for some constant $k$,
  $\alpha - k/{z}$ is exact.  So the proof of the claim will be complete if it can be shown that for $h$ fixed,
  \[  \alpha \mapsto \lim_{\epsilon \searrow 0} \int_{C_\epsilon} \alpha(w) h(w)  \]
  is a continuous functional on $\mathcal{A}(\mathbb{A})$.  
  
  With $\epsilon_0$ and $A'$ the region bounded by $\Gamma_{\epsilon_0}$ and $\mathbb{S}^1$, let $M = \sup_{w \in \Gamma_{\epsilon_0}} |h(w)|$.  Since $\Gamma_{\epsilon_0}$ is a compact subset of $\mathbb{A}$, by a standard result for Bergman spaces there is a constant $C$ independent of $\alpha(w) = a(w) dw$ such that 
  \[ \sup_{w \in \Gamma_{\epsilon_0}}  |a(w)| \leq C \| \alpha \|_{A(\mathbb{A})}.  \]
  Therefore Stokes' theorem and Cauchy-Schwarz's inequality yield  that
  \begin{align*}
   \left| \lim_{\epsilon \searrow 0} \int_{\Gamma_\epsilon} \alpha(w) h(w) \right|& = 
      \left|  \int_{\Gamma_{\epsilon_0}} \alpha(w) h(w) + \iint_{A'} \alpha \wedge dh \right|  \\
      & \leq 2 \pi e^{-\epsilon_0} M \sup_{w \in \Gamma_{\epsilon_0}} |\alpha(w) | +  \| \alpha \|_{A(A')}\| h \|_{\mathcal{D}_{\text{harm}}(A')}  \\
      & \leq \left( 2 \pi e^{-\epsilon_0} M + \| h \|_{\mathcal{D}_{\text{harm}}(A')} \right) \| \alpha \|_{A(\mathbb{A})}
  \end{align*}
  which completes the proof of (\ref{eq:anchor_lemma_main}).  
  
  The proof of the second claim follows immediately from the observation that if $h$ has CNT boundary values zero except possibly on a null set, then $\mathbf{G}_{{A},\disk^+} h = 0$. 
 \end{proof} 
\end{subsection}
\end{section}
\begin{section}{Schiffer and Cauchy operator} \label{se:Schiffer_Cauchy}
\begin{subsection}{Schiffer operators} \label{se:Schiffer}
  
  We will define certain operators on the Bergman space of anti-holomorphic one-forms which we call ``Schiffer operators''.  We require an identity to facilitate the definition. 
  
    Given a Jordan domain $\Omega \subset \sphere$, let $g_\Omega(z,w)$ denote Green's function of $\Omega$, that is, the harmonic function in $z$ on $\Omega \backslash \{w\}$ such that $g_\Omega(z,w) + \log{|z-w|}$ is harmonic near $w$ and whose limit as $z \rightarrow z_0$ is zero for any point $z_0$ on the boundary of $\Omega$.  
  Schiffer considered the following kernel function
  \[ L_\Omega(z,w) \, dz \, dw = \frac{1}{\pi i}\frac{\partial^2 g_\Omega}{\partial z \partial w}(z,w) \,dz \,dw.   \]  
  Note that $L_\Omega$ is a meromorphic function in $z$ on $\Omega$ with a pole of order two at $z=w$ and no other poles.   In fact, it is symmetric, so it is also holomorphic in $w$ except for a pole at $w=z$.   
  \begin{theorem} \label{th:Schiffer_identity}
   Let $\Gamma$ be a Jordan curve, and let $\Omega$ be one of the connected components of the complement of $\Gamma$ in $\sphere$.  Let $g_\Omega(z,w)$ denote Green's function of $\Omega$.  Then for any one-form $\bar{\alpha} = \overline{h(z)}\, d\bar{z} \in {\overline{\mathcal{A}(\Omega)}}$
   
   \begin{equation} \label{eq:Schiffer_operator_domain}
     \Big(\iint_{\Omega}  L_\Omega(z,w) \, \overline{h(w)}\, {d\bar{w} \wedge dw }\Big) \cdot dz =0
   \end{equation}
   as a principal value integral. 
  \end{theorem}

  \begin{proof}
  Let $f:\disk^+ \rightarrow \Omega$ be a biholomorphism, chosen so that $f(0)=z$ and $\eta$ be such that $f(\eta)=w$. Let $\Gamma_\epsilon$ be the image of the circle with center at the origin and radius $e^{-\epsilon}$ under the biholomorphic map $f$, with positive orientation with respect to $w$. By Stokes' theorem, if we denote by $C_r$ the  circle of radius $r$ centred at $w$ with winding number one with respect to $w$, then 
   \begin{align*} \iint_\Omega L_\Omega(z,w) \overline{h(w)}\, {d\bar{w} \wedge dw} \cdot dz
      & = \lim_{r \searrow 0} \iint_{\Omega \backslash B(w;r)} L_\Omega(z,w) \overline{h(w)}\, {d\bar{w} \wedge dw} \cdot dz \\
      &  = \lim_{\epsilon \searrow 0} \frac{1}{\pi i} \int_{\Gamma_\epsilon} 
       \frac{\partial g_\Omega}{\partial z}(z,w) \overline{h(w)}\, d\bar{w}\, dz \\
       & \ \ \ \ \ \ \ -
         \lim_{r \searrow 0} \frac{1}{\pi i} \int_{C_r} 
       \frac{\partial g_\Omega}{\partial z}(z,w) \overline{h(w)}\, d\bar{w}  \,dz .  
   \end{align*}
   Note that all integrals take place over the $w$ variable while $z$ is fixed. To say that the output of the integral on the left hand side is zero as a form, is equivalent to demanding that for fixed $z$ the coefficient of $dz$ of the output is zero.  Therefore it is enough to show that 
   \begin{equation} \label{eq:Schiffer_identity_temp1}
     \lim_{\epsilon \searrow 0} \frac{1}{\pi i} \int_{\Gamma_\epsilon} 
       \frac{\partial g_\Omega}{\partial z}(z,w) \overline{h(w)}\, d\bar{w}  =0,
   \end{equation}
   and 
   \begin{equation} \label{eq:Schiffer_identity_temp2}
    \lim_{r \searrow 0} \frac{1}{\pi i} \int_{C_r} 
       \frac{\partial g_\Omega}{\partial z}(z,w) \overline{h(w)}\, d\bar{w} =0.    
   \end{equation}
   
   Now Green's function of the disk is given by 
   \[  g_{\disk^+}(\zeta,\eta) = -\log{\left| \frac{\zeta-\eta}{1-\bar{\eta}\zeta} \right|}    \]
   so by conformal invariance of Green's function $g_{\Omega}(f(\zeta),f(\eta)) = 
   g_{\disk^+}(\zeta,\eta)$ we have that
   \begin{equation} \label{eq:Schiffer_identity_temp3}
    \frac{\partial g_\Omega}{\partial z}(z,f(\eta)) = \frac{1}{f'(0)}  
    \left. \frac{\partial}{\partial \zeta} \right|_{\zeta =0} g_{\disk^+}(\zeta,\eta)
     = \frac{1}{2 f'(0)} \left( \frac{1}{\eta} - \bar{\eta} \right). 
   \end{equation}
   Now let $\mathbb{A}_r = \{z: r< |z| <1\}$ for any $r \in (0,1)$ and set $A = f(\mathbb{A}_r)$.  One can see explicitly from (\ref{eq:Schiffer_identity_temp3}) that 
   for fixed $z$, the function 
   \[  K(\eta) =  \frac{\partial g_\Omega}{\partial z}(z,f(\eta))      \]
   is in $\mathcal{D}_{\text{harm}}(\mathbb{A}_r)$, so by conformal invariance of the Dirichlet space 
   \[  k(w) = K(f^{-1}(w)) = \frac{\partial g_\Omega}{\partial z}(z,w)  \]
   is in $\mathcal{D}_{\text{harm}}(A)$. Thus we can apply the anchor lemma Lemma \ref{le:anchor_lemma} to $\bar{k}$ and $\alpha(w) = h(w) dw$ to conclude that the integral in (\ref{eq:Schiffer_identity_temp1}) is zero.  On the other hand, for $w$ in an open neighbourhood of $z$, by (\ref{eq:Schiffer_identity_temp3}) (or directly from the definition of Green's function) we can write 
   \[   \frac{\partial g_\Omega}{\partial z}(z,w) = \frac{1}{2(w-z)} + H(w)    \]
   where $H(w)$ is harmonic in $w$.  Inserting this into the left side of  (\ref{eq:Schiffer_identity_temp2}) we obtain that the integral is indeed zero.  
  \end{proof}
  
  We may now define the Schiffer operator.
 \begin{definition}\label{defn:Schiffer operators}
  Let $\Gamma$ be a Jordan curve in $\sphere$, and let $\Omega_1$ and $\Omega_2$ denote the connected components of the complement of $\Gamma$.  For $\overline{h(w)} d\bar{w} \in{\overline{\mathcal{A}(\Omega_1)}}$ we define for $j=1,2$
  \begin{equation} \label{eq:Schiffer_operator_definition}
    \mathbf{T}_{\Omega_1,\Omega_j} \overline{h(w)} d\bar{w}  = \frac{1}{\pi} \iint_{\Omega_1} \frac{\overline{h(w)}}{(w-z)^2} \frac{d\bar{w} \wedge dw}{2i} \cdot dz \ \ \ z \in \Omega_j.    
  \end{equation}
  Note that the output is a one-form on $\Omega_j$.  In the case that $j=1$, we interpret (\ref{eq:Schiffer_operator_definition}) as a principal value integral.  
  We will see that this is in general a bounded map into $A(\Omega_j)$, 
  and in that role we refer to $\mathbf{T}_{\Omega_1,\Omega_j}$ as Schiffer operators.
 \end{definition}

  First we establish the existence of this integral.  Assume for the moment that 
  $\Omega_1$ is bounded, that is, $\infty \in \Omega_2$.  For fixed $z \in \Omega_j$, the integrand $1/(w-z)^2$ is obviously in $L^2(\Omega_1)$, so this is immediate. 
  If $z \in \Omega_1$, for a biholomorphism $f:\disk^+ \rightarrow \Omega_1$ let $\Gamma_\epsilon$ be the image of the curve $|z|=e^{-\epsilon}$ under $f$ with the same orientation, and let $C_r$ be the circle centred on $z$ traced counterclockwise.  Then 
  \begin{equation} \label{eq:Schiffer_integral_converge_temp}
   \frac{1}{\pi}\iint_{\Omega_1} \frac{\overline{h(w)}}{(w-z)^2} \frac{d\bar{w} \wedge dw}{2i} \cdot dz  = \lim_{\epsilon \searrow 0} \frac{1}{2\pi i}  \int_{\Gamma_\epsilon} \frac{\overline{h(w)}\, d \bar{w}}{(w-z)} \,dz - 
   \lim_{r \searrow 0} \frac{1}{2\pi i}  \int_{C_r} \frac{\overline{h(w)}\, d\bar{w}}{(w-z)} \,dz.
  \end{equation}
  Let $\mathbb{A}_s = \{ z: s < |z| <1 \}$ where $s$ is fixed so that $z$ is not in the closure of $B_s = f(\mathbb{A}_s)$.  
  The first limit exists, by the fact that $h(w)\, dw$ and $dw/(w-z)^2$ are in  $\mathcal{A}(B_s)$ and  
  \[  \lim_{\epsilon \searrow 0} \frac{1}{2\pi i}  \int_{\Gamma_\epsilon} \frac{\overline{h(w)}\, d \bar{w}}{(w-z)} \,dz  = 
     \frac{1}{2\pi i}  \int_{\Gamma_{-\log{s}}} \frac{\overline{h(w)}\, d \bar{w}}{(w-z)} \,dz  +  \frac{1}{\pi}\iint_{B_s} \frac{\overline{h(w)}}{(w-z)^2} \frac{d\bar{w} \wedge dw}{2i} \cdot dz.   \]
   The limit of the second term in (\ref{eq:Schiffer_integral_converge_temp}) can be 
   shown to be zero by an explicit computation.  
   
  Theorem \ref{th:Schiffer_identity} can now be applied to de-singularize the kernel function.  We have for $\alpha(w) = h(w)\, dw \in {{\mathcal{A}(\Omega_1)}}$ 
  \begin{equation} \label{eq:Schiffer_desingularized} 
    \mathbf{T}_{\Omega_1,{\Omega_1}} \bar{\alpha} (z) =  \iint_{\Omega_1} \left( \frac{1}{2 \pi i} \frac{1}{(w-z)^2} - L_{\Omega_1}(z,w) \right) \overline{h(w)} \,  {d\bar{w} \wedge dw}  \cdot dz      
  \end{equation}
  since this new term does not have an effect on the existence or value of the integral. 
  
  We deal with the general case that $\Omega_1$ might be unbounded by establishing the invariance of the integrals under M\"obius transformations, which is interesting on its own. To this end define the pull-back of $\bar{\alpha}$ under $w=M(z)$ by 
  \[  M^* \bar{\alpha} (z) = \overline{h(M(z))} \overline{M'(z)} \,d \bar{z} \] 
  and similarly define the pull-back of $\beta(w) = g(w) \,dw$ by
  \[  M^* \beta(z) = g(M(z)) M'(z)\, dz.  \]
  \begin{theorem} \label{th:Schiffer_Mobius_invariance}
   If $M:\sphere \rightarrow \sphere$ is a M\"obius transformation taking $\Omega_j$ bijectively to $\tilde{\Omega}_j$ for $j=1,2$, then for all $\bar{\alpha} =\overline{h(w)}\,d\bar{w} \in \overline{\mathcal{A}(\tilde{\Omega}_1)}$
   we have 
   \begin{equation} \label{eq:Schiffer_Mobius_invariant}
    \left[\mathbf{T}_{\Omega_1,\Omega_j} \, M^* \bar{\alpha} \right] = M^* \left[\mathbf{T}_{\tilde{\Omega}_1,\tilde{\Omega}_j} \, \bar{\alpha} \right].   
   \end{equation}
  \end{theorem}
  \begin{proof}
   Assume first that $j=2$. Setting $w=M(z)$, $\eta = M(\zeta)$,  $\overline{\alpha(\eta)} = \overline{h(\eta)}\, d\bar{\eta}$ and using the identity 
   \begin{equation} \label{eq:Mobius_identity}
     \frac{M'(\zeta) M'(z)}{(M(\zeta)-M(z))^2} = \frac{1}{(\zeta - z)^2},  
   \end{equation}
  which holds for arbitrary M\"obius transformations, yield that
  \begin{align*}
   M^* \left[\mathbf{T}_{\tilde{\Omega}_1,\tilde{\Omega}_2} \bar{\alpha} \right](z) & 
   = \frac{1}{\pi} \iint_{\tilde{\Omega}_1} \frac{\overline{h(\eta)}}{(\eta-M(z))^2}
   \frac{d\bar{\eta} \wedge d \eta}{2i} \cdot M'(z) dz \\
   & = \frac{1}{\pi} \iint_{{\Omega}_1} \frac{\overline{h(M(\zeta))}}{(M(\zeta)-M(z))^2}
   \overline{ M'(\zeta)} {M'(\zeta)} \frac{d\bar{\zeta} \wedge d \zeta}{2i} \cdot M'(z) dz \\
   & = \frac{1}{\pi} \iint_{{\Omega}_1} \frac{\overline{h(M(\zeta))} \overline{M'(\zeta)}}{(\zeta-z)^2}
   \frac{d\bar{\zeta} \wedge d \zeta}{2i} \cdot dz \\
   & = \left[ \mathbf{T}_{{\Omega}_1,{\Omega}_2} M^* \bar{\alpha} \right] (z).
  \end{align*}
  
  In the case that $j=1$, we use the identity 
  \[    L_{M(\Omega)}(M(\zeta),M(z)) M'(\zeta) M'(z) = L_\Omega(\zeta,z),   \]
  which follows from $g_{M(\Omega)}(M(\zeta),M(z))=g_\Omega(\zeta,z)$.  When combined with (\ref{eq:Mobius_identity}), the argument above may be repeated using the expression (\ref{eq:Schiffer_desingularized}). 
  \end{proof}
  Note that the de-singularization of the integral allowed the application of change of variables in the proof above.  As a consequence, we see that the M\"obius transformation preserves the original principal value integral. This can also be shown directly.
  \begin{remark}
   If one views the Schiffer operators as acting on a Bergman space of functions $\overline{h(z)}$ rather than on the $L^2$ space of one-forms $\overline{h(z)}\, d\bar{z}$, their M\"obius invariance is obscured. 
  \end{remark}
  
  As promised, Theorem \ref{th:Schiffer_Mobius_invariance} implies the existence of the integrals defining the Schiffer operator, since we may apply a M\"obius transformation to reduce the general case to the case that $\Omega_1$ is bounded, which we dealt with above. 
 
  \begin{remark}  
   If $z \in \Omega_j$, the meaning of this one-form at $z = \infty$ is obtained by 
   applying the change of coordinates $z= 1/\zeta$, $dz = -d\zeta/\zeta^2$ to express it in coordinates at $\infty$:
   \[ \frac{1}{\pi} \iint_{\Omega_1} \frac{\overline{h(w)}}{(1-w\zeta)^2} \frac{d\bar{w} \wedge dw}{2i} \cdot d\zeta \ \ \ \zeta \in 1/\Omega_j.  \]
   Alternatively one may transform both the input and output simultaneously using Theorem \ref{th:Schiffer_Mobius_invariance} with $M(z) = 1/z$.  
  \end{remark}
  
  Finally, we have the following. 
  \begin{theorem} \label{th:Schiffer_bounded}
   Let $\Gamma$ be a Jordan curve in $\sphere$, and let $\Omega_1$ and $\Omega_2$ be the components of the complement of $\Gamma$ in $\sphere$. The Schiffer operators $\mathbf{T}_{\Omega_1,\Omega_j}$ are bounded from $\overline{\mathcal{A}(\Omega_1)}$ to $\mathcal{A}(\Omega_j)$ for $j=1,2$.  
  \end{theorem}
    
  \begin{proof}
   By Theorem \ref{th:Schiffer_Mobius_invariance}, we may assume that $\infty \in \Omega_2$, so that $\Omega_1$ is bounded.
   The integrands in the definitions of $\mathbf{T}_{\Omega_1,\Omega_2}$ and $\mathbf{T}_{\Omega_1,\Omega_1}$ given in \eqref{eq:Schiffer_operator_definition} and \eqref{eq:Schiffer_desingularized} respecrively are  non-singular and holomorphic in $z$ for each $w\in \Omega_2$ and $w\in \Omega_1$ (in each case), and  
 {furthermore both integrals are locally bounded in $z$.
 Therefore the holomorphicity of $\mathbf{T}_{\Omega_1,\Omega_j}$ follows by moving the $\overline{\partial}$ inside \eqref{eq:Schiffer_operator_definition} and \eqref{eq:Schiffer_desingularized}, and using the holomorphicity of the integrands in each case.
 
The $L^p$-boundedness of these operators for $1<p<\infty$, considered as singular integral operators, is a consequence of the boundedness of singular integral operators of Calder\'on-Zygmund type, see e.g. \cite{Lehto} page 26.}  
  \end{proof}
  
  As in the case of the overfare operator $\mathbf{O}$, we will use the notation $\mathbf{T}_{j,k}$ in place of $\mathbf{T}_{\Omega_k,\Omega_k}$ wherever possible. 
  
\end{subsection}
\begin{subsection}{Cauchy operator} \ \label{se:Cauchy}
 As usual, consider a Jordan curve $\Gamma$ in $\sphere$.  For now we assume that $\infty$ is not in $\Gamma$.  Let $\Omega_1$ and $\Omega_2$ denote the components of the complement.  
 
 \begin{definition}\label{defn:Cauchy operator}
  For $h \in \mathcal{D}(\Omega_1)$ we will consider a kind of Cauchy integral obtained as follows.  Let $f:\disk \rightarrow \Omega_1$ be a biholomorphism. If we let $\Gamma_\epsilon$ be the image of the closed curve $|z|=e^{-\epsilon}$ under $f$, with the same orientation, and $q \notin \Gamma$, then we define  
 \begin{align}  \label{eq:Cauchy_operator_definition}
   \mathbf{J}^q_{\Omega_1} h(z) & = \frac{1}{2 \pi i} \lim_{\epsilon \searrow 0} 
      \int_{\Gamma_\epsilon} h(w) \left(\frac{1}{w-z} - \frac{1}{w-q} \right) \,dw \ \ \ z \in \sphere \backslash \Gamma    
        \\ 
      & = \frac{1}{2 \pi i} \lim_{\epsilon \searrow 0} 
      \int_{\Gamma_\epsilon} h(w) \frac{z-q}{(w-z)(w-q)} \,dw  \ \ \ \ \ \ \ z \in \sphere \backslash \Gamma. \nonumber
 \end{align}
 The term involving $q$ amounts to an arbitrary choice of normalization. In the case that $q = \infty$, this reduces to 
 \[
     \mathbf{J}^q_{\Omega_1} h(z)  = \frac{1}{2 \pi i} \lim_{\epsilon \searrow 0} 
      \int_{\Gamma_\epsilon} \frac{h(w)}{w-z} \,dw \ \ \ z \in \sphere \backslash \Gamma.
 \]
 \end{definition}

 This is almost a Cauchy integral, of course.  We will motivate the definition of $\mathbf{J}^q_{\Omega_1}$ after first establishing some of its properties. 

 First, we observe that $\mathbf{J}^q$ is M\"obius invariant in a certain sense. 
 The invariance follows from the identity
 \begin{equation} \label{eq:Cauchy_Mobius_invariance}
  \frac{M'(w)(M(z)-M(q))}{(M(w)-M(z))(M(w)-M(q))} = \frac{z-q}{(w-z)(w-q)}
 \end{equation}
 which holds for any M\"obius transformation $M$.  
 Observe that the usual normalization $q=\infty$ obscures the M\"obius invariance of the Cauchy kernel.  Using this identity, together with a change of variables and conformal invariance of the Dirichlet space, we obtain the following. 
 \begin{theorem} \label{th:J_Mobius_invariant}
  Let $\Gamma$ be a curve in $\sphere$ and $\Omega_1$, $\Omega_2$ be the connected components of the complement. Let $M$ be a M\"obius transformation. Then for any $h \in \mathcal{D}(M(\Omega_1))$, we have
  \[    \left[ \mathbf{J}^{M(q)}_{M(\Omega_1)} h \right] \circ M =  \mathbf{J}^q_{\Omega_1}(h \circ M).   \]
 \end{theorem}
 Observe that Theorem \ref{th:J_Mobius_invariant} extends the definition of the integral to the case that $\infty \in \Gamma$. 
 For the moment, this says only that if the limit exists on one side, then it exists on both, and the two sides are equal.  We will show that the limit exists whenever $h \in \mathcal{D}_{\text{harm}}(\Omega_1)$.  
 
 In the remainder of the section we will:  (1) provide identities relating $\mathbf{J}^q$ to the Schiffer operators; (2) show that the output is in $\mathcal{D}_{\text{harm}}(\Omega_1 \sqcup \Omega_2)$; and (3) show that for quasicircles, the limiting integral is in a certain sense independent of which side of $\Gamma$ you choose to take the limit in.
 
   Let $\partial$ and $\overline{\partial}$ denote the Wirtinger operators on the Riemann sphere.  
 \begin{theorem} \label{th:J_Schiffer_identities}
  Let $\Gamma$ be a Jordan curve separating $\sphere$ into connected components $\Omega_1$ and $\Omega_2$.  Assume that $q \notin \Gamma$.   Then 
  \begin{align}
      \partial \mathbf{J}^q_{\Omega_1} h (z) & = -\mathbf{T}_{\Omega_1,\Omega_2} \overline{\partial} h(z)  \hspace{3.2cm}  z \in \Omega_2 \\
      \partial \mathbf{J}^q_{\Omega_1} h (z) & = \partial h(z) - \mathbf{T}_{\Omega_1,\Omega_1} 
      \overline{\partial} h (z) \hspace{2cm}  z \in \Omega_1 \\
      \overline{\partial} \mathbf{J}^q_{\Omega_1} h(z) & = 0 \hspace{5cm}
      z \in \Omega_1 \cup \Omega_2.  
  \end{align}
 \end{theorem}
 \begin{proof}
  If $q \in \riem_2$, then the first claim follows by applying Stokes theorem and bringing $\partial$ under the integral sign, as does the third in the case that $z \in \Omega_2$. Denote 
  the circle $|z-q|=r$ traced counter-clockwise by $C_r$. 
  Using the fact that 
  \[  \lim_{r \searrow 0} \frac{1}{\pi}  \int_{C_r} \frac{\partial g_{\Omega_1}}{\partial w} (w;z) h(w) = h(q)  \]
  by Stokes' theorem we have that for $q \in \Omega_1$ and $z \in \Omega_2$,
  \begin{align} \label{eq:Schiffer_Cauchy_identities_temp1}
   \mathbf{J}^q_{\Omega_1} h(z) & = \lim_{\epsilon \searrow 0}  
   \int_{\Gamma_\epsilon} \left[ \frac{1}{2\pi i} \left( \frac{1}{w-z} - 
   \frac{1}{w-q} \right) - \frac{1}{\pi} \frac{\partial g_{\Omega_1}}{\partial w}(w,q)  \right] h(w) \,  dw - h(q) \nonumber \\
   & =  
   \iint_{\Omega_1} \left[ \frac{1}{2\pi i} \left( \frac{1}{w-z} - 
   \frac{1}{w-q} \right) - \frac{1}{\pi} \frac{\partial g_{\Omega_1}}{\partial w}(w,q)  \right] h(w) \, d\bar{w} \wedge dw - h(q) 
  \end{align}
  Noting that the integrand is non-singular, applying $\partial_z$ to both sides using Theorem \ref{th:Schiffer_identity}) proves the first claim.  Applying $\overline{\partial}_z$ proves the third claim in the case that $q \in \Omega_1$
  and $z \in \Omega_2$.  
  
  If $z \in \Omega_1$ and $q \in \Omega_2$, we have similarly that 
  \begin{equation} \label{eq:Schiffer_Cauchy_identities_temp2}
   \mathbf{J}^q_{\Omega_1} h(z) =  
   \iint_{\Omega_1} \left[ \frac{1}{2\pi i} \left( \frac{1}{w-z} - 
   \frac{1}{w-q} \right) + \frac{1}{\pi} \frac{\partial g_{\Omega_1}}{\partial w}(w,z)  \right] h(w) \, d\bar{w} \wedge dw + h(z).  
  \end{equation}
  Applying $\overline{\partial}_z$ completes the proof of the third claim, and applying $\partial_z$ using (\ref{eq:Schiffer_desingularized}) proves the second claim in the case that $q \in \Omega_2$.  To prove the second claim in the case that $q \in \Omega_1$, we add a further term to (\ref{eq:Schiffer_Cauchy_identities_temp2}) which removes the singularity at $q$ as in (\ref{eq:Schiffer_Cauchy_identities_temp1}) and apply $\partial_z$.
  
 \end{proof}
 For $j=1,2$ we denote
 \[  \mathbf{J}^q_{\Omega_1,\Omega_j} h  = \left. \mathbf{J}^q_{\Omega_1} h \right|_{\Omega_j}.     \]
 Then Theorem \ref{th:J_Schiffer_identities} immediately implies that these are bounded with respect to the Dirichlet energy.
 \begin{corollary}  \label{co:J_bounded} 
  Let $\Gamma$ be a Jordan curve in $\sphere$ and choose $q \notin \Gamma$.  Then 
  \[  \mathbf{J}^q_{\Omega_1}:\mathcal{D}_{\mathrm{harm}}(\Omega_1) \rightarrow \mathcal{D}_{\mathrm{harm}}(\Omega_1 \cup \Omega_2)_q   \]
  and 
  \[   \mathbf{J}^q_{\Omega_1,\Omega_j}:\overline{\mathcal{D}(\Omega_1)} \rightarrow \mathcal{D}(\Omega_j) \ \ \ j=1,2. 
     \]
  If $q \in \Omega_j$, then the image of $\mathbf{J}^q_{\Omega_1,\Omega_j}$ is $\mathcal{D}_q(\Omega_j)$.
  Furthermore, each of the operators above are bounded with respect to Dirichlet energy.  
 \end{corollary}
 \begin{proof}
  This follows immediately from Theorems \ref{th:Schiffer_bounded} and \ref{th:J_Schiffer_identities}.
 \end{proof}
 
 As in the case of the overfare and Schiffer operators, we will use the notation $\mathbf{J}^q_{k}$ in place of $\mathbf{J}^q_{\Omega_k}$ and $\mathbf{J}^q_{j,k}$ in place of $\mathbf{J}^q_{\Omega_j,\Omega_k}$ wherever possible.

 The operator $\mathbf{J}^q_{1}$ is motivated as follows. Setting aside the normalization at $q$, we would like to define the Cauchy integral 
 \[  \frac{1}{2\pi i} \int_{\Gamma} \frac{h(w)}{w-z} \,dw    \]
 of a function $h \in \mathcal{H}(\Gamma,\Omega_1)$, but there are two obstacles: the curve $\Gamma$ is not rectifiable, and functions in $h$ are not particularly regular.  This problem is solved by considering instead $\mathbf{J}^q_{\Omega_1} \mathbf{e}_{\Gamma,\Omega_1} h$.  
 
 The question immediately arises: if one considers instead $\mathbf{J}^q_{\Omega_2} \mathbf{e}_{\Gamma,\Omega_2} h$, is the result the same?  Of course, this  requires that $\mathcal{H}(\Gamma,\Omega_1) \subseteq \mathcal{H}(\Gamma,\Omega_2)$ at least, which we know is true for quasicircles.  In fact, it is indeed sufficient that $\Gamma$ is a quasicircle.
 \begin{theorem} \label{th:J_same_both_sides}
   Let $\Gamma$ be a quasicircle in $\sphere$, and let $\Omega_1$ and $\Omega_2$ be the connected components of the complement.  Fix $q \notin \Gamma$.  For any $h \in \mathcal{D}_{\mathrm{harm}}(\Omega_1)$    
   \[   \mathbf{J}^q_{1} h = - \mathbf{J}^q_{2} \mathbf{O}_{1,2} h.     \]
   The same result holds switching the roles of $\Omega_1$ and $\Omega_2$.  
 \end{theorem}
 \begin{proof}
  Let $B_1$ and $B_2$ be collar neighbourhoods of $\Gamma$ in $\Omega_1$ and $\Omega_2$ respectively.  Let $U = B_1 \cup B_2 \cup \Gamma$.  This is an open set bordered by two analytic curves. By Corollary \ref{co:density_smaller_collar}, the class $\mathbf{R}_1 \mathcal{D}(U)$ of elements of $\mathcal{D}(U)$ to $\mathcal{D}(U)$ is dense in $\mathcal{D}(B_1)$. Furthermore, by Theorem \ref{th:bounce_of_holomorphic_dense}, $\mathbf{G}_{A_1,\Omega_1} \mathcal{D}(B_1)$ is dense in $\mathcal{D}_{\text{harm}}(\Omega_1)$.  Thus since $\mathbf{G}_{B_1,\Omega_1}$ is bounded by Theorem \ref{th:bounce_bounded}, $\mathbf{G}_{B_1,\Omega_1} \mathbf{R}_1 \mathcal{D}(U)$ is dense in $\mathcal{D}_{\text{harm}}(\Omega_1)$.  By Theorem \ref{th:transmission} $\mathbf{O}_{1,2}$ is bounded, so it is enough to prove the theorem for such functions.  
  
  Let $h \in \mathcal{D}(U)$. Since $h$ extends continuously to $\Gamma$, the CNT boundary values of $\mathbf{R}_1 h$ and $\mathbf{R}_2 h$ with respect to $\Omega_1$ and $\Omega_2$ both equal the continuous extension and hence each other. Thus 
  \begin{equation}  \label{eq:equality_limiting_temp_one}
   \mathbf{O}_{1,2} \mathbf{G}_{B_1,\Omega_1} \mathbf{R}_1 h = 
   \mathbf{G}_{B_2,\Omega_2} \mathbf{R}_2 h.
  \end{equation} 
  
  Fix $z \in \sphere \backslash \Gamma$.  Since $B_i$ are collar domains, by definition there are biholomorphisms $f_i:\disk^+ \rightarrow \Omega_i$ so that $f_i(\mathbb{A}_{r_i}) = B_i$ for annuli $\mathbb{A}_{r_i} = \{ z : r_i <|z|<1 \}$ for $i=1,2$. Let $\Gamma^i_\epsilon$ denote the limiting curves $f_i(|z|=e^{-\epsilon})$ with orientations induced by $f_i$.  By Carath\'eodory's theorem, the maps $f_i$ extend homeomorphically to maps from $\mathbb{S}^1$ to $\Gamma$, so for any fixed $\epsilon$, the curves $\Gamma^i_{\epsilon}$ are each homotopic to $\Gamma$ and hence to each other.  Thus, since $z$ and $q$ are eventually not in the domain bounded by $\Gamma^1_\epsilon$ and $\Gamma^2_\epsilon$,
  \begin{equation}  \label{eq:equality_limiting_temp_two}
    \lim_{\epsilon \searrow 0} \frac{1}{2 \pi i} \int_{\Gamma^1_\epsilon} 
    h(w) \left( \frac{1}{w-z} - \frac{1}{w-q} \right) \,dw = 
    - \lim_{\epsilon \searrow 0} \frac{1}{2 \pi i} \int_{\Gamma^2_\epsilon} 
    h(w) \left( \frac{1}{w-z} - \frac{1}{w-q} \right) \,dw   
  \end{equation}
  where the negative sign arises from the change of orientation between the integrals.
  
  Finally, applying  the anchor lemma \ref{le:anchor_lemma} for fixed $z$ with
  \[  \alpha(w) = \left( \frac{1}{w-z} - \frac{1}{w-q} \right) \,dw,  \]
  we have that for $i=1,2$
  \begin{equation} \label{eq:equality_limiting_temp_three}
    \mathbf{J}^q_{i} \mathbf{G}_{B_i,\Omega_i} \mathbf{R}_i h (z) = \lim_{\epsilon \searrow 0} \frac{1}{2 \pi i} \int_{\Gamma^i_\epsilon} 
    h(w) \left( \frac{1}{w-z} - \frac{1}{w-q} \right) \,dw.  
  \end{equation}
  Here, we may have to shrink the domain $B_j$ so that neither $z$ nor $q$ are in the closure, to ensure that $\alpha \in L^2(B_j)$. This does not affect the validity of the argument, since given nested collar neighbourhoods $B_j' \subset B_j$, by definition 
  \[  \mathbf{G}_{B_j',\Omega_j} \left.  h \right|_{B_j'} = \mathbf{G}_{B_j,\Omega_j}
   \left. h \right|_{B_j}. \]
  Thus combining (\ref{eq:equality_limiting_temp_one}), (\ref{eq:equality_limiting_temp_two}), and (\ref{eq:equality_limiting_temp_three}) we have 
  \[   \mathbf{J}^q_{1} \mathbf{G}_{B_1,\Omega_1} \mathbf{R}_1 h = 
     - \mathbf{J}^q_{2} \mathbf{O}_{1,2} \mathbf{G}_{B_1,\Omega_1} \mathbf{R}_1 h \]
  which completes the proof.
 \end{proof}
 \begin{remark}
  The negative sign is an artifact of the change of orientation induced by the switch from the domain $\Omega_1$ to $\Omega_2$.  In previous publications \cite{Schippers_Staubach_JMAA,Schippers_Staubach_Plemelj} we chose the orientations in such a way that the sign did not change.
 \end{remark}

 Finally, we record the following obvious fact.
 \begin{theorem} \label{th:jump_on_holomorphic} 
  Let $\Gamma$ be a Jordan curve in $\sphere$ and let $\Omega_1$ and $\Omega_2$ be the components of the complement.  Fix $q \notin \Gamma$, and let $h \in \mathcal{D}(\Omega_j)$.  
  
  If $q \in \Omega_j$, then 
  \[  \mathbf{J}^q_{j} h (z) = \left\{ \begin{array}{lr} h(z)- h(q) & z \in \Omega_j 
  \\  -h(q) & z \notin \Omega_j \cup \Gamma  \end{array} \right.       \]
  whereas if $q \notin \Omega_j$ then 
   \[  \mathbf{J}^q_{j} h (z) = \left\{ \begin{array}{lr} h(z)  & z \in \Omega_j 
  \\   0  & z \notin \Omega_j \cup \Gamma.  \end{array} \right.       \]
 \end{theorem}
 \begin{proof}
  This follows from the ordinary Cauchy integral formula.  
 \end{proof}
 
\end{subsection} 
\begin{subsection}{The Schiffer isomorphisms and the Plemelj-Sokhotski isomorphisms}
 \label{se:jump_isomorphisms}
 In this section we show that the Schiffer operator $\mathbf{T}_{\Omega_1,\Omega_2}$ and the jump decomposition induced by the Cauchy operator $\mathbf{J}^q_{1}$ are isomorphisms precisely for quasicircles. 
 
 We refer to the isomorphism induced by the jump decomposition as the Plemelj-Sokhotski isomorphism.  
  The classical Plemelj-Sokhotski jump decomposition says the following, in the smooth case.  Let $\Gamma$ be a smooth Jordan curve $\Gamma$ separating $\mathbb{C}$ into components $\Omega_1$ and $\Omega_2$; assume that $\Omega_2$ is the unbounded component.  For a smooth function $u$ on $\Gamma$, define the functions
 \[   h_k(z) = \frac{1}{2 \pi i} \int_{\Gamma} \frac{h(\zeta)}{\zeta - z} d\zeta \ \ \ \  z \in \Omega_k, \ \ k =1,2. \]
 For any point $w \in \Gamma$, it is easily proven that
 \[  \lim_{z \rightarrow w} h_2(z) -  \lim_{z \rightarrow w} h_1(z) = u(w). \]
 In fact, this formula can be written in a stronger form involving the principal value integral of $u$ on the boundary; see Section \ref{se:Schiffer_Cauchy_notes}.   
 
 The map taking $u$ to $(h_1,h_2)$ is what we call the Plemelj-Sokhotski isomorphism.  Using the limiting integral in place of the Cauchy integral, we will prove that for $u \in \mathcal{H}(\Gamma)$, this is an isomorphism if and only if $\Gamma$ is a quasicircle. We will also show the closely related result of Napalkov and Yulmukhametov that $T_{1,2}$ is an isomorphism if and only if $\Gamma$ is a quasicircle.  
 
 To do this, we first require a lemma. 
 \begin{lemma} \label{le:pre-transmitted_jump}
  Let $\Gamma$ be a Jordan curve in $\sphere$ and let $\Omega_1$ and $\Omega_2$ be the connected components of the complement.  Let $B$  be a collar neighbourhood of $\Gamma$ in $\Omega_1$.  Assume that $q$ is not in the closure of $B$. If $H \in \mathcal{D}(B)$ then 
  $\mathbf{J}^q_{1,2} \mathbf{G}_{B,\Omega_1} H$ extends to a holomorphic function $H_2$ on $\Omega_2 \cup \Gamma \cup B$ which satisfies
  \[   H_2(z) = \mathbf{J}^q_{1,1} \mathbf{G}_{B,\Omega_1} H(z) - H(z).  
   \ \ \ z \in B. \]  
  
  Furthermore, $\mathbf{J}^q_{1,2} \mathbf{G}_{B,\Omega_1} H$ has a transmission in $\mathcal{D}_{\mathrm{harm}}(\Omega_1)$, given explicitly by 
  \begin{align*}
    \hat{\mathbf{O}}_{2,1}  \mathbf{J}^q_{1,2} \mathbf{G}_{B,\Omega_1} H & = \mathbf{G}_{B,\Omega_1} H_2 \\ 
    & = \mathbf{J}^q_{1,1} \mathbf{G}_{B,\Omega_1} H - \mathbf{G}_{B,\Omega_1} H.
  \end{align*} 
 \end{lemma} 
 Recall that $\hat{\mathbf{O}}_{2,1}$ is the solution of the
 Dirichlet problem on $\Omega_1$ with continuous boundary values $\left. H_2 \right|_{\Gamma}$.
 \begin{proof}
  By Theorem \ref{th:J_Mobius_invariant}, it is enough to prove this in the case that $q=\infty$ and $\Omega_2$ is the unbounded component of the complement of $\Gamma$.
  
  The first claim is just the ordinary Cauchy integral formula combined with the anchor lemma.   Let $f:\disk^+ \rightarrow \Omega_1$ be the biholomorphism such that $f(\mathbb{A}) = B$ for an annulus $\mathbb{A} = \{ z : r <|z|<1 \}$, and let $\Gamma_1^\epsilon$ be the corresponding images under $f$ of circles $|z|=e^{-\epsilon}$ as usual, with orientation induced by $f$.  Let $\gamma$ be the analytic curve which is the inner boundary of $B$; that is, the image of $|z|=r$ under $f$.  
  
  Define   
  \[   H_2(z) = \frac{1}{2\pi i} \int_{\gamma} \frac{H(w)}{w-z} \,dw  \]
  which is holomorphic in the open set $\Omega_2 \cup \Gamma \cup B$.   
  By the anchor lemma \ref{le:anchor_lemma} and the fact that $H$ is holomorphic, for all $z \in \Omega_2$ 
  \begin{equation} \label{eq:transmitted_jump_temp_one}
   \mathbf{J}^q_{1,2} \mathbf{G}_{B,\Omega_1} H(z) = \lim_{\epsilon \searrow 0} 
   \frac{1}{2 \pi i} \int_{\Gamma^1_\epsilon} \frac{H(w)}{w-z} \,dw =  H_2(z).
  \end{equation}
  By the ordinary Cauchy integral formula,  for all $z \in B$ 
  \[  H(z) = \lim_{\epsilon \searrow 0} \frac{1}{2 \pi i} \int_{\Gamma^1_\epsilon} 
     \frac{H(w)}{w-z} \,dw - H_2(z).      \]
  Applying the anchor lemma \ref{le:anchor_lemma} again we see
  \begin{equation} \label{eq:transmitted_jump_temp_two}
   H(z) = \mathbf{J}^q_{1,1} \mathbf{G}_{B,\Omega_1} H (z) - H_2(z) 
  \end{equation}
  for all $z \in B$.  
  
   We now prove the second claim.  Since $H_2$ extends continuously to $\Gamma$, its CNT boundary values with respect to $\Omega_1$ equal its CNT boundary values with respect to $\Omega_2$, which are equal to those of $\mathbf{J}^q_{1,2} \mathbf{G}_{B,\Omega_1} H$ by (\ref{eq:transmitted_jump_temp_two}).  Of course the CNT boundary values are all continuous extensions.  Thus 
  \begin{equation} \label{eq:transmitted_jump_temp_three}
   \mathbf{G}_{B,\Omega_1} H_2 = \hat{\mathbf{O}}_{1,2} \mathbf{J}^q_{1,2} \mathbf{G}_{B,\Omega_1} H.  
  \end{equation}
  
  To see that $\mathbf{G}_{B,\Omega_1} H_2 \in \mathcal{D}_{\text{harm}}(\Omega_2)$, 
  let $B_1 =f(\mathbb{A}')$ be a collar neighbourhood of $\Gamma$ in $\Omega_1$ where $\mathbb{A}'$ is chosen so that its inner boundary is contained in $\mathbb{A}$. 
  Since $H_2$ is holomorphic on an open neighbourhood of the closure of $B_1$, its restriction to $B_1$ is in $\mathcal{D}(B_1)$.  Since
  $\mathbf{G}_{B,\Omega_1}$ is bounded by Theorem \ref{th:bounce_bounded}, the transmission $\mathbf{G}_{B,\Omega_1}H_2 = \mathbf{G}_{B_1,\Omega_1}H_2$ (where $H_2$ is restricted to $B_1$) is in $\mathcal{D}_{\text{harm}}(\Omega_1)$ as claimed.
  
  Finally, applying $\mathbf{G}_{B,\Omega_1}$ to both sides of (\ref{eq:transmitted_jump_temp_two}), which leaves the first term of the right hand side unchanged, and using (\ref{eq:transmitted_jump_temp_three}) we obtain
  \[  \mathbf{G}_{B,\Omega_1} H = \mathbf{J}^q_{1,1} \mathbf{G}_{B,\Omega_1} H(z) - \hat{\mathbf{O}}_{2,1} \mathbf{J}^q_{1,2} \mathbf{G}(B,\Omega_1)H.       \]
  on $\Omega_1$.  This completes the proof.
 \end{proof}
 
 \begin{theorem} \label{th:transmitted_jump}
  Let $\Gamma$ be a quasicircle in $\sphere$ and let $\Omega_1$ and $\Omega_2$ be the connected components of the complement. For all $h \in \mathcal{D}_{\mathrm{harm}}(\Omega_1)$ 
  \[  h = \mathbf{J}^q_{1,1} h - \mathbf{O}_{2,1} \mathbf{J}^q_{1,2} h.      \]
 \end{theorem}
 \begin{proof}  
  By Lemma \ref{le:pre-transmitted_jump} the claim holds for all $h$ of the form $\mathbf{G}_{B,\Omega_1} H$ for $H \in \mathcal{D}(B)$.  
  By Theorem \ref{th:bounce_of_holomorphic_dense}, $\mathbf{G}_{B,\Omega_1} \mathcal{D}(B)$ is dense in $\mathcal{D}_{\text{harm}}(\Omega_1)$.  Thus the theorem follows from the fact that $\mathbf{J}^q$ is bounded by Corollary \ref{co:J_bounded}.  
 \end{proof}
 
 \begin{remark}
 This can be thought of as the classical jump formula expressed in terms of the transmission.  
 \end{remark}

  Lemma \ref{le:pre-transmitted_jump} generates a large class of functions in 
  the Dirichlet space with continuous transmission. Namely, the bounce of any holomorphic Dirichlet-bounded function in the collar has a continuous transmission. We show this now, as well as the corresponding fact for Bergman space. Recall that the overfare operator for one-forms $\mathbf{O}'$ used below was defined by equation (\ref{eq:Oprime_definition}).  
 \begin{lemma} \label{le:dense_transmission}  
  Let $\Gamma$ be a Jordan curve separating $\sphere$ into components $\Omega_1$ and $\Omega_2$.
  \begin{enumerate}
    \item   For all $\overline{h} \in \mathbf{G}_{B,\Omega_1} \mathcal{D}(B) \cap \overline{\mathcal{D}(\Omega_1)}$, $\mathbf{J}^q_{1,2} \overline{h}$ has a continuous transmission in $\mathcal{D}_{\mathrm{harm}}(\Omega_1)$ given by 
    \[  \hat{\mathbf{O}}_{2,1}  \mathbf{J}^q_{1,2}  \overline{h} = \mathbf{J}^q_{1,1} \overline{h} - \overline{h}.   \]
    \item For all $\overline{\alpha} \in \overline{\partial} [ \mathbf{G}_{B,\Omega_1} \mathcal{D}(B) \cap \overline{\mathcal{D}(\Omega_1)}]$, $\mathbf{T}_{1,2} \overline{h}$ has a continuous transmission in $\mathcal{A}_{\mathrm{harm}}(\Omega_1)$ given by 
    \[ \hat{\mathbf{O}}'_{2,1} \mathbf{T}_{1,2} \overline{\alpha} 
     = \overline{\alpha} + \mathbf{T}_{1,1} \overline{\alpha}.  \]
  \end{enumerate}
 \end{lemma}
 \begin{proof}
  The first claim follows directly from Lemma \ref{le:pre-transmitted_jump}. 
  Now let $\overline{\alpha} = \overline{\partial} \overline{h}$.  Applying now Theorem \ref{th:J_Schiffer_identities} to the right hand side of (1), we see that 
  \[  \partial \hat{\mathbf{O}}_{2,1} \mathbf{T}_{1,2} \overline{\alpha}  = - \mathbf{T}_{1,1} \overline{\alpha}       \]
  and 
  \[  \overline{\partial} \hat{\mathbf{O}}_{2,1} \mathbf{T}_{1,2} 
   \overline{\alpha} = - \overline{\alpha}.       \]
  Applying $\partial$ to the left hand side of (1) and using Theorem \ref{th:J_Schiffer_identities} again proves the claim.
 \end{proof}
 
 We can now prove that $\mathbf{T}_{1,2}$ is one-to-one.
 \begin{theorem} \label{th:T_injective} Let $\Gamma$ be a Jordan curve separating $\sphere$ into components $\Omega_1$ and $\Omega_2$. 
  \begin{enumerate}  
      \item  $\mathbf{T}_{1,2}$ is injective. 
      \item  For any collar neighbourhood $B$ of $\Gamma$ in $\Omega_1$, $T(1,2)$ restricted to $\overline{\partial} [ \mathbf{G}_{B,\Omega_1} \mathcal{D}(B) \cap \overline{\mathcal{D}(\Omega_1)}]$ 
 has left inverse $\overline{\mathbf{P}(\Omega_1)} \hat{\mathbf{O}}'_{2,1}$, where $\overline{\mathbf{P}(\Omega_1)}$ is the projection defined in \eqref{defn:projections in bergman}.  
  \end{enumerate}
 \end{theorem}
 \begin{proof}
  The second claim follows immediately from Lemma \ref{le:dense_transmission} part (2). 
  
  Let $B = f(\mathbb{A})$ be a collar neighbourhood of $\Gamma$ in $\Omega_1$ induced by some biholomorphism $f:\disk^+ \rightarrow \Omega_1$ and annulus $\mathbb{A}$. Now for any $n >0$, by conformal invariance of the bounce operator  (\ref{eq:bounce_conformally_invariant}),  
  \[  \mathbf{G}_{B,\Omega_1} \mathbf{C}_{f^{-1}} w^{-n} = \mathbf{C}_{f^{-1}}
    \mathbf{G}_{\mathbb{A},\disk^+} w^{-n} = \mathbf{C}_{f^{-1}} \bar{w}^n      \]
  so 
  \[  \overline{\partial} \mathbf{C}_{f^{-1}} \mathbb{C}[\bar{z}] \subseteq 
  \overline{\partial} [ \mathbf{G}_{B,\Omega_1} \mathcal{D}(B) \cap \overline{\mathcal{D}(\Omega_1)}].  \]
  Furthermore, $\overline{\partial} \mathcal{C}_{f^{-1}} \mathbb{C}[\bar{z}]$ is dense in  $\overline{A(\Omega_1)}$.  
  
  By the second claim, for any $\overline{\alpha} \in \overline{\partial} \mathbf{C}_{f^{-1}} \mathbb{C}[\bar{z}]$ if $\mathbf{T}_{\Omega_1,\Omega_2} \overline{\alpha} =0$, 
  then $\overline{\alpha} =0$.  
  On the other hand, for $z \in \Omega_2$ fixed, $dw/(w-z)^2 \in \mathcal{A}(\Omega_1)$, and for any $\overline{\alpha} \in \overline{\mathcal{A}(\Omega_1)}$
  \[  \mathbf{T}_{1,2} \overline{\alpha} = \left( \alpha(w), \frac{dw}{(w-z)^2} \right).  \]
  This proves the first claim.
 \end{proof}
 
 This implies that the Cauchy-type operator $\mathbf{J}^q_{1,2}$  is injective. 
 It is convenient to record the two cases $q \in \Omega_{1}, \Omega_2$.  
 {In the following, see (\ref{eq:holomorphic_normalization_projection}) and (\ref{eq:antiholomorphic_normalization_projection}) for the definitions of the projections }. 
 \begin{corollary} \label{co:J_injective}
  Let $\Gamma$ be a Jordan curve $\Gamma$ in $\sphere$.  
  \begin{enumerate}
   \item 
    Fix $q \in \Omega_2$. 
    \begin{enumerate} 
     \item For any $p \in \Omega_1$, $\mathbf{J}^q_{1,2}$ is   injective from $\overline{\mathcal{D}_p(\Omega_1)}$ to $\mathcal{D}_q(\Omega_2)$.  
     \item  For any collar neighbourhood $B$ of $\Gamma$ in $\Omega_1$, on $\overline{\mathcal{D}_p(\Omega_1)} \cap \mathbf{G}_{B,\Omega_1} \mathcal{D}(B)$, the left inverse is given by $- { \overline{\mathbf{P}^{a}_p(\Omega_1)} } \hat{\mathbf{O}}_{2,1}$. 
    \end{enumerate}
   \item Fix $q \in \Omega_1$.  
   \begin{enumerate}
       \item $\mathbf{J}^q_{1,2}$ is injective from $\overline{\mathcal{D}(\Omega_1)}$ 
        to $\mathcal{D}(\Omega_2)$.
        \item  For any collar neighbourhood $B$ of $\Gamma$ in $\Omega_1$, on $\overline{\mathcal{D}(\Omega_1)} \cap \mathbf{G}_{B,\Omega_1} \mathcal{D}(B)$, the left inverse is given by $-{ \overline{\mathbf{P}^{h}_q(\Omega_1)}} \hat{\mathbf{O}}_{2,1}$.
   \end{enumerate}
  \end{enumerate}
 \end{corollary}
 \begin{proof} We first prove the (b) claims.  
   Lemma \ref{le:pre-transmitted_jump} tells us that for any $\overline{h} \in \mathbf{G}_{B,\Omega_1} \overline{\mathcal{D}(\Omega_1)}$ 
  \[   - \hat{\mathbf{O}}_{2,1} \mathbf{J}^q_{1,2} \overline{h} = 
          -   \mathbf{J}^q_{1,1} \overline{h} +  \overline{h}.  \]
  (2) (b) follows by observing that the right hand side is the desired decomposition and applying ${ \overline{\mathbf{P}^{h}_q(\Omega_1)}}$ to both sides. (2) (a) follows similarly once one adds the assumption that $\overline{h}(p)=0$.  
          
  To prove the (a) claims, by Theorem \ref{th:T_injective} part (1) and Theorem \ref{th:J_Schiffer_identities}, if $\mathbf{J}^q_{1,2} \overline{h} =0$ then $\overline{h}$ is a constant $c$.  If $\overline{h} \in \overline{\mathcal{D}_p(\Omega)}$,
  then $c=h(p)=0$ so $h=0$. This proves (1)(a).  If $q \in \Omega_1$, then $c=h(q) =0$. This proves (2)(a). 
 \end{proof}
 
 \begin{theorem}  \label{th:surjective_then_quasicircle} Let $\Gamma$ be a Jordan curve separating $\sphere$ into components $\Omega_1$ and $\Omega_2$. If any of the following three conditions hold, then $\Gamma$ is a quasicircle.  
  \begin{enumerate}
      \item $\mathbf{T}_{1,2}$ is surjective.
      \item  The restriction of $\mathbf{J}^q_{1,2}$ to $\overline{\mathcal{D}_p(\Omega_1)}$ is surjective onto $\mathcal{D}_q(\Omega_2)$ for some $q \in \Omega_2$ and $p \in \Omega_1$. 
      \item The restriction of $\mathbf{J}^q_{1,2}$ to $\overline{\mathcal{D}(\Omega_1)}$ is surjective onto $\mathcal{D}(\Omega_2)$ for $q \in \Omega_1$. 
  \end{enumerate} 
    
 \end{theorem}
 \begin{proof} The first claim follows from the second or third, since by Theorem \ref{th:J_Schiffer_identities}, $\partial \mathbf{J}^q_{1,2} = - \mathbf{T}_{1,2} \overline{\partial} \overline{h}$ for any $\overline{h} \in \overline{\mathcal{D}(\Omega_1)}$.  
  
  Assume that the restriction of $\mathbf{J}^q_{1,2}$ to $\overline{\mathcal{D}(\Omega_1)}$ is an isomorphism onto $\mathcal{D}(\Omega_2)$. Let 
  \[  \mathbf{K}:\mathcal{D}(\Omega_2)\rightarrow \overline{\mathcal{D}(\Omega_1)} \] 
  be its inverse.  
    Choose a collar neighbourhood $B = f(\mathbb{A})$ of $\Gamma$ in $\Omega_1$, 
   where $f:\disk^+ \rightarrow \Omega_1$ is a biholomorphism and $\mathbb{A} = \{z : r<|z|< 1\}$ for some $r \in (0,1)$.  
   As in the proof of Theorem \ref{th:T_injective}, $\mathbf{G}_{B,\Omega_1} \mathbf{C}_{f^{-1}} \mathbb{C}[1/z]$ 
   is dense in $\overline{\mathcal{D}(\Omega_1)}$, since 
   \[   \mathbf{G}_{B,\Omega_1} \mathbf{C}_{f^{-1}} \mathbb{C}[1/z]   
        = \mathbf{C}_{f^{-1}} \mathbf{G}_{\mathbb{A},\disk^+} \mathbb{C}[1/z]
          = \mathbf{C}_{f^{-1}} \mathbb{C}[\overline{z}] \]
   and polynomials are dense in $\mathcal{D}(\disk^+)$.  
   Since $\mathbf{J}^q_{1,2}$ is bounded and surjective, the set 
   \[  \mathscr{L} = \mathbf{J}^q_{1,2} \mathbf{G}_{B,\Omega_1} \mathbf{C}_{f^{-1}} \mathbb{C}[1/z] \]
   is dense in $\overline{\mathcal{D}(\Omega_1)}$.  Furthermore, Lemma \ref{le:pre-transmitted_jump} guarantees that $\mathscr{L} \subset \mathcal{C}(\text{cl} \Omega_2)$, and 
   by Lemma \ref{le:dense_transmission} for every element $\overline{h} \in \mathscr{L}$ we have 
   \[   \hat{\mathbf{O}}_{2,1} \overline{h} =  \hat{\mathbf{O}}_{2,1} \mathbf{J}^q_{1,2} \mathbf{K} \overline{h} = (\mathbf{J}^q_{1,1} \mathbf{K} - \mathbf{K}) \overline{h}.   \]
   We can also conjugate to get transmission of elements $h \in \overline{\mathscr{L}} \subset \mathcal{D}(\Omega_1)$, that is 
   \[   \hat{\mathbf{O}}_{2,1} h  = \overline{(\mathbf{J}^q_{1,1} \mathbf{K} - \mathbf{K}) {h}}.   \]
   Since $\mathbf{J}^q_{1,1} \mathbf{K} - \mathbf{K}$ is bounded, 
   Theorem \ref{th:experimental_transmission} applies, and we can conclude that $\Gamma$ is a quasicircle.  This proves (3).  
   
   If we assume that $q \in \Omega_2$, then the argument above shows that we have bounded transmission on $\mathcal{D}_p(\Omega_2)$ and $\overline{\mathcal{D}_p(\Omega_2)}$. Since constants are transmittable this proves (2).  
 \end{proof}

 \begin{theorem} \label{th:one-sided_isomorphisms} Let $\Gamma$ be a Jordan curve separating $\sphere$ into components $\Omega_1$ and $\Omega_2$.  The following are equivalent.
  \begin{enumerate}
     \item $\Gamma$ is a quasicircle.
      \item $\mathbf{T}_{1,2}$ is a bounded isomorphism.
      \item For $q \in \Omega_2$ and $p \in \Omega_1$, $\mathbf{J}^q_{1,2}$ is a bounded isomorphism from $\overline{\mathcal{D}_p(\Omega_1)}$ into $\mathcal{D}_q(\Omega_2)$.
      \item For $q \in \Omega_1$, $\mathbf{J}^q_{1,2}$ is a bounded isomorphism from $\overline{\mathcal{D}(\Omega_1)}$ into $\mathcal{D}(\Omega_2)$. 
  \end{enumerate}
  In case (2), the inverse is $-\overline{\mathbf{P}(\Omega_1)} \mathbf{O}'_{2,1}$; in case (3) the inverse is $-{ \overline{\mathbf{P}^{a}_p(\Omega_1)}} \mathbf{O}_{2,1}$; and in case (4), the inverse is $-{ \overline{\mathbf{P}^{h}_q(\Omega_1)}} \mathbf{O}_{2,1}$.  
 \end{theorem}
 \begin{proof}
  If (2), (3), or (4) holds, then by Theorem \ref{th:surjective_then_quasicircle} $\Gamma$ is a quasicircle.  
 
  Conversely, assume that $\Gamma$ is a quasicircle. By Theorem  \ref{th:T_injective} and Corollary \ref{co:J_injective}, we have that the maps in (2), (3), and (4) are injective. 
  
  By the inverse mapping theorem it is enough to show that the maps in (2), (3), and (4)
  are surjective.
  Assume that $q \in \Omega_2$.  To see that $\mathbf{J}^q_{1,2}$ is surjective from $\overline{\mathcal{D}_p(\Omega_1)}$ to $\mathcal{D}_q(\Omega_2)$, let $h \in \mathcal{D}(\Omega_2)$. Let $H = - \mathbf{O}_{2,1} h$, where the bounded transmission $\mathbf{O}_{2,1}$ exists by Theorem \ref{th:transmission}. 
  Now $H = H_1 + \overline{H_2}$ where $H_1 \in \mathcal{D}(\Omega_1)$ and $\overline{H_2} \in \overline{\mathcal{D}_p(\Omega_1)}$.  For all $z \in \Omega_2$
  \[  \mathbf{J}^q_{1,2} \overline{H_2} (z)= \mathbf{J}^q_{1,2} H (z) = 
  \mathbf{J}^q_{2,2} h (z) = h(z)-h(q)=h(z) \]
   where the first equality is by part one of Theorem \ref{th:jump_on_holomorphic} with $j=2$, the second equality is by Theorem \ref{th:J_same_both_sides}, and the third equality is by Theorem \ref{th:jump_on_holomorphic} part two with $j=2$.  Thus (3) holds.
   
   A similar argument, after adjustment of the constants and decompositions, proves surjectivity in case (4). Finally, (2) follows from (3) or (4) and the fact that $\partial \mathbf{J}^q_{1,2} \overline{h} 
   = \mathbf{T}_{1,2} \overline{\partial} \overline{h}$. Thus (1) implies (2), (3), and (4), completing the proof.
 \end{proof}
 
 We now prove that the Plemelj-Sokhotski jump decomposition is an isomorphism precisely for quasicircles.  For $q \in \Omega_2$, define
 \begin{align*}
  \mathbf{M}^q(\Omega_1) : \mathcal{D}_{\text{harm}}(\Omega_1) &\rightarrow \mathcal{D}(\Omega_1) \oplus \mathcal{D}_q(\Omega_2) \\
  h & \mapsto \left( \mathbf{J}^q_{1,1} h, \mathbf{J}^q_{1,2} h\right).
 \end{align*} 
 and for $q \in \Omega_1$, define
 \[  \mathbf{M}^q (\Omega_1):\mathcal{D}_{\text{harm}}(\Omega_1) \rightarrow \mathcal{D}_q(\Omega_1) \oplus \mathcal{D}(\Omega_2).  \]  
 Similarly we have the following operator on harmonic Bergman space: 
 \begin{align} \label{eq:jump_isomorphism_definition}
  \mathbf{M}'(\Omega_1) : \mathcal{A}_{\text{harm}}(\Omega_1) & \rightarrow \mathcal{A}(\Omega_1) \oplus \mathcal{A}(\Omega_2)    \\
  \alpha + \overline{\beta} & \mapsto \left( \alpha - \mathbf{T}_{1,1} \overline{\beta}, -\mathbf{T}_{1,2} \overline{\beta} \right)  \nonumber
 \end{align}
 where $\alpha \in A(\Omega_1)$ and $\overline{\beta} \in \overline{A(\Omega_1)}$.  
 
 With this notation, we have the following theorem. 
 \begin{theorem} \label{th:two-sided_isomorphism}
  Let $\Gamma$ be a Jordan curve separating $\sphere$ into components $\Omega_1$ and $\Omega_2$.  The following are equivalent.
  \begin{enumerate}
      \item $\Gamma$ is a quasicircle.
      \item For any $q \in \sphere \backslash \Gamma$, $\mathbf{M}^q$ is an isomorphism.  
      \item $\mathbf{M}'(\Omega_1)$ is an isomorphism.
  \end{enumerate}
  It is enough that (2) holds for a single $q$.  
 \end{theorem}
 \begin{proof}
  We first prove that (1) implies (2).  Assuming that $\Gamma$ is a quasicircle, by Theorem \ref{th:one-sided_isomorphisms} $-\mathbf{T}_{2,2}$ is an isomorphism. 
  Given $\tau = \alpha + \overline{\beta} \in \mathcal{A}_{\text{harm}}(\Omega_1)$, assume that $\mathbf{M} \tau =0$.  Then $-\mathbf{T}_{1,2} \overline{\beta}=0$ so $\overline{\beta} =0$.  Since $\alpha = \alpha -\mathbf{T}_{1,1} \overline{\beta}=0$, we see that $\tau =0$ so $\mathbf{M}$ is injective.  Given any $(\tau,\sigma) \in \mathcal{A}(\Omega_1) \oplus \mathcal{A}(\Omega_2)$, choose $\overline{\beta}$ such that $- \mathbf{T}_{1,2} \overline{\beta} =\sigma$.  Setting $\alpha = \tau + \mathbf{T}_{1,1} \overline{\beta}$ we have $\mathbf{M}(\alpha+ \overline{\beta}) = (\tau,\sigma)$.
  The converse is just the reversal of these arguments. 
 
  A nearly identical argument using Theorem \ref{th:one-sided_isomorphisms}, as well as Theorem \ref{th:jump_on_holomorphic} to deal with the constants, shows that (1) holds if and only if (2) holds.  
 \end{proof}
 
 In the case that $\Gamma$ is a quasicircle, we call $\mathbf{M}^q$ the Plemelj-Sokhotski jump isomosphism.  This establishes that the jump decomposition holds on quasicircles, with data in $\mathcal{H}(\Gamma)$. 
 \begin{theorem}
  Let $\Gamma$ be a quasicircle separating $\sphere$ into components $\Omega_1$ and $\Omega_2$.  For any $u \in \mathcal{H}(\Gamma)$, there exist $h_j \in \Omega_j$ such that the \emph{CNT} boundary values $u_j$ of $h_j$ satisfy 
  \[  u = u_1 - u_2 \]
  except possibly on a null set.  Fixing $q$ in one of the components $\Omega_j$, $h_1$ and $h_2$ are uniquely determined by the normalization $h_j(q)=0$, and are given explicitly by
  \[  (h_1,h_2) =  \mathbf{M}^q \, \mathbf{e}_{\Gamma,\Omega_1} \, u.   \]
 \end{theorem}
 \begin{proof} Fix $q \in \sphere \backslash \Gamma$.
  Given $u \in \mathcal{H}(\Gamma)$, denote $h = \mathbf{e}_{\Gamma,\Omega_1} u$ so that $(h_1,h_2) = \mathbf{M}^q h$. To show that $u_1-u_2 = u$ it suffices to show that 
  \[ h = h_1 - \mathbf{O}_{2,1} h_2. \]
  But this is precisely Theorem \ref{th:transmitted_jump}.  
  
  To see that the decomposition is unique, let $H_j$ be another suitably normalized pair of functions such that $u=\mathbf{b}_{\Gamma,\Omega_1} H_1- \mathbf{b}_{\Gamma,\Omega_2} H_2$.  In that case
  $h = H_1 - \mathbf{O}_{2,1} H_2$ so 
  \begin{equation} \label{eq:Plemelj_temp}
   h_1 - H_1 + \mathbf{O}_{2,1}(H_2 - h_2) =0.
  \end{equation}
  In the case that $q \in \Omega_1$, by Theorem \ref{th:one-sided_isomorphisms} ${ \overline{\mathbf{P}^{h}_q(\Omega_1)}} \mathbf{O}_{1,1}$ is one-to-one on $\mathcal{D}(\Omega_2)$. Since $h_1(q) - H_1(q) =0$, applying this to (\ref{eq:Plemelj_temp}) we obtain that $H_2 - h_2 =0$. Inserting this back into (\ref{eq:Plemelj_temp}) yields $h_1-H_2 =0$.  In the case that $q \in \Omega_2$, $h_2 - H_2 \in \mathcal{D}_q(\Omega_2)$. Fixing $p \in \Omega_1$, by Theorem \ref{th:one-sided_isomorphisms} again, ${ \overline{\mathbf{P}^{a}_p(\Omega_1)}}\mathbf{O}_{2,1}$ is injective on $\mathcal{D}_q(\Omega_2)$.  Applying this to (\ref{eq:Plemelj_temp}) as above yields $H_2-h_2=0$ and $h_1-H_1=0.$
 \end{proof}
\end{subsection}
\end{section}
\begin{section}{Faber and Grunsky operator}  \label{se:Faber_Grunsky}
\begin{subsection}{The Faber operator and Faber series} \label{se:Faber}
 The Faber operator \cite{Suetin} arises in the theory of approximation by Faber series in domains in the plane or sphere, and has its origin in the work of G. Faber \cite{Faber}.  The Faber operator is typically defined as follows. Let $\Gamma$ be a rectifiable Jordan curve separating $\sphere$ into components $\Omega_1$ and $\Omega_2$.  Assume for the moment that $\Omega_1$ is bounded, that is $\infty \in \Omega_2$, and $0 \in \Omega_1$.  Let $f:\disk^+ \rightarrow \Omega_1$ be a conformal map, such that $f(0)=0$.  Let $h$ be a holomorphic function on $\disk^-$ and assume that $h \circ f^{-1}$ extends to an integrable function on $\Gamma$; that is, $h$ has boundary values in some sense (e.g. non-tangential) and $h \circ f^{-1} \in L^1(\Gamma)$. 
 Then the Faber operator is defined by 
 \[   \mathcal{F} h(z)= - \frac{1}{2 \pi i} \int_{\Gamma}  \frac{h \circ f^{-1}(w)}{w-z} \,dw.  \]

 For various choices of the regularity of $\Gamma$ and the space of holomorphic functions on $\disk^+$, this is called the Faber operator.   
 The Faber operator is closely related to the approximation by Faber polynomials of a holomorphic function on $\Omega_2$ in general. The $n$th Faber polynomial corresponding to the domain $\Omega_2$ {is defined by 
 \[  F_n(z) = \mathcal{F}((\cdot)^{-n})(z)=  - \frac{1}{2 \pi i} \int_{\Gamma} \frac{(f^{-1}(w))^{-n}}{w-z}\,dw.     \]}
 
 The Faber operator produces a Faber series as follows.  Let $h(z)$ be a holomorphic function in $\text{cl}( \disk^- )$ which vanishes at the origin, and assume that $\Gamma$ is an analytic Jordan curve, so that we may focus on the heuristic idea.  If $h(z) = h_1 z^{-1} + h_2 z^{-2} + \cdots$, 
 then it is easily verified that the function $H(z) = \mathcal{F} h(z)$ is a holomorphic on the closure of $\Omega_2$ and vanishes at $\infty$.  Furthermore, one has the polynomial series
 \[  H(z) = \mathcal{F} h (z)= \sum_{n=1}^\infty h_n \mathcal{F}((\cdot)^{-n}) 
    = \sum_{n=1}^\infty h_n F_n(z) \]
 called the Faber series of $H$.  
 This series converges uniformly on the closure of $\Omega_2$.  One of the advantages of Faber series over power series, is that sufficiently regular functions converge uniformly on compact subsets of the domain. That is, unlike power series, they are adapted to the geometry of the domain.  
 
 If one refines the analytic setting, as we do below, then one can investigate different kinds of convergence of the series.  Existence and uniqueness of a Faber series correspond to surjectivity and injectivity of the Faber operator respectively.

 
 \begin{remark}
 Since $h$ is defined on $\Omega_2$ and $f^{-1}$ on $\Omega_1$,  the composition $h \circ f^{-1}$ is not necessarily defined anywhere except on $\Gamma$. Thus the boundary behaviour of $h$ and $f^{-1}$ play a central role in the study of the Faber operator.
 \end{remark}
 
 The analytic properties of the Faber operator as they relate to the regularity of the curve and the function space, and approximability in various senses by series of Faber polynomials
 has been extensively studied; see Section \ref{se:Faber_Grunsky} for references. We will now choose specific conditions.  
 
 We define a Faber operator with domain $\mathcal{D}(\disk^-)$ for arbitrary Jordan curves using transmission on the circle.  Since the boundary behaviour of a holomorphic function $h(z)$ on $\disk^+$ is identical in every sense to that of $\mathbf{O}_{\disk^-,\disk^+} h(z) = h(1/\bar{z})$, we will replace the domain $\mathcal{D}(\disk^+)$ of the operator by $\overline{\mathcal{D}(\disk^+)}$. 
\begin{definition}\label{defn:Faberoperator}
For $q \in \Omega_2$, we define a Faber operator by setting
 \begin{equation} \label{eq:Faber_operator_defn}
  \mathbf{I}^q_f = - \mathbf{J}^q_{1,2} \mathbf{C}_{f^{-1}}: \overline{\mathcal{D}_0(\disk^+)} \rightarrow \mathcal{D}_q(\Omega_2).
 \end{equation}
\end{definition} 
 
  It follows immediately from Corollary \ref{co:J_bounded} and conformal invariance of Dirichlet space that $\mathbf{I}_f^q$ is a bounded operator.  The choice $q = \infty$ and $p=f(0)$ corresponds to the case described above. From here on, we refer to (\ref{eq:Faber_operator_defn}) as the Faber operator, and use the new notation to distinguish it from the heuristic discussion above.
  
  \begin{remark}
  It's also possible to retain the constants in $\overline{\mathcal{D}(\disk^+)}$ and $\mathcal{D}(\Omega_2)$ by placing $q \in \Omega_1$.  
 \end{remark}

 Denote the set of polynomials vanishing at $q$ by $\mathbb{C}_q[z]$.   Assume that $f(0)=p \in \sphere$, and let $\Gamma'$ be a fixed simple closed analytic curve in $\Omega_1$ with winding number zero with respect to $p$.   
 By Lemma \ref{le:anchor_lemma}, for any $\bar{h} \in \mathbb{C}_0[\bar{z}]$  we have
 \[ \mathbf{I}^q_f \bar{h} = \mathbf{I}^q_f \mathbf{O}_{\disk^-,\disk^+} u 
  = -\frac{1}{2\pi i} \int_{\Gamma'} {u \circ f^{-1}(w)} \left( \frac{1}{w-z} - \frac{1}{w-q} \right) \,dw   \]
  where $u(z) = \overline{h}(1/\bar{z})$.
 It is easily shown that the output is a polynomial in $(z-p)^{-1}$.  In particular, we define the Faber polynomials as follows.  
 \begin{definition} \label{de:Faber_polynomials}
  Let $\Gamma$ be a Jordan curve separating $\sphere$ into $\Omega_1$ and $\Omega_2$.
  Assume that $q \in \Omega_2$ and let $p=f(0)$. 
  Let $f:\disk^+ \rightarrow \Omega_1$ be a conformal map.
  The $n$th Faber polynomial with respect to $f$ is 
  \[  \Phi_n(z) = \mathbf{I}_f^q (z^{-n}) \in \mathbb{C}_q[1/(z-p)].   \]
 \end{definition}
 If $q= \infty$ and $p=0$ we have $\Phi_n(z) \in \mathbb{C}_*[1/z]$.  It is easily checked that $\Phi_n$ has degree $-n$ in $(z-p)$.
 \begin{remark} \label{re:image_Faber_operator_dense}
  For a bounded domain $D$ bounded by a Jordan curve, polynomials are dense in $\mathcal{A}(D)$  \cite{Markushevich}, so polynomials vanishing at a fixed point $q$ are dense in $\mathcal{D}_q(D)$.  So for $p \in \Omega_1$, setting $M(z)=1/(z-p)$ and $D=M(\Omega_1)$, 
  we see that $\mathbb{C}_q[1/(z-p)]$ is dense in $\mathcal{D}_q(\Omega_1)$.  Thus since $\Phi_n$ has degree $-n$ in $(z-p)$ for each $n$, we see that the image of $\mathbf{I}^q_f$ is dense in $\mathcal{D}_q(\Omega_2)$.  
 \end{remark}
 
 By a Faber series we mean a series of the form 
 \[  \sum_{n=1}^\infty \lambda_n \Phi_n(z), \]   
 whether or not it converges in any sense.
 We also define what we call the {\it sequential Faber operator}: with notation as in 
 Definition \ref{de:Faber_polynomials}, 
 \begin{align*}
   \mathbf{I}_f^{\text{seq}}: \ell^2 & \rightarrow \mathcal{D}(\Omega_2)_q \\
   (\lambda_1,\lambda_2,\ldots ) & \mapsto \sum_{k=1}^\infty \frac{\lambda_k}{\sqrt{k}}
    \Phi_k.  
 \end{align*}

 \begin{theorem} \label{th:Faber_operator_isomorphism}
  Let $\Gamma$ be a Jordan curve separating $\sphere$ into components $\Omega_1$ and $\Omega_2$.  Let $q \in \Omega_2$ and fix a conformal map $f:\disk^+ \rightarrow \Omega_1$. The following are equivalent.
  \begin{enumerate}
      \item $\Gamma$ is a quasicircle.
      \item The Faber operator $\mathbf{I}_f^q$ is a bounded isomorphism. 
      \item The sequential Faber operator is a bounded isomorphism.
      \item Every element of $\mathcal{D}_q(\Omega_2)$ is approximable in norm by a  unique Faber series $\sum_{n=1}^\infty h_n \Phi_n$
       satisfying $(h_1, h_2/\sqrt{2}, h_3/\sqrt{3},\ldots) \in \ell^2$. 
  \end{enumerate}
  If any of conditions \emph{(2)-(4)} hold for a single $q$ and single choice of $f:\disk^+ \rightarrow \Omega_1$, then they hold for every $q \in \Omega_2$ and every choice of $f$.
 \end{theorem}
 \begin{proof}
  The equivalence of (1) and (2) follows immediately from Theorem \ref{th:one-sided_isomorphisms} together with the fact that $\mathbf{C}_{f^{-1}}:\overline{\mathcal{D}_0(\disk^+)} \rightarrow \overline{\mathcal{D}_p(\Omega_1)}$ is a bounded isomorphism, where $p=f(0)$.  The equivalence of (2) and (3) follows from the fact that 
  \begin{equation} \label{eq:sequential_Faber_disk}
   (\lambda_1,\lambda_2,\lambda_3,\ldots) \mapsto \sum_{k=1}^\infty \frac{\lambda_k}{\sqrt{k}} \bar{z}^k  
  \end{equation}
  is a bounded isomorphism from $\ell^2$ to $\overline{\mathcal{D}_0(\disk^+)}$.
  
 To show that (2) and (4) are equivalent, first observe that for any Jordan curve $\mathbf{I}^q_f$ is injective, since $\mathbf{C}_{f^{-1}}:\overline{\mathbf{D}_0(\disk^+)}
  \rightarrow \overline{\mathbf{D}_p(\Omega_1)}$ is an isomorphism and $\mathbf{J}^q_{1,2}$ is injective by Corollary \ref{co:J_injective} part (1).   
  Now assume that (2) holds.  Given any $H(z) \in \mathcal{D}_q(\Omega_2)$ let 
  $H = \mathbf{I}_f^q \overline{h}$. This function $h$ has a power series expression
  \[ \overline{h}(z) = h_1 \bar{z} + h_2 \bar{z}^2 + \cdots,  \] which converges in $\overline{\mathcal{D}_0(\disk^+)}$ to $\bar{h}$.  Since $\mathbf{I}^q_f$ is 
  bounded, applying it to both sides we see that 
  \[  H(z) = \sum_{n=1}^\infty h_n \Phi_n(z) \]
  is convergent in the norm.  Uniqueness follows from injectivity of $\mathbf{I}^q_f$.  
  
  To see that (4) implies (2), observe that (4) implies that the sequential Faber operator is surjective.  Since \eqref{eq:sequential_Faber_disk} is an isomorphism, $\mathbf{I}_f^q$ is surjective, and hence an isomorphism.
 \end{proof}
 The inverse can be given explicitly.    
 \begin{theorem} \label{th:Faber_isomorphism_inverse}
   Let $\Gamma$ be a quasicircle separating $\sphere$ into components $\Omega_1$ and $\Omega_2$.  Let $q \in \Omega_2$, and fix $f:\disk^+ \rightarrow \Omega_1$.  The inverse of $\mathbf{I}_f^q$ is ${ \overline{\mathbf{P}^{a}_0(\disk^+)}} \mathbf{C}_f \mathbf{O}_{2,1}$.   
 \end{theorem}
 \begin{proof}
  Let $p=f(0)$.  Observe that ${ \overline{\mathbf{P}^{a}_0(\disk^+)}} \mathbf{C}_f = \mathbf{C}_f { \overline{\mathbf{P}^{a}_p(\Omega_1)}}$.  Thus by Theorem \ref{th:one-sided_isomorphisms},  
  \begin{align*}
   -{\overline{\mathbf{P}^{a}_0(\disk^+)}} \mathbf{C}_f \mathbf{O}_{\Omega_2,\Omega_1} \mathbf{J}^q_{1,2} \mathbf{C}_{f^{-1}} \overline{h} &  = - \mathbf{C}_f {\overline{\mathbf{P}^{a}_p(\Omega_1)}} \mathbf{O}_{2,1} \mathbf{J}^q_{1,2}  \mathbf{C}_{f^{-1}} \overline{h} \\
   & = \overline{h}   
 \end{align*}
  for all $\overline{h} \in \overline{\mathcal{D}(\disk^+)_0}$.
  So this is a left inverse, which must also be the right inverse by Theorem \ref{th:Faber_operator_isomorphism}.  
 \end{proof}
\end{subsection}
\begin{subsection}{Grunsky inequalities} \label{se:Grunsky}
   The Grunsky inequalities originally stem from H. Grunsky's studies in the context of univalent function theory \cite{Grunsky}. The operators (or matrices) involved in those studies have grown to become a powerful tool in many areas of mathematics.\\
   
  We shall first define the Grunsky operators acting on polynomials, and later extend them by Theorem \ref{th:Grunsky_extension} to Dirichlet spaces. In Theorem \ref{th:Grunsky_expression_quasicircle} we will define the Grunsky operators in the more general setting of quasicircles.
  \begin{definition}
 Given a Jordan curve $\Gamma$ separating $\sphere$ into $\Omega_1$ and $\Omega_2$ as above, let $f:\disk^+ \rightarrow \Omega_1$ be a conformal map with $f(0)=p$ and fix $q \in \Omega_2$. The Grunsky operator on polynomials is defined by
     
  \begin{equation}   \label{eq:Grunsky_definition}
   {\bf{Gr}}_f = { \mathbf{P}^{h}_0(\disk^+)}\, \mathbf{C}_f\, \hat{\mathbf{O}}_{2,1}\, \mathbf{I}_f^q: 
  \mathbb{C}_0[\bar{z}]  \rightarrow \mathcal{D}_0(\disk^+).
   \end{equation}
  \end{definition}
  
  As we saw in the previous section, $\mathbf{I}_f^q$ takes polynomials to polynomials in $\mathbb{C}_q[1/(z-p)]$, which have continuous transmission.   It follows from Lemma \ref{le:pre-transmitted_jump} that the output of $\mathbf{Gr}_f$ on polynomials is in $\mathcal{D}_{0}(\disk^+)$.  
  Since for any $\overline{h}\in \overline{\mathcal{D}_0(\disk^+)}$
  \[     \mathbf{I}^q_f \overline{h} - \mathbf{I}^{q_1}_f \overline{h} =  
     [\mathbf{J}^{q_1}_{1,2} -  \mathbf{J}^{q}_{1,2} ] \mathbf{C}_{f^{-1}} \overline{h}    \]
  is constant, and the transmission and pull-back of constants are also constant, the Grunsky operator is independent of $q$.
  \begin{remark} \label{re:Grunsky_is_classical_Grunsky} By the anchor lemma \ref{le:anchor_lemma}, this agrees with the 
   classical definition of the Grunsky coefficients.  We choose $q = \infty \in \Omega_1$ and $p=0$ to be consistent with the usual conventions, though the reasoning works for arbitrary $q$ and $p$. The classical definition (in fact, one of several) is that the Grunsky coefficients $b_{nk}$ of a univalent map of the disk are given by
   \begin{equation} \label{eq:Faber_compose_map}
     \Phi_n(f(z)) = z^{-n} + \sum_{k=1}^\infty b_{nk} z^k    
   \end{equation} 
   where $\Phi_n$ is the nth Faber polynomial. Recalling that $\mathbf{I}^0_f (z^{-n}) = \Phi_n$, the fact that $\Phi_n(f(z))$ has this form follows from a simple contour integration argument (or from Corollary \ref{co:J_injective}).  By the Anchor Lemma \ref{le:anchor_lemma} applied to $\Phi_n$, together with the fact that 
   \[  { \mathbf{P}^{h}_0(\disk^+)}\, \mathbf{C}_f = \mathbf{C}_f { \overline{\mathbf{P}^{a}_0(\disk^+)}}  \]
   we have 
   \[  {\bf{Gr}}_f (\overline{z}^{n}) =  \sum_{k=1}^\infty b_{nk} z^k.    \]
   Thus we see that the Grunsky coefficients are just the coefficients of the matrix representation of $\mathbf{Gr}_f$.
  \end{remark}
  
  We now extend $\mathbf{Gr}_f$ to the full Dirichlet space.
  \begin{theorem} \label{th:Grunsky_extension} 
   Let $\Gamma$ be a Jordan curve separating $\sphere$ into components $\Omega_1$ and $\Omega_2$.  Fix $q \in \Omega_2$.  Let $f:\disk^+ \rightarrow \Omega_1$ be a conformal map.
   ${\bf{Gr}}_f$ extends to a bounded operator 
   \[ {\bf{Gr}}_f: 
   \overline{\mathcal{D}_0(\disk^+)}\rightarrow \mathcal{D}_0(\disk^+)   \]
   of norm less than or equal to one.
   For all $\overline{h} \in \overline{\mathcal{D}_0(\disk^+)}$ the extended operator 
   satisfies 
   \begin{equation} \label{eq:Grunsky_Faber_identity}
    \| \mathbf{I}^{q}_f \overline{h} \|^2_{\mathcal{D}_q(\Omega_1)} \leq \| \overline{h} \|^2_{\overline{\mathcal{D}_0(\disk^+)}} - \| {\bf{Gr}}_f \, \overline{h} \|^2_{\mathcal{D}_0(\disk^+)}.  
   \end{equation}
   If $\Gamma$ has measure zero, then equality holds.  
  \end{theorem}
  \begin{proof}
   First observe that the Grunsky operator satisfies the following invariance property when restricted to polynomials.  For any M\"obius transformation $M$ 
   \begin{equation} \label{eq:Grunsky_Mobius_invariance_temp}
       {\bf{Gr}}_{M \circ f} = {\bf{Gr}}_{f}.  
   \end{equation}
   This follows from the facts that $\mathbf{C}_M\, \hat{\mathbf{O}}_{M(\Omega_2),M(\Omega_1)} = \hat{\mathbf{O}}_{2,1}\, \mathbf{C}_M$ and by Theorem \ref{th:J_Mobius_invariant} $$\mathbf{C}_M \mathbf{J}^{M(q)}_{M(\Omega_1),M(\Omega_2)} = \mathbf{J}^{q}_{1,2}\, \mathbf{C}_M.$$  
   
   We will first show the inequality (\ref{eq:Grunsky_Faber_identity}) for polynomials.  By the above observation, it is enough to prove it when $\Omega_2$ contains $\infty$, and $p=0 \in \Omega_1$.  We can also assume that $q =\infty$.
   
   For $r \in (0,1)$ let $C_r$ be the positively oriented curve $|z|=r$. 
   For $\bar{h} = h_1 \bar{z} + \cdots + h_m \bar{z}^m \in \mathbb{C}_0[\bar{z}]$ 
   we set
   \[  H(w) = \mathbf{I}^{q}_f \bar{h} = \sum_{n=1}^m h_n \Phi_n.  \] 
   Observe that since $H(w)$ vanishes at $\infty$, 
   \[  \lim_{R \nearrow \infty} \int_{|z|=R}  \overline{H(w)} H'(w) \, dw =0.    \]
   Thus by (\ref{eq:Faber_compose_map})
   we have, 
   using the fact that $\bar{z} = r^2/z$ on $C_r$, 
   \begin{align*}
    \| \mathbf{I}^{q}_{f} \bar{h} \|_{\mathcal{D}_q(\Omega_1)}^2  &  \leq  - \lim_{r \nearrow 1}  \frac{1}{2 \pi i} 
    \int_{f(C_r)} \overline{H}(w) H'(w) \,dw 
    = - \lim_{r \nearrow 1}  \frac{1}{2 \pi i} 
    \int_{C_r} \overline{H(f(z))} (H \circ f)'(z) \,dz \\
    & =  \lim_{r \nearrow 1}  \frac{1}{2 \pi i} \int_{C_r} \left( \sum_{n=1}^m \overline{h_n} \bar{z}^{-n} + \sum_{n=1}^m \sum_{k=1}^\infty \overline{b_{nk}} \bar{z}^k \right) 
    \left( \sum_{n=1}^m  n h_n z^{-n-1} - \sum_{n=1}^m \sum_{k=1}^\infty k {b_{nk}} z^{k-1} \right)\,dz  \\
    & =  \lim_{r \nearrow 1}  \frac{1}{2 \pi i} \int_{C_r} \left( \sum_{n=1}^m \overline{h_n} r^2 {z}^{n} + \sum_{n=1}^m \sum_{k=1}^\infty \overline{b_{nk}} r^2 z^{-k} \right) 
    \left( \sum_{n=1}^m  n h_n z^{-n-1} - \sum_{n=1}^m \sum_{k=1}^\infty k b_{nk} z^{k-1} \right)\,dz  \\
    & = \| h \|^2_{\overline{\mathcal{D}_0(\disk^+)}} - \left\| \sum_{k=1}^\infty 
    \sum_{n=1}^m b_{nk} h_n \right\|^2_{\mathcal{D}_0(\disk^+)}.
   \end{align*}
   If $\Gamma$ has measure zero, equality holds. The theorem now follow from density of $\mathbb{C}_0[\bar{z}]$ in $\overline{\mathcal{D}_0(\disk^+)}$. 
  \end{proof}
  From now on, ${\bf{Gr}}_f$ refers to this extended operator.  
  
  \begin{corollary}
   For any M\"obius transformation $M$, ${\bf{Gr}}_{M \circ f} = {\bf{Gr}}_f$.  
  \end{corollary}
  \begin{proof}
  This follows from (\ref{eq:Grunsky_Mobius_invariance_temp}), since the extended operator must also satisfy this identity. 
  \end{proof}
  \begin{theorem} \label{th:Grunsky_expression_quasicircle}
   If $\Gamma$ is a quasicircle dividing $\sphere$ into $\Omega_+$ and $\Omega_-$, and $f:\disk^+ \rightarrow \Omega_+$ is a biholomorphism, then 
   \[ {\bf{Gr}}_f = {\mathbf{P}^{h}_0(\disk^+)}\, \mathbf{C}_f\, {\mathbf{O}}_{2,1}\, \mathbf{I}_f^q.  \]
  \end{theorem}
  \begin{proof}
   The expression is a bounded extension of (\ref{eq:Grunsky_definition}) by Theorem \ref{th:transmission}.
  \end{proof}
   
  \begin{remark} As is well-known, 
   using the identity in Remark \ref{re:Grunsky_is_classical_Grunsky} one can show that only injectivity of $f$ is necessary in order to define the bounded Grunsky operator on Dirichlet space.  This is usually formulated as an extension to sequences in $\ell^2$.   
  \end{remark}
  
 \begin{theorem}
  Let $\Gamma$ be a Jordan curve in $\sphere$.  The following are equivalent.
  \begin{enumerate} 
      \item $\Gamma$ is a quasicircle. 
      \item The Grunsky operator has norm strictly less than one. 
      \item The Grunsky operator has norm strictly less than one on polynomials.
      \item There is a $\kappa$ such that or all $\bar{h} \in \overline{\mathcal{D}_0(\disk^+)}$, 
         \begin{equation} \label{eq:weak_Grunsky}  \mathrm{Re} \left< \mathbf{O}_{\disk^+,\disk^-} \bar{h} , {\bf{Gr}}_f \, \bar{h} \right>   \leq \kappa \| h\|^2. 
         \end{equation}
      \item The inequality \emph{\eqref{eq:weak_Grunsky}} holds for polynomials.  
  \end{enumerate}
 \end{theorem}
  Before giving the proof, we note the connection with the usual formulation of the Grunsky inequalities.  
   Setting $\bar{h}(z) = {\lambda_1} \overline{z} + \cdots \lambda_n \overline{z}^n$, (3) says that there is some $\kappa<1$ such that for all choices of parameters $\lambda_1,\ldots,\lambda_n \in \mathbb{C}$ 
   \begin{equation} \label{eq:strong_Grunsky}
         \sum_{k=1}^n k \left| \sum_{m=1}^n b_{mk} \lambda_k \right|^2 \leq \kappa \sum_{k=1}^n k |\lambda_k|^2.  
   \end{equation}
  Item (5) says that there is some $\kappa <1$ such that for all such choices of parameters
  \begin{equation} \label{eq:weak_Grunsky_polynomial}
   \left| \sum_{k=1}^n \sum_{m=1}^n b_{mk} \lambda_k \lambda_m \right|   \leq \kappa
  \sum_{k=1}^n k |\lambda_k|^2. 
  \end{equation}
 \begin{proof}
  By Theorem \ref{th:Grunsky_extension} and density of polynomials (2) and (3) are equivalent, and that (4) and (5) are equivalent. By the Cauchy-Schwarz inequality applied to 
  \[  \text{Re} \left< \bar{h},\mathbf{O}_{\disk^-,\disk^+} {\bf{Gr}}_f \bar{h} \right> =
    \text{Re} \left< \mathbf{O}_{\disk^+,\disk^-} \bar{h} , {\bf{Gr}}_f \, \bar{h} \right> \]
  (2) implies (4), using the fact that $\mathbf{O}_{\disk^-,\disk^+}$ is norm-preserving.  
  Thus it is enough to show (1) $\Rightarrow (2)$ and (5) $\Rightarrow (1)$.
 
  (1) $\Rightarrow$ (2). If $\Gamma$ is a quasicircle, then by Theorem \ref{th:Faber_operator_isomorphism} 
  $\mathbf{I}^{q}_{f}$ is an isomorphism, so there is a $c>0$ such that $\| \mathbf{I}^{q}_f \overline{h} \|_{\mathcal{D}_q(\Omega_1)} \geq c \| \overline{h} \|_{\overline{\mathcal{D}_0(\disk^+)}}$ for all $\overline{h}$.   Inserting this in \eqref{eq:Grunsky_Faber_identity}  we see that 
  $\| {\bf{Gr}}_f \overline{h} \|_{\mathcal{D}_0(\disk^+)} \leq \sqrt{1-c^2} \| \overline{h} \|_{\overline{\mathcal{D}_0(\disk^+)}}$.   
  
  (5) $\Rightarrow$ (1).  This is \cite[Theorem 9.12]{Pommerenkebook} applied to (\ref{eq:weak_Grunsky}), applied to $g(z)=1/f(1/z)$.  The different convention for
  the mapping function does not alter the result; see (\ref{eq:one_over_z_transform_Grunsky}) ahead.
 \end{proof}

 \begin{remark} A simple functional analytic proof that any of (2)-(5) implies (1) can be given, if we assume in addition that $\Gamma$ is a measure zero Jordan curve. Assuming that $\| {\bf{Gr}}_f \| \leq k <1$ say, and applying the equality case of (\ref{eq:Grunsky_Faber_identity}) we obtain $\| \mathbf{I}^{q}_{f} \overline{h}  \|_{\mathcal{D}_q(\Omega_1)} \geq \sqrt{1-k^2} \| \overline{h} \|_{\overline{\mathcal{D}_0(\disk^+)}}$.  So the image of $\mathbf{I}^q_f$ is closed, and by Remark \ref{re:image_Faber_operator_dense} it is $\mathcal{D}_q(\Omega_2)$.  Hence by the open mapping theorem $\mathbf{I}_f^q$ is a bounded isomorphism, and therefore Theorem \ref{th:Faber_operator_isomorphism} yields that $\Gamma$ is a quasicircle.  
 \end{remark} 
 
 \begin{remark}  In order to define the Faber polynomials and the Grunsky coefficients $b_{nk}$ in (\ref{eq:Grunsky_Faber_identity}), it is only required that $f$ is defined in a neighbourhood of $0$ and has non-vanishing derivative there.  It is classical that (\ref{eq:strong_Grunsky}) and (\ref{eq:weak_Grunsky}) each hold for $\kappa=1$ if and only if $f$ extends to a one-to-one holomorphic function on $\disk^+$ \cite{Duren_book}. 
  Equation (\ref{eq:strong_Grunsky}) (with $\kappa = 1$) is usually called the strong Grunsky inequalities, while (\ref{eq:weak_Grunsky}) is usually called the weak Grunsky inequalities \cite{Duren_book}.  The computation in the proof of Theorem \ref{th:Grunsky_extension}
  is the usual proof of the strong Grunsky inequalities.
 \end{remark} 
 \begin{remark}
  The proof of Theorem \ref{th:Grunsky_extension} is easily modified to show that for any one-to-one holomorphic function $f$ on $\disk$, the 
  Grunsky operator (expressed as a function of the parameters $\alpha_k = \lambda_k/\sqrt{k}$) extends to a bounded operator on $\ell^2$, see e.g. \cite{Duren_book,Pommerenkebook}.    
 \end{remark}
  
  Finally, we include an integral expression for the Grunsky operator, due to Bergman and Schiffer \cite{BergmanSchiffer}.  It is most conveniently expressed in terms of the Bergman space of one-forms.   
  
  If $\Omega$ is simply connected, then 
  \[  d: \mathcal{D}(\Omega) \rightarrow A(\Omega) \] 
  is norm-preserving, and in fact if $\Omega$ is simply connected it becomes an isometry when restricted to $\mathcal{D}(\Omega)_q$ for any $q \in \Omega$.    We then define 
  \[    \hat{\bf{Gr}}_f: \overline{\mathcal{A}(\disk^+)} \rightarrow \mathcal{A}(\disk^+)   \]
  by 
  \[ \partial \, {\bf{Gr}}_f = \hat{{\bf{Gr}}}_f \, \overline{\partial}.  \]
  With this definition, we have
  \begin{theorem} \label{th:Grunsky_integral_formula} For any Jordan curve $\Gamma$ and conformal map $f:\disk^+ \rightarrow \Omega_1$, 
  \begin{align*}
   \hat{\bf{Gr}}_f \overline{\alpha} & = f^* \mathbf{T}_{1,1} (f^{-1})^* \overline{\alpha} \\
   & = \iint_{\disk^+}  \frac{1}{2 \pi i} \left( \frac{f'(w) f'(z)}{(f(w)-f(z))^2} - \frac{1}{(w-z)^2} \right) \overline{\alpha(w)} \wedge dw \cdot dz.   
  \end{align*}
  \end{theorem}
  \begin{proof} Set $p = f(0)$.  
   Assume for the moment that $\Gamma$ is a quasicircle.  Then, fixing some $q \in \Omega_2$,  
   \begin{align*}
    \hat{{\bf{Gr}}}_f \overline{\alpha} & = - \partial {\mathbf{P}^{h}_0(\disk^+)} \mathbf{C}_f \mathbf{O}_{2,1} \mathbf{J}^q_{1,2} \mathbf{C}_{f^{-1}} {\overline{\partial}}^{-1} \overline{\alpha} \\
    & = - \partial \mathbf{C}_f {\mathbf{ P}^{h}_p(\Omega_1)} \mathbf{O}_{2,1} \mathbf{J}^q_{1,2} {\overline{\partial}}^{-1} (f^{-1})^* \overline{\alpha} \\
    & = - f^* \mathbf{P}(\Omega_1) \partial \mathbf{O}_{2,1} \mathbf{J}^q_{1,2} {\overline{\partial}}^{-1} (f^{-1})^* \overline{\alpha}.
   \end{align*}
   Here, it is understood that $\overline{\partial}^{-1}$ is a choice of inverse of $\overline{\partial}: \mathcal{D}_0(\disk^+) \rightarrow \mathcal{A}(\disk^+).$ 
   Now applying Theorems \ref{th:transmitted_jump} and \ref{th:jump_on_holomorphic} we see that
   \begin{align*}
       \hat{{\bf{Gr}}}_f \overline{\alpha} & =  - f^* \mathbf{P}(\Omega_1) \partial ( \mathbf{J}^q_{1,1} - \mathrm{Id}) {\overline{\partial}}^{-1} (f^{-1})^* \overline{\alpha} \\ 
       & = f^* \mathbf{P}(\Omega_1) \mathbf{T}_{1,1}(f^{-1})^* \overline{\alpha} 
   \end{align*}
   which proves the first claim.  The integral formula is obtained by  
   substituting (\ref{eq:Schiffer_desingularized}) into the right hand side and changing variables, using the fact that 
   \[   L_{\Omega_1}(w,z) = \frac{1}{2 \pi i} \frac{(f^{-1})'(w) (f^{-1})'(z)}{(f^{-1}(w)-f^{-1}(z))^2} dw \cdot dz.  \]
   
   If $\Gamma$ is a Jordan curve but not a quasicircle, we apply the same argument to polynomials $\overline{h} \in \mathbb{C}_0[\bar{z}]$, using Lemma \ref{le:dense_transmission} in place of Theorem \ref{th:transmitted_jump}.  The result is then extended to $\overline{\mathcal{D}_0(\disk^+)}$ using Theorem \ref{th:Grunsky_extension}.    
  \end{proof}
  Bergman and Schiffer directly defined an operator using this integral formula, and observed that it recovers the Grunsky operator when applied to polynomials, see   
  \cite[eq (9.7), (9.9)]{BergmanSchiffer} and nearby text (in fact, their formulation is for germs of maps in arbitrarily simply connected domain, and the above is a special case).  In particular, their integral formulation agrees with the unique operator extension of the Grunsky matrix to $\ell^2$, if we identify sequences with elements of $\mathcal{A}(\disk^+)$ as in Theorem \ref{th:Faber_operator_isomorphism}.
  
  \begin{remark}  
   It can be shown that the integral formula in Theorem \ref{th:Grunsky_integral_formula} is a bounded operator for arbitrary one-to-one $f:\disk \rightarrow \mathbb{C}$, and this agrees with the extension of the Grunsky operator of Theorem \ref{th:Grunsky_extension}.  Although Bergman and Schiffer \cite{BergmanSchiffer} assume that the boundaries are analytic Jordan curves, they were certainly aware of this fact.   
  \end{remark}
\end{subsection}
\end{section}

\begin{section}{Notes and literature} \label{se:notes}
\begin{subsection}{Notes on the Introduction}
  We begin with some further remarks on attributions and proofs.  
  Below, (n) refers to the claim that the statements (1) and (n) in the introduction are equivalent, unless the direction of the implication is specified.\\

 As mentioned in the introduction, the characterizations (5) and (6) of quasicircles are due to Shen \cite{ShenFaber}, where the implications (1) $\Rightarrow$ (5,6) were obtained earlier by \c{C}avu\c{s} \cite{Cavus}.   Note that \c{C}avu\c{s} and Shen work with the conformal map on the outside of the disk, but this is only a difference of convention. Similarly, they phrase their results in terms of Bergman spaces, but our formulation here is the same, after
 application of the isometry $h \mapsto h'$ between Dirichlet and Bergman spaces.
 
 Characterization (4) is due to the authors \cite{Schippers_Staubach_CAOT}.  The Faber operator typically involves an integral over the Jordan curve, which must therefore be rectifiable in order to make sense. Rectifiability is added as an assumption in Wei, Wang and Hu \cite{WeiWangHu}, who showed (6) for rectifiable Jordan curves.  By using the limiting integral, we were able to remove the assumption of rectifiability.  A key result is the equality of limiting integrals from either side, stated in this paper as Theorem \ref{th:J_same_both_sides}, which was originally proven in \cite{Schippers_Staubach_JMAA}.  
 
 It should be noted that although we have established the equivalences, the result of Shen is at face value stronger in the direction (5) $\Rightarrow$ (1) in comparison to (4) $\Rightarrow$ (1), and weaker than (1) $\Rightarrow$ (4) in the direction (1) $\Rightarrow$ (5).  Also, Shen's result is stronger in that it does not require assuming that the domain is a Jordan curve.  It is not immediately clear what the meaning of transmission would be when the complement of $f(\disk^+)$ is more general than the closure of a Jordan domain.  We did not pursue this issue, since we could not shed new light on his results. The issue seems to be of interest, in light of the fact that Faber polynomials have meaning for degenerate domains (a classic example being the Chebyshev polynomials for an interval), among other things.\\

 The characterization (7) is due to Napalkov and Yulmukhametov \cite{Nap_Yulm}. It was proven independently by the authors \cite{Schippers_Staubach_CAOT}, using our characterization (4).  Unfortunately none of the aforementioned authors, including us, were aware of the results of Napalkov and Yulmukhametov.  Both our proof and that of Wei, Wang and Hu use the result or approach of Shen.  In some sense, transmission provides a bridge between the result (7) of Napalkov and Yulmukhametov \cite{Nap_Yulm} and (5) of \c{C}avu\c{s} \cite{Cavus} and Shen \cite{ShenFaber}, by making it possible to replace the sequential Faber operator with the Faber operator for non-rectifiable curves.  Theorem \ref{th:J_same_both_sides} is an essential ingredient of our approach (ultimately relying on both the bounded transmission theorem and the anchor lemma).   
 
 The result (8) $\Rightarrow$ (1), that the strict Grunsky inequality implies that $\Gamma$ is a quasicircle, is used in every proof that (4), (5), (6), and (7) implies (1) given in the literature so far.  We give an alternate proof in this paper which uses transmission only.  Our present proofs in the reverse direction (that (1) implies (2) through (7)) also differ from previous ones given by the authors.  We first applied this alternate approach in \cite{Schippers_Staubach_Plemelj} in the case of Jordan curves on Riemann surfaces.  
 
\end{subsection}
\begin{subsection}{Notes on Section \ref{se:boundary_values}}

The fact that quasiconformal maps and quasisymmetries preserve compact sets of capacity zero is well-known   \cite{Ahlfors-Beurling}. N. Arcozzi and R. Rochberg \cite{Arcozziroch} gave a combinatorial proof that if $\phi:\mathbb{S}^1 \to \mathbb{S}^1 $ is a quasisymmetry and $I$ is
a closed subset of $\mathbb{S}^1$, then there is a constant $K > 0$ depending only on $\phi$ such that $\frac{1}{K}c(I) \leq c(\phi(I))\leq K c(I)$. This of course implies Corollary \ref{co:quasisymmetry_nullpreserve}. Also, E. Villamor \cite[Theorem 3]{Villamor} showed that if $g:\disk^- \rightarrow \mathbb{C}$ is a one-to-one holomorphic $\kappa$-quasiconformally extendible map satisfying $g(z)=z + \cdots$ near $\infty$, then there is a $\kappa$ depending only on the quasiconformal constant such that for any closed $I \subset \mathbb{S}^1$, $c(I)^{1+\kappa} \leq c(g(A)) \leq c(I)^{1-\kappa}$.   
This implies Remark \ref{re:quasicircle_capacity_zero} and hence Theorem \ref{th:null_both_sides_quasicircle}.

It is natural to ask whether the converse of Theorem \ref{th:null_both_sides_quasicircle} holds.  That is, let $\Gamma$ be a Jordan curve separating $\sphere$ into components $\Omega_1$ and $\Omega_2$. Assume that any set $I \subset \Gamma$ which is null with respect to $\Omega_1$ is null with respect to $\Omega_1$ is also null with respect to $\Omega_2$, and vice versa.  Must $\Gamma$ be a quasicircle?

In \cite{Schippers_Staubach_JMAA,Schippers_Staubach_CAOT} we used limits along hyperbolic geodesics (equivalently, orthogonal curves to level curves of Green's function) in place of CNT limits, following H. Osborn \cite{Osborn}.  If the CNT limit exists, then the radial limit exists.  Besides being a stronger property, the CNT limits we later defined \cite{Schippers_Staubach_transmission,Schippers_Staubach_Plemelj} seem to be more convenient.  We use the name ``Osborn space'' for the set of boundary values of Dirichlet bounded harmonic functions, in order to draw attention to the paper \cite{Osborn}.
\end{subsection}
\begin{subsection}{Notes on Section \ref{se:transmission}}

 It is obvious that some notion of null set is necessary to formulate transmission. As we saw, the fact that quasisymmetries or quasiconformal maps preserve capacity zero sets was central to defining a notion of null sets which allowed the formulation of transmission on quasicircles.
 
 On the other hand, quasisymmetries do not preserve sets of harmonic measure zero \cite{Ahlfors-Beurling}. In particular, even for a quasicircle $\Gamma$, sets of harmonic measure zero with respect to one component of $\sphere \backslash \Gamma$ need not be of harmonic measure with respect to the other component. This can be seen immediately by a proof by contradiction using the conformal welding theorem. Thus harmonic measure is inadequate for our purposes. 

 We have shown that a bounded transmission exists for quasicircles separating a compact Riemann surface into two components in  \cite{Schippers_Staubach_transmission}. The results of that paper develop a foundation for applying quasisymmetric sewing techniques to boundary value problems for general Riemann surfaces, and ultimately to a ``scattering theory'' viewpoint of Teichm\"uller theory \cite{Schippers_Staubach_monograph}.     
 
 Our proof of Theorem \ref{th:J_same_both_sides} given in \cite{Schippers_Staubach_JMAA} contains a gap, which is not hard to fill in a couple of ways.  Here it is filled by the proof of the anchor lemma, which we stated and proved for the first time in \cite{Schippers_Staubach_Plemelj}.  
\end{subsection}
\begin{subsection}{Notes on Section \ref{se:Schiffer_Cauchy}} \label{se:Schiffer_Cauchy_notes}
  
  The operators $\mathbf{T}_{j,k}$ were first defined by Schiffer \cite{Schiffer_first}.  Schiffer investigated these operators extensively with others; see e.g. Bergman and Schiffer \cite{BergmanSchiffer}, Schiffer and Spencer \cite{Schiffer_Spencer}, Schiffer \cite{Courant_Schiffer}.  The connection to the jump problem was explicit from the beginning; see e.g. Bergman and Schiffer \cite{BergmanSchiffer}, and especially the survey \cite{Schiffer_expository} which focusses on the real jump problem and its relation to boundary layer potentials.  The connection to the complex jump theorem which we give here is more direct.   The paper of   
  Royden \cite{Royden} connects the Schiffer kernel functions to the jump problem on Riemann surface. His results are phrased somewhat differently in terms of topological conditions for extensions of holomorphic and harmonic extensions on domains; indeed, the Plemelj-Sokhotski jump formula is not mentioned explicitly.  However it can be derived as a special case of his results, but with more restrictive analytic assumptions on the function on the curve; namely, that it extend holomorphically to a neighbourhood of the curve.  Our paper \cite{Schippers_Staubach_Plemelj} considers jump decompositions and Schiffer kernels on Riemann surface, in the setting of Dirichlet spaces and quasicircles, extending some of the results given here to higher genus. \\    
  
  The terminology surrounding the Schiffer operators is a bit confused.  As a Calder\'on-Zygmund singular integral operator acting on functions in the plane, the Schiffer operator is bounded on $L^2$ (more generally on $L^p$, $1<p<\infty$). The integral operator on general functions in $L^2(\mathbb{C})$ is called the Beurling transform.  It is also sometimes called the Hilbert transform, a term used more widely (in the harmonic analysis and integral equations context) for a principal value integral along the real line (the explicit formula is \eqref{eq:classical_Hilbert_transform} ahead, if one chooses there $\Gamma = \mathbb{R}$). Napalkov and Yulmukhametov use the term Hilbert transform to refer specifically to $\mathbf{T}_{1,2}$.  Of course these integral operators are all closely related.  We reserve 
  the term ``Schiffer operator'' for the restriction of the singular integral operator to anti-holomorphic functions on a subset of $\sphere$.\\
  
  This ``nesting'' - that the kernel function is derived from the Green's function of a larger domain than the domain of integration - is a central feature of the Schiffer operators, which he explored at length in \cite{Courant_Schiffer}.   
   On general domains and Riemann surfaces, there is also a related operator derived from integrating against the Bergman kernel obtained from a larger domain.  (This does not appear in the present paper, because the Bergman kernel of the sphere is zero). Adding to the terminological confusion described above, some authors refer to the Bergman kernel on Riemann surfaces as the Schiffer kernel, while at the same time there is indeed a distinct Schiffer kernel (related to the so-called fundamental bi-differential).  On the double, certain identities relate the Schiffer and Bergman kernels \cite{Schiffer_Spencer}.\\

   The authors proved that a jump decomposition holds for the special case of Weil-Petersson class (WP-class) quasidisks in \cite{RSS_WP_quasidisks}. 
  This class arises naturally in geometric function theory, Teichm\"uller theory and the theory of Schramm-Loewner evolution. As was demonstrated in \cite{RSS_WP_quasidisks} the rectifiability and Ahlfors-regularity of the WP-class quasicircles enables one to prove a Plemelj-Sokhotzki-type jump decomposition. However the proof of \cite[Theorem 2.8]{RSS_WP_quasidisks} (the chord-arc property of the WP-quasicircles) has a gap arising from our misinterpretation of the definition of a quasicircle given in the paper \cite{falconermarsh} by K. Falconer and D. Marsh, which in fact only applies to weak quasicircles, see e.g. J-F. Lafont, B. Schmidt, W. van Limbeek \cite{Lafont}. Therefore the theorem in \cite{RSS_WP_quasidisks} is not true as stated. 
  Our claim that WP-class quasicircles are chord-arc (and hence Ahlfors regular) was proven by C. Bishop \cite{Bishop}.
  As with the case of quasicircles, there are an extraordinary number of characterizations of WP-class quasicircles. Bishop  \cite{Bishop} has listed over twenty, many of which are new (answering among many others a question raised by Takhtajan-Teo \cite{TakhtajanTeo_memoirs}). His paper also contains other far-reaching generalisations of the concept (to higher dimensions).\\
  
  In the case of WP-class quasidisks, not only is the curve rectifiable, but the boundary values of Dirichlet bounded harmonic functions on such domains lie in a certain Besov space. So the contour integral could be defined directly.
  On the other hand, quasicircles are not in general rectifiable, which creates a hindrance to formulation of the jump decomposition in the setting of this paper. 
 B. Kats studied Riemann-Hilbert problems on non-rectifiable curves, see e.g. \cite{Kats_2017} for the case of H\"older continuous boundary values, and the survey article \cite{Kats_survey} and references therein.  The CNT boundary values of elements of $\mathcal{D}_{\text{harm}}(\Omega_k)$ are not continuous, so this technology was not available.  The jump decomposition was shown to hold for a range of Besov spaces of boundary values by the authors, for $d$-regular quasidisks \cite{SchippersStaubachBesov}, {which are not necessarily rectifiable}.  This result did not include the case of boundary values of the Dirichlet space, so it was also not available for use here.\\

   An interesting open question arises in association with the jump formula on quasidisks.  The classical Plemelj-Sokhotski jump formula can be expressed using a principle value integral on the curve.  That is, if $u$ is a smooth function and $\Gamma$ is a smooth Jordan curve in $\mathbb{C}$ define, for $z_0 \in \Gamma$, 
   \begin{equation} \label{eq:classical_Hilbert_transform}
     \mathscr{H} u(z_0) = \text{P.V.} \frac{1}{2 \pi i} \int_{\Gamma} 
   \frac{u(\zeta)}{\zeta-z_0} \,d\zeta.    
   \end{equation}
   Of course, one could weaken the analytic assumptions.  We have \cite{Bell_book}
   \begin{equation} \label{eq:improved_jump}
    \lim_{z \rightarrow z_0^\pm} \frac{1}{2 \pi i} \int_{\Gamma} \frac{u(\zeta)}{\zeta - z} \,d\zeta = \pm \frac{1}{2}  u(z_0) + \mathscr{H}u(z_0) 
   \end{equation}
   where $\lim_{z \rightarrow z_0^\pm}$ respectively denotes the limits taken in the bounded and unbounded components $\Omega_+$ and $\Omega_-$ of the complement of $\Gamma$.  This of course implies the jump formula.  The question is: can a meaningful principal value integral $\mathscr{H} u$ be defined when $\Gamma$ is a quasicircle and $u \in \mathcal{H}(\Gamma)$, and a corresponding formula \eqref{eq:improved_jump} found?  This would have many applications to the study of integral kernels.

   As mentioned above, Napalkov and Yulmukhametov were the first to recognize and prove that the Schiffer operator $\mathbf{T}_{1,2}$ is an isomorphism for quasicircles; as far as we know this was not known to Schiffer even for stronger assumptions on the curve.  We have generalized this and the jump isomorphism to various settings (taking into account topological obstacles); namely 
   to compact Riemann surfaces separated by a quasicircle \cite{Schippers_Staubach_Plemelj}; and with M. Shirazi, to compact Riemann surfaces with $n$ quasicircles enclosing simply connected domains in \cite{RSSS_detline}. The converse result to that of Napalkov and Yulmukhametov, that if $\mathbf{T}_{1,2}$ is an isomorphism then $\Gamma$ is a quasicircle, only exists in genus zero. It is an open question whether a suitably formulated converse holds in genus $g >0$, though it seems plausible once the topological differences are taken into account.
   
   The operator $\overline{\mathbf{P}^{a}(\Omega_1)} \mathbf{O}_{2,1}$ appears in conformal field theory (usually with stronger analytic assumptions).  Theorem \ref{th:Faber_isomorphism_inverse} generalizes to higher genus, that is this operator is  inverse to a kind of Faber operator.  This fact can be exploited to give an explicit description of the determinant line bundle of this operator; see \cite{RSSS_detline} for the case of genus $g$ surfaces with one boundary curve.  
   The general case of genus $g$ with $n$ boundary curves is work in progress with D. Radnell.     
\end{subsection}
\begin{subsection}{Notes on Section \ref{se:Faber_Grunsky}}
  There is a vast literature on the Faber operator, Faber series, and their approximation properties. As mentioned in the main text, they are defined with various regularities.  See the books of J. M. Anderson \cite{Anderson} and P. K. Suetin \cite{Suetin} (note that the 1998 English translation of the 1984 original has an extensively updated bibliography). Some more recent papers are Wei, Wang and Hu \cite{WeiWangHu}, D. Gaier \cite{Gaier}, Y.E. Y{\i}ld{\i}r{\i}r and R. \c{C}etinta\c{s} \cite{YildirirCetintas}. 
  
  The Grunsky operator has been explored by many authors, for example A. Baranov and H. Hedenmalm \cite{Baranov_Hedenmalm} and G. Jones \cite{GavinJones}.  L.A. Takhtajan and L.-P Teo showed that it provides an analogue of the classical period mapping of compact surfaces for the universal Teichm\"uller space \cite{TakhtajanTeo_memoirs}.  See also V. L. Vasyunin and N. K. Nikol'ski\u{\i} for an exposition of its appearance in de Branges' work on complementary spaces \cite{VasyuninNikolskii}. There are many interesting results relating the analytic properties of the Grunsky matrix $\mathbf{Gr}_f$ to the geometric or analytic properties of the conformal map $f$ and/or its image $f(\disk^+)$; see for example Jones \cite{GavinJones}, Shen \cite{ShenGrunsky}, or Takhtajan and Teo \cite{TakhtajanTeo_memoirs}. 
  
  The treatment as an integral operator goes at least as far back as Bergman and Schiffer's classic paper \cite{BergmanSchiffer}, as described in the explanation following Theorem \ref{th:Grunsky_integral_formula}. \\
  
  The Grunsky inequalities have been generalized in many ways.  J. A. Hummel \cite{Hummel} generalized the inequalities to pairs of non-overlapping maps.  The authors have extended this to arbitrary numbers of non-overlapping maps in genus zero \cite{RSS_Dirichlet_multiply}; Grunsky inequalities were proven for the case of $n$ non-overlapping maps into a compact surface of genus $g$ by M. Shirazi \cite{ShiraziThesis,ShiraziGrunsky}. This has applications to Teichm\"uller theory and is related to generalizations of the classical period mapping to the infinite-dimensional Teichm\"uller space of bordered surfaces of arbitrary genus and number of boundary curves  \cite{RSS_period_genus_zero}, \cite{Schippers_Staubach_monograph}.  The thesis of Shirazi \cite{ShiraziThesis} also contains a historical survey of the Faber and Grunsky operator.\\

 The Faber polynomials, Faber operator, Grunsky operator, and Grunsky inequalities are formulated with an array of differing conventions and approaches.  The existence of certain identities also complicates matters. We will not attempt to untangle the conventions here, but rather content ourselves with a few remarks.
  For an early overview with a tidy description of the algebraic identities and relation involved, see E. Jabotinsky \cite{Jabotinsky}. See also the historical outline in Shirazi \cite{ShiraziThesis}.
 
  Usually normalizations are imposed on the function classes, especially the derivative at the origin or at $\infty$.  These normalizations obscure the M\"obius invariance of various objects, such as the Grunsky operator and Cauchy integral operator, and furthermore limit the applicability of the stated theorems unnecessarily.  So we have removed them as much as possible throughout the paper.   
 
 The Grunsky inequalities and Faber polynomials are often formulated for conformal maps of the form $g:\disk^- \rightarrow \Omega_1$, of the form $g(z) = z + b_0 + b_1/z + \cdots$, where $\Omega_1$ contains the point at $\infty$.  Choosing $q = 0$, the Faber polynomials are then defined by $\Phi_n = \mathbf{I}_g^0(z^n)$, and the Grunsky coefficients by 
 \[  \Phi_n(g(z)) = z^{-n} + \sum_{n=-\infty}^{-1} b_{nk} z^k.   \]
 The convention that $g$ takes $\disk^-$ onto a domain containing $\infty$ seems to provide an advantage in some proofs \cite{Duren_book,Pommerenkebook}, in that the area of the complement of $g(\disk^-)$ is finite; this appears to be the motivation for the choice.  However, the advantage is illusory: the important fact is that the functions to which the Grunsky operator is applied have finite Dirichlet energy.  
 The following identity shows that either choice is as good as the other. Setting $f(z) = 1/g(1/z)$, then it is easily checked that for $n>0$ and $m>0$,
 \begin{equation} \label{eq:one_over_z_transform_Grunsky}
   \sqrt{n m } \, b_{-n,-m}(g) = \sqrt{n m} \, b_{nm}(f).    
 \end{equation}

  Another approach to the definition of the Grunsky coefficients is through generating functions, e.g.  
 \[  \log{\frac{g(z)-g(w)}{z-w}} - \log{g'(\infty)}  =  \sum_{n=-\infty,m=-\infty}^{-1}
      b_{n,m} z^{n} w^{m}     \]
  where $g'(\infty) = \lim_{z \rightarrow \infty} g(z)/z$.
  for suitably chosen branches of logarithm.  One can recognize immediately the 
  relation to the integral kernel in Theorem \ref{th:Grunsky_integral_formula}. 
  This also visibly demonstrates (\ref{eq:one_over_z_transform_Grunsky}).  
  The Faber polynomials can also be defined using generating functions related to the integral kernel of the Faber operator, see e.g. \cite{Duren_book,Jabotinsky,Suetin}. 
\end{subsection}
\end{section}

\end{document}